\documentclass[11pt]{article}
\usepackage[tbtags]{amsmath}
\usepackage{epstopdf}
\usepackage{amssymb}
\usepackage{amsthm}
\usepackage[misc]{ifsym}
\usepackage{graphicx}
\usepackage{tikz}
\usepackage{color}
\usepackage{float}
\usepackage{graphicx}
\usepackage{subfigure}
\usetikzlibrary{shapes,arrows}
\usepackage{cases}
\numberwithin{equation}{section}
\newtheorem{definition}{Definition}[section]
\newtheorem{theorem}{Theorem}[section]
\newtheorem{lemma}{Lemma}[section]

\normalsize
\usepackage{hyperref,amsmath,amssymb,amscd}
\title{\bf A Linear Quadratic Partially Observed Stackelberg Stochastic Differential Game with Applications
\thanks{This work is supported by National Key R\&D Program of China (Grant No. 2018YFB1305400), National Natural Science Foundations of China (Grant Nos. 11971266, 11831010, 11571205), and Shandong Provincial Natural Science Foundations (Grant Nos. ZR2020ZD24, ZR2019ZD42).}}
\author{\normalsize Yueyang Zheng\thanks{\it School of Mathematics, Shandong University, Jinan 250100, P.R.China, E-mail: zhengyueyang0106@163.com} , Jingtao Shi\thanks{\it Corresponding author. School of Mathematics, Shandong University, Jinan 250100, P.R.China, E-mail: shijingtao@sdu.edu.cn}}
\date{}

\newtheorem{Remark}{Remark}[section]
\begin{document}
\maketitle

\noindent{\bf Abstract:}\quad This paper is concerned with a linear-quadratic partially observed Stackelberg stochastic differential game with correlated state and observation noises, where the diffusion coefficient does not contain the control variable and the control set is not necessarily convex. Both the leader and the follower have their own observation equations, and the information filtration available to the leader is contained in that to the follower. By spike variational, state decomposition and backward separation techniques, necessary and sufficient conditions of the Stackelberg equilibrium points are derived. In the follower's problem, the state estimation feedback of optimal control can be represented by a forward-backward stochastic differential filtering equation and some Riccati equation. In the leader's problem, via the innovation process, the state estimation feedback of optimal control is represented by a stochastic differential filtering equation, a semi-martingale process and three high-dimensional Riccati equations. At the same time, the uniqueness and existence of solutions to adjoint equations can be guaranteed by a new combined idea, and a kind of fully coupled forward-backward stochastic differential equations with filtering is studied as a by-product. Then we give explicit expressions of Stackelberg equilibrium points in a special case. As a practical application, an inspiring dynamic advertising problem with asymmetric information is studied, and the effectiveness and reasonability of the theoretical result is illustrated by numerical simulations. Moreover, the relationship between optimal control, state estimate and some practical parameters is analyzed in detail.

\vspace{1mm}


\vspace{1mm}


\section{Introduction}

Among various dynamic games, the Stackelberg game was introduced firstly by Stackelberg \cite{S52}, which is a kind of hierarchical noncooperative game (Ba\c{s}ar and Olsder \cite{BO98}). The Stackelberg game is also know as the leader-follower game, whose economic background can be derived from some markets where certain companies have advantages of domination over others. In the Stackelberg game, there are usually two players with asymmetric roles, one follower and one leader. The leader claims his/her strategy in advance, and the follower, considering the given leader's strategy, makes an instantaneous reaction by optimizing his/her cost functional. Then, by taking the rational response of the follower into account, the leader would like to seek an optimal strategy to optimize his/her cost functional. Therefore, one player must make a decision after the other player's decision is made. Since its meaningful structure and background, the Stackelberg game has received substantial interest. Especially, the leader-follower's feature can be applied in many aspects, such as the newsvendor/wholesaler problem (\O ksendal et al. \cite{OSU13}), the optimal reinsurance problem (Chen and Shen \cite{CS18}), the operations management and marketing channel problem (Li and Sethi \cite{LS17}) and the principal-agent/optimal contract problem (Cvitani\'{c} and Zhang \cite{CZ13}).

In the past decades, there have been a great deal of significant works on this issue. Let us mention a few. First, Castanon and Athans \cite{CA76} studied an LQ stochastic dynamic Stackelberg strategies in the early, and obtained a feedback Stackelberg solution for two-person nonzero sum game. An indefinite LQ leader-follower stochastic differential game with random coefficients and control-dependent diffusion was studied by Yong \cite{Yong02}. {\it Forward-backward stochastic differential equations} (FBSDEs for short) and Riccati equations are applied to obtain the state feedback representation of the open-loop Stackelberg equilibrium points. Ba\c{s}ar et al. \cite{BBS10} introduced the notion of mixed leadership in nonzero-sum differential games where one player could act as both leader and follower, depending on the control variable. Bensoussan et al. \cite{BCS15} introduced several solution concepts in terms of the players' information sets, and derived the maximum principle for the leader's global Stackelberg solution under the adapted closed-loop memoryless information structure. Mukaidani and Xu \cite{MX15} investigated the Stackelberg games for linear stochastic systems governed by It\^{o} differential equations with multiple followers. The Stackelberg strategies, obtained by using sets of cross-coupled algebraic nonlinear matrix equations, are developed under two different settings: the followers act either cooperatively to attain Pareto optimality or non-cooperatively to arrive at a Nash equilibrium. Xu and Zhang \cite{XZ16} and Xu et al. \cite{XSZ18} investigated the LQ Stackelberg differential games with time delay. Li and Yu \cite{LY18} proved the solvability of a kind of coupled FBSDEs with a multilevel self-similar domination-monotonicity structure, then it is used to characterize the unique equilibrium of an LQ generalized Stackelberg stochastic differential game with hierarchy in a closed form. Moon and Ba\c{s}ar \cite{MB18} considered the LQ mean-field Stackelberg differential game with the adapted open-loop information structure of the leader where the $N$ followers play a Nash game. Lin et al. \cite{LJZ19} studied the open-loop LQ Stackelberg game of the mean-field stochastic systems and obtained the feedback representation form involving the new state and its mean. Du and Wu \cite{DW19} investigated an LQ Stackelberg game of mean-field {\it backward stochastic differential equations} (BSDEs for short). Zheng and Shi \cite{ZS19} researched the Stackelberg game of BSDEs with complete information. Bensoussan et al. \cite{BCCSSY19} characterized the feedback equilibrium of a general infinite horizon Stackelberg-Nash differential game where the roles of the players are mixed. Recently, Huang et al. \cite{HSW20} studied a controlled linear-quadratic-Gaussian large population system combining major leader, minor leaders and minor followers. The {\it Stackelberg-Nash-Cournot} (SNC for short) approximate equilibrium is derived from the combination of a major-minor mean-field game and a leader-follower Stackelberg game, and the feedback form of the SNC approximate equilibrium strategy is obtained via some coupled Riccati equations. Feng et al. \cite{FHH20} considered the LQ Stackelberg game of BSDEs with constraints.

Noticing all the Stackelberg game considered in the above are set in the complete information background. In other word, both the leader and the follower can observe the state process of the stochastic system directly, which, however, is impractical in reality. Actually, both of them can only acquire partial information because of the information-delay, market competition, private information and limitation policy by the government. Therefore, it is necessary to study the Stackelberg game with partial and asymmetric information. There exist a number of works on this issue, very recently. Shi et al. \cite{SWX16,SWX17} studied the Stackelberg stochastic differential game and introduced a new explanation for the asymmetric information feature, that the information available to the follower is based on the some sub-$\sigma$-algebra of that available to the leader. Shi et al. \cite{SWX20} investigated the LQ Stackelberg stochastic differential game with overlapping information, where the follower's and the leader's information have some joint part, while they have no inclusion relations. Wang et al. \cite{WWZ20} discussed an asymmetric information mean-field type LQ Stackelberg stochastic differential game with one leader and two followers. Zheng and Shi \cite{ZS20CCC} investigated a LQ Stackelberg game of BSDEs with partial information, where the information of the leader is a sub-$\sigma$-algebra of that of the follower.

However, the asymmetric information feature in the above literatures only shows that all the partial information available to the leader and the follower are generated by the noisy processes - standard Brownian motions. In reality, the player's available information usually comes from the observation process. For example, in the financial market, the investor can only observe the security prices. Therefore, the investors would like to make their portfolio selection based on the information filtration generated by the security prices process (Xiong and Zhou \cite{XZ07}). The essential works associated with the partially observed problem is the filtering estimate problem which consists of two processes: the signal process and the observation process. Then the aim of the filtering theory is to estimate the signal process based on the information generated by the observation processes (Liptser and Shiryayev \cite{LS77}, Bensoussan \cite{Ben92}, Xiong \cite{Xiong08} and Wang and Wu \cite{WW08}). Partially observed optimal control and differential game problems have been studied by many researchers, such as Li and Tang \cite{LT95}, Tang \cite{Tang98}, Wang and Wu \cite{WW09}, Huang et al. \cite{HWX09}, Wu \cite{Wu10}, Shi and Wu \cite{SW10}, Wang et al. \cite{WWX13,WWX15,WWX18}, Wu and Zhuang \cite{WZ18}, Xiong et al. \cite{XZZ19}. Up to now, few studies have focussed on the partially observed Stackelberg differential game. Li et al. \cite{LFCM19} derived the necessary condition of this leader-follower stochastic differential game under a partially observed information where the controlled system equation and the observation equation are looked as the equality constraints. Zheng and Shi \cite{ZS20} studied a Stackelberg stochastic differential game with asymmetric noisy observation where the follower cannot observe the state process directly, while the leader can completely observe it. Very recently, Li et al. \cite{LDFCM21} investigated an LQ Gaussian Stackelberg stochastic differential game under asymmetric information. Two decision makers implement control strategies using different information patterns: The follower uses its observation information to design its strategy, whereas the leader implements its strategy using global information.

Our work distinguishes itself from the existing literature in the following aspects:

(1) To our best acknowledge, there is no previous research on the LQ partially observed Stackelberg stochastic differential game using the state decomposition and backward separation techniques, where the leader and the follower have their own observation equations different from each other. For technical need, in this paper we assume that the information filtration available to the leader is included in that available to the follower, compared to Li et al. \cite{LFCM19} where both the leader and the follower's controls are adapted to the filtration generated by the same observation process and Li et al. \cite{LDFCM21}, Zheng and Shi \cite{ZS20} where the information available to the leader is more than that available to the follower. However, it is not a trivial setting on account of the technical difficulty appeared in the follower's part when utilizing the state decomposition technique. First, different from Xiong et al. \cite{XZZ19}, in which the Girsanov's theorem and measure transformation method are used, in this paper we apply state decomposition and backward separation techniques from Wang et al. \cite{WWX15}, although in \cite{XZZ19} two observation equations are also considered. Secondly, different from Wu and Zhuang \cite{WZ18} where the controls of two players are adapted to the filtration generated by the same observation process, therefore, the similar result can be obtained easily as Lemma 2.1 in \cite{WWX15}. Different from \cite{WWX15} in which they consider the control problem, in this paper the Stackelberg game is researched and it fails to proceed in the follower's part without our new assumption (Lemma 3.1), since the two control variables appears in the hierarchical structure. Moreover, many practical examples, such as the market channel and operations management problems, can be applied to verify the assumption's reasonability.

(2) For characterizing the adaptation and relationship of these filtrations well appearing in this paper, in Remark 2.2, we make a detailed analysis for the partially observed problem, especially for the essence of observation equations. More importantly, we compare it with the partial information problem which should be another case, and show their intrinsic distinctions.

(3) We need to solve a partially observed optimal control problem of FBSDFEs in the leader's problem, after dealing with the partially observed optimal control problem of the follower. However, different from all the existing literatures of Stackelberg stochastic differential game with partial information, in which the state equation's driven noises of the leader only contains the independent $\mathcal{F}_t$-standard Brownian motions, in this paper when observing the equations \eqref{leaderstate} and \eqref{together} (or \eqref{together2}), the appearance of a new $\mathcal{F}_t^{Y_1}$-Brownian motion $\widetilde{W}$ (the innovation process) makes our leader's problem become an entirely new problem which is not investigated before. Therefore, we have to consider a kind of state equations with the fact that $\widetilde{W}$ is not independent of the standard Brownian motions $W$ or $\bar{W}$. Moreover, the foremost issue in this step, is to guarantee the uniqueness and existence of solution to this kind of equations. We solve this difficult problem by a new combined idea from Yong \cite{Yong02} and Lim and Zhou \cite{LZ01}.

(4) It is worth noting that we are supposed to give the existence and uniqueness of solution to \eqref{together3} with three different Brownian motions first, then decouple it for obtaining the state estimation feedback of optimal control. Motivated by the usual step and the classical method in Peng and Wu \cite{PW99}, we study a kind of fully coupled FBSDEs with filtering which can be regarded as a special case of \eqref{together3}, and get the existence and uniqueness of their solutions by constructing a contraction mapping and using fixed-point theorem. However, we fail to prove the unique solvability of the FBSDFE \eqref{together3}, because it is not only a fully coupled one, but also contains two kinds of filtering form, $(\hat{X},\hat{Y})$ and $(\check{X},\check{Y})$. Therefore, we can not solve it directly by existing methods. Fortunately, we find an idea to obtain its unique solvability as follows: we first obtain the existence of solution to \eqref{together3} by decoupling technique as usual, and further guarantee that the solution is indeed unique inspired by the method in \cite{LZ01}. More importantly, the existence and uniqueness of adjoint equations \eqref{lambda} and \eqref{K} in the leader's problem should also be given. In fact, we also have difficult in dealing with it because of the existence of $W,\bar{W},\widetilde{W}$. But we can observe that the adjoint equations are related to \eqref{together3}, therefore, we can solve \eqref{together3} first and consequently ensure the unique solvability of \eqref{lambda} and \eqref{K}.

(5) The semi-martingale $\Phi$ of \eqref{Phi} is introduced in the relation \eqref{Y3X} between $Y$ and $X,\hat{X},\check{X}$. However, different from \cite{LFCM19}, the $\Phi$ is not only used to ``absorb" the $\widetilde{W}$ in the filtering equations \eqref{hatX}, but also to explain the reasonable existence of $\widetilde{W}$ appeared in the leader's state and adjoint equations as a driven noise which is control-independent. Moreover, for obtaining the state estimation feedback of the Stackelberg equilibrium point, we make a detailed analysis on the filtrations generated by $Y_1, Y_2$ and $\widetilde{W}$, respectively, and discuss their relationships in order to facilitate the filtering problem.

(6) For giving an explicit form for Stackelberg equilibrium points, we reduce the original problem to a special one under some conditions and obtain some explicit expressions. Then an application to the dynamic advertising problem with asymmetric information is studied, which extends an example in Jorgensen et al. \cite{JTZ03} to the stochastic case with partially observable information. This inspiring motivation is an excellent explanation for the condition that both the leader and the follower can not observe the complete information, and the information available to the leader is less than that available to the follower. Its numerical simulation is also given to verify the effectiveness of the proposed theoretical results, moreover, we also analyze the relationship between the optimal advertising effort, brand image and its estimate and some practical parameters.

The rest of this article is organized as follows. In Section 2, we formulate an LQ partially observed Stackelberg stochastic differential game model and adopt the state decomposition technique to deal with it. In Section 3, the necessary and sufficient conditions of the partially observed optimal control problem of the follower are given by using the backward separation technique, and the state estimation feedback representation is obtained via an FBSDFE and one Riccati equation in Section 3.1. In Section 3.2, we prove the maximum principle of the partially observed optimal control problem of FBSDFEs of the leader, and give the sufficient condition of the optimal control. For obtaining the state estimate feedback, we first obtain the unique solvability of leader's state in a form of FBSDFE by a new combined idea. Then the optimal control can be represented by an SDFE, three high-dimensional Riccati equations and a semi-martingale. In Section 4, we focus on a special partially observed case and give the explicit expressions of Stackelberg equilibrium points. In Section 5, a practical application to the dynamic advertising problem is studied, and its numerical simulation is given to illustrate the effectiveness of the proposed theoretical results. Moreover, the relationship between the Stackelberg equilibrium points, state estimate and some practical parameters is also discussed. Finally, some concluding remarks are list in Section 6.

\section{Problem Formulation}

Let $T>0$ be fixed. Consider a complete filtered probability space $(\Omega,\mathcal{F},(\mathcal{F}_t^{W,\bar{W}})_{0\leq t \leq T},\mathbb{P})$ and two one-dimensional independent standard Brownian motions $W(\cdot)$ and $\bar{W}(\cdot)$ defined in $\mathbb{R}^2$ with $W(0)=\bar{W}(0)=0$, which generates the filtration $\mathcal{F}_{t}^{W,\bar{W}}:=\sigma\{W(r),\bar{W}(r): 0\leq r\leq t\}$ augmented by all the $\mathbb{P}$-null sets in $\mathcal{F}$. We set $\mathcal{F}_t\equiv\mathcal{F}_{t}^{W,\bar{W}}$ if there is no ambiguity. In this paper, $L_{\mathcal{F}_T}^2(\Omega,\mathbb{R})$ denotes the set of $\mathbb{R}$-valued, $\mathcal{F}_T$-measurable, square-integrable random variables, $L^2_\mathcal{F}(0,T;\mathbb{R})$ denotes the set of $\mathbb{R}$-valued, $\mathcal{F}_t$-adapted, square integrable processes on $[0,T]$, and $L^\infty(0,T;\mathbb{R})$ denotes the set of $\mathbb{R}$-valued, bounded functions on $[0,T]$. For any stochastic process $\Delta(\cdot): [0,T]\times\Omega\rightarrow\mathbb{R}$, $\mathcal{F}_t^{\Delta}$ denote the filtration generated by stochastic process $\Delta(\cdot)$ in $[0,t]$, that is, $\mathcal{F}_t^{\Delta}:=\sigma\{\Delta(r):0\leq r \leq t\}$.

\vspace{1mm}
We consider the following controlled linear {\it stochastic differential equation} (SDE for short):
\begin{equation}\label{LQsde}
\left\{
\begin{aligned}
dx^{v_1,v_2}(t)&=\big[A(t)x^{v_1,v_2}(t)+B_1(t)v_1(t)+B_2(t)v_2(t)+\alpha(t)\big]dt\\
               &\quad+c(t)dW(t)+\bar{c}(t)d\bar{W}(t),\ t\in[0,T],\\
 x^{v_1,v_2}(0)&=x_0,
\end{aligned}
\right.
\end{equation}
and the observation processes of the follower and the leader satisfy
\begin{equation}\label{observation f}
\left\{
\begin{aligned}
dY_1^{v_1,v_2}(t)&=\big[f_1(t)x^{v_1,v_2}(t)+g(t)\big]dt+dW(t),\ t\in[0,T],\\
 Y_1^{v_1,v_2}(0)&=0,
\end{aligned}
\right.
\end{equation}
\begin{equation}\label{observation l}
\left\{
\begin{aligned}
dY_2^{v_2}(t)&=\big[f_2(t)+v_2(t)\big]dt+dW(t),\ t\in[0,T],\\
 Y_2^{v_2}(0)&=0,
\end{aligned}
\right.
\end{equation}
respectively. Here $(v_1(\cdot),v_2(\cdot))\in\mathbb{R}^2$ are the control process pair of the follower and the leader, $(x^{v_1,v_2}(\cdot),Y_1^{v_1,v_2}(\cdot))\in\mathbb{R}^2$ are the state process and corresponding follower's observation process related to $(v_1(\cdot),v_2(\cdot))$, and $Y_2^{v_2}(\cdot)\in\mathbb{R}$ is the leader's observation process independent of the control $v_1(\cdot)$.

\begin{Remark}
Observing the leader's observation equation \eqref{observation l}, we can see that the drift term, which is free of $x^{v_1,v_2}$, is chosen by the leader, then it is easy to check that $\mathcal{F}_t^{Y_2^{v_2}}=\mathcal{F}_t^W$. In other words, the observation filtration is influenced by the control process $v_2(\cdot)$ in form, but the circular dependence between $Y_2^{v_2}(\cdot)$ and $v_2(\cdot)$ disappear because of the fact $v_2(\cdot)\in\mathcal{F}_t^{Y_2^{v_2}}=\mathcal{F}_t^W$. Moreover, due to the state decomposition in the following, \eqref{Y20} implies $\mathcal{F}_t^{Y_2^0}=\mathcal{F}_t^W$. Therefore, we get an important result that $\mathcal{F}_t^{Y_2^{v_2}}=\mathcal{F}_t^{Y_2^0}=\mathcal{F}_t^W$, which is one of the main conclusion used many times in the following discussion.
\end{Remark}

\begin{Remark}
We note that the leader's observation equation of $Y_2^{v_2}(\cdot)$ is not general enough. Indeed, we consider that the drift term of (\ref{observation l}) is not related to the state variable $x^{v_1,v_2}(\cdot)$, but is related to the leader's control variable $v_2(\cdot)$ directly. It is necessary to explain why we set the form and illustrate the difficulties with the general form. If we consider the general form of $Y_2(\cdot)$ as the following
\begin{equation}\label{another}
\left\{
\begin{aligned}
dY_2^{v_1,v_2}(t)&=\big[f_2(t)x^{v_1,v_2}(t)+g_2(t)\big]dt+dW(t),\ t\in[0,T],\\
  Y_2^{v_1,v_2}(0)&=0,
\end{aligned}
\right.
\end{equation}
and we separate $Y_2^{v_1,v_2}(\cdot)$ into two parts:
\begin{equation}\label{Y220}
\left\{
\begin{aligned}
dY_2^0(t)&=f_2(t)x^0(t)dt+dW(t),\ t\in[0,T],\\
 Y_2^0(0)&=0,
\end{aligned}
\right.
\end{equation}
and
\begin{equation}\label{Y221}
\left\{
\begin{aligned}
dY_2^1(t)&=\big[f_2(t)x^1(t)+g_2(t)\big]dt,\ t\in[0,T],\\
 Y_2^1(0)&=0.
\end{aligned}
\right.
\end{equation}
Observing the above equations \eqref{Y220} and \eqref{Y221}, and combining the separated equations \eqref{separationxY1Y2}, \eqref{x0}, \eqref{Y0}, \eqref{x1}, \eqref{Y1} in the following, two challenging problems arise naturally in Section 3.2. The first one is the appearance of another new $\mathcal{F}_t^{Y_2^{\bar{v}_1,v_2}}$-adapted Brownian motion $\widetilde{W}_2$ which is also called the innovation process when we calculate the filtering equations with respect to $\mathcal{F}_t^{Y_2^{\bar{v}_1,v_2}}$ in the leader's problem, which make the problem more complicated taking the $W,\bar{W},\widetilde{W}$ and $\widetilde{W}_2$ into account. The second one is that, under the assumption in {\bf Lemma 3.1}, we fail to guarantee that $\mathcal{F}_t^{Y_2^{\bar{v}_1,v_2}}=\mathcal{F}_t^{Y_2^0}$ holds which make it impossible for us to overcome the circular dependency between $v_2$ and $\mathcal{F}_t^{Y_2^{\bar{v}_1,v_2}}$.

From what has been discussed above, we consider the form in \eqref{observation l} which is convenient to study the leader's problem although the circulation still exists.
\end{Remark}

Now some assumptions are needed.

{\bf(A1)}
$A(\cdot),B_1(\cdot),B_2(\cdot),\alpha(\cdot),c(\cdot),\bar{c}(\cdot),f_1(\cdot),f_2(\cdot)$ and $g(\cdot)$ are uniformly bounded, deterministic functions. $x_0$ is a constant.

Under $\bf(A1)$, for any given $v_1(\cdot)\in L^2_\mathcal{F}(0,T;\mathbb{R})$ and $v_2(\cdot)\in L^2_\mathcal{F}(0,T;\mathbb{R})$, the SDE \eqref{LQsde} admits a unique solution $x^{v_1,v_2}(\cdot)\in L^2_\mathcal{F}(0,T;\mathbb{R})$.

For the LQ partially observed Stackelberg differential game problem, we firstly consider the optimal control problem of the follower with the given control $v_2(\cdot)$ of the leader. Therefore, it is natural to determine the control $v_1(\cdot)$ by the observation $Y_1^{v_1,v_2}(\cdot)$, which, however, depends on the controls $(v_1(\cdot),v_2(\cdot))$ via the state $x^{v_1,v_2}(\cdot)$ in \eqref{observation f}. In other words, $v_1(\cdot)$ has an effect on $Y_1^{v_1,v_2}(\cdot)$. The circular dependence between $Y_1^{v_1,v_2}$ and $v_1(\cdot)$ results in an intrinsic difficulty to study such a class of problems. In the following, we will apply the state decomposition technique \cite{Ben92,BV75} and backward separation technique \cite{WWX15} to overcome this obstacle.

We first separate the state $x^{v_1,v_2}(\cdot)$, the observations $Y_1^{v_1,v_2}(\cdot)$ and $Y_2^{v_2}(\cdot)$ into
\begin{equation}\label{separationxY1Y2}
\left\{
\begin{aligned}
&x^{v_1,v_2}(\cdot)=x^0(\cdot)+x^1(\cdot);\\
&Y_1^{v_1,v_2}(\cdot)=Y_1^0(\cdot)+Y_1^1(\cdot);\\
&Y_2^{v_2}(\cdot)=Y_2^0(\cdot)+Y_2^1(\cdot),
\end{aligned}
\right.
\end{equation}
where $x^0(\cdot), Y_1^0(\cdot)$ and $Y_2^0(\cdot)$ are independent of the control process $v_i(\cdot)(i=1,2)$. Let $v_1(\cdot)$ be adapted to $\sigma\{Y_1^{v_1,v_2}(s),0\leq s\leq t\}$ and $\sigma\{Y_1^{0}(s),0\leq s\leq t\}$.

We define the processes $x^0(\cdot), Y_1^0(\cdot)$ and $Y_2^0(\cdot)$ by
\begin{equation}\label{x0}
\left\{
\begin{aligned}
dx^0(t)&=A(t)x^0(t)dt+c(t)dW(t)+\bar{c}(t)d\bar{W}(t),\ t\in[0,T],\\
  x^0(0)&=x_0,
\end{aligned}
\right.
\end{equation}
\begin{equation}\label{Y0}
\left\{
\begin{aligned}
dY_1^0(t)&=f_1(t)x^0(t)dt+dW(t),\ t\in[0,T],\\
  Y_1^0(0)&=0,
\end{aligned}
\right.
\end{equation}
and
\begin{equation}\label{Y20}
\left\{
\begin{aligned}
dY_2^0(t)&=dW(t),\ t\in[0,T],\\
  Y_2^0(0)&=0.
\end{aligned}
\right.
\end{equation}
Then define the processes $x^1(\cdot), Y_1^1(\cdot)$ and $Y_2^1(\cdot)$ by
\begin{equation}\label{x1}
\left\{
\begin{aligned}
dx^1(t)&=\big[A(t)x^1(t)+B_1(t)v_1(t)+B_2(t)v_2(t)+\alpha(t)\big]dt,\ t\in[0,T],\\
  x^1(0)&=0,
\end{aligned}
\right.
\end{equation}
\begin{equation}\label{Y1}
\left\{
\begin{aligned}
dY_1^1(t)&=\big[f_1(t)x^1(t)+g(t)\big]dt,\ t\in[0,T],\\
  Y_1^1(0)&=0,
\end{aligned}
\right.
\end{equation}
and
\begin{equation}\label{Y21}
\left\{
\begin{aligned}
dY_2^1(t)&=\big[f_2(t)+v_2(t)\big]dt,\ t\in[0,T],\\
  Y_2^1(0)&=0,
\end{aligned}
\right.
\end{equation}
It is not hard to see that, under the assumption $\bf(A1)$, \eqref{x0}, \eqref{Y0}, \eqref{Y20}, \eqref{x1}, \eqref{Y1} and \eqref{Y21} admit unique solutions, respectively.

First, by Remark 2.1, the admissible control set of the leader is defined by
\begin{equation}\label{l0}
\begin{aligned}
\mathcal{U}_{ad}^{l}[0,T]=&\Big\{v_2(\cdot)\big|v_2(\cdot)\mbox{ is }\mathbb{R}\mbox{-valued and }\mathcal{F}_t^{Y_2^{v_2}}\mbox{-adapted process}\\
&\mbox{\ \ such that }\mathbb{E}\Big[\sup_{0\leq t\leq T}(v_2(t))^2\Big]<\infty\Big\}.
\end{aligned}
\end{equation}

\begin{definition}
A control process $v_2(\cdot)$ is called admissible for the leader, if $v_2(\cdot)\in\mathcal{U}_{ad}^{l}[0,T]$.
\end{definition}
In order to overcome the circular dependence between the control and the controlled filtration, a definition of the admissible control of the follower is given based on the separation processes.

Next, we consider the follower's problem. Set
\begin{equation}\label{f0}
\begin{aligned}
\mathcal{U}_{ad}^{f0}[0,T]:=&\Big\{v_1(\cdot)\big|v_1(\cdot)\mbox{ is }\mathbb{R}\mbox{-valued and }\mathcal{F}_t^{Y_1^0}\mbox{-adapted process}\\
&\mbox{\ \ such that }\mathbb{E}\Big[\sup_{0\leq t\leq T}(v_1(t))^2\Big]<\infty\Big\}.
\end{aligned}
\end{equation}
\begin{definition}
A control process $v_1(\cdot)$ is called admissible for the follower, for the given $v_2(\cdot)\in\mathcal{U}_{ad}^{l}[0,T]$, if $v_1(\cdot)\in\mathcal{U}_{ad}^{f0}[0,T]$ is $\mathcal{F}_t^{Y_1^{v_1,v_2}}$-adapted. And the set of these admissible controls is denoted by $\mathcal{U}_{ad}^f[0,T]$.
\end{definition}
In the game, the follower would like to choose an $\mathcal{F}_t^{Y_1^{\bar{v}_1,v_2}}$-adapted optimal control $\bar{v}_1(\cdot)\in\mathcal{U}_{ad}^f[0,T]$ such that his/her cost functional
\begin{equation}\label{LQ cost functional}
\begin{aligned}
J_1(v_1(\cdot),v_2(\cdot))&=\frac{1}{2}\mathbb{E}\bigg\{\int_0^T\big[L(t)(x^{v_1,v_2}(t))^2+R(t)v_1^2(t)+2l(t)x^{v_1,v_2}(t)+2r(t)v_1(t)\big]dt\\
&\qquad\qquad+M(x^{v_1,v_2}(T))^2+2mx^{v_1,v_2}(T)\bigg\}
\end{aligned}
\end{equation}
is minimized.

In the leader's game, knowing that the follower would take $\bar{v}_1(\cdot)\in\mathcal{U}_{ad}^f[0,T]$, the leader would like to choose an $\mathcal{F}_t^{Y_2^{\bar{v}_2}}$-adapted optimal control $\bar{v}_2(\cdot)\in\mathcal{U}_{ad}^l[0,T]$ such that his/her cost functional
\begin{equation}\label{LQ cost functional2}
\begin{aligned}
J_2(v_2(\cdot))\equiv J_2(\bar{v}_1(\cdot),v_2(\cdot))&=\frac{1}{2}\mathbb{E}\bigg\{\int_0^T\big[\bar{L}(t)(x^{v_2}(t))^2+\bar{R}(t)v_2^2(t)+2\bar{l}(t)x^{v_2}(t)\\
&\qquad\qquad+2\bar{r}(t)v_2(t)\big]dt+\bar{M}(x^{v_2}(T))^2+2\bar{m}x^{v_2}(T)\bigg\}
\end{aligned}
\end{equation}
is minimized. The following assumption is introduced necessarily.

{\bf(A2)} $L(\cdot)\geq0,\bar{L}(\cdot)\geq0,R(\cdot)\geq0,\bar{R}(\cdot)\geq0,l(\cdot),\bar{l}(\cdot),r(\cdot)$ and $\bar{r}(\cdot)$ are uniformly bounded, deterministic functions. $M\geq0,\bar{M}\geq0$ and $m,\bar{m}$ are constants.

Similar to the idea in Section 2.3.2 of \cite{Ben92}, we give a more restrictive definition of admissibility, which is helpful to overcome the circulation and to apply the filtering formula in the following procedure. More precisely, we consider the admissible control $v_1(\cdot)$ is adapted both to the filtration $\mathcal{F}_t^{Y_1^{v_1,v_2}}$ and to the uncontrolled filtration $\mathcal{F}_t^{Y_1^0}$.

\begin{Remark}
When it comes to the partially observed problem, there is often a misleading for the filtrations. To be specific, we obtain an uncontrolled state equation and observation equation after using state decomposition technique, that is, the equations \eqref{x0} and \eqref{Y0}. By Lemma 3.1 in the next section, we have $\mathcal{F}_t^{Y_1^{v_1,v_2}}=\mathcal{F}_t^{Y_1^0}$. In this case, the misunderstanding is that one will think that $\mathcal{F}_t^{Y_1^0}=\mathcal{F}_t^{W,x^0}=\mathcal{F}_t^{W,\bar{W}}$, which means that the controller seems to know the complete information, and the observation equation, to some degree, loses its original meaning. The misleading result is mainly due to the misunderstanding for observation equations. Indeed, we have to know, in fact, that the state equation, $x(\cdot)=x^1(\cdot)+x^0(\cdot)$, can not be observed directly although we will think of $\mathcal{F}_t^{x^0}=\mathcal{F}_t^{W,\bar{W}}$ in the mathematical form of \eqref{x0}, that is, we do not know any information about $x^{v_1,v_2}(x^0)$, thus we should only focus on the information $\mathcal{F}_t^{Y_1^{v_1,v_2}}$ instead of $\mathcal{F}_t^{x^0}$.

Therefore, we will correct the meaning of observation equation in the partially observed problem. We have to emphasize that the state equation can not be observed directly, therefore, by follower's observation equation, the right hand of \eqref{observation f}(\eqref{Y0}) is unknown because of the existence of state variable $x^{v_1,v_2}(x^0)$, and it is unreasonable to consider the information generated by the right hand of \eqref{observation f}(\eqref{Y0}), but we can obtain the partial information by the observation process $Y_1(Y_1^0)$, which is the left hand of \eqref{observation f}(\eqref{Y0}). In other word, the observation process is the only channel for us to partially offer information which can not be observed by the right hand of \eqref{observation f}(\eqref{Y0}). Moreover, due to the absence of state variable $x(\cdot)$ in the drift term of the leader's observation equation \eqref{observation l}, the drift term is observable because it only contains $v_2(\cdot)\in\mathcal{F}_t^{Y_2^{v_2}}$. So we can get $\mathcal{F}_t^{Y_2^{v_2}}=\mathcal{F}_t^W$ by connecting the left hand and right hand of \eqref{observation l}. Now, we can find that the essential difference between \eqref{observation f} and \eqref{observation l} is the existence of state variable $x(\cdot)$. In conclusion, we can only concentrate on the observation process $Y_1(\cdot)(Y_1^0(\cdot))$ in the case of unobservable drift, and it is noteworthy that there is no relation between both sides of \eqref{observation f}(\eqref{Y0}), that is, the only information available to the follower root in $Y_1(\cdot)(Y_1^0(\cdot))$ and do not mix $Y_1(\cdot)(Y_1^0(\cdot))$ and right hand of \eqref{observation f}(\eqref{Y0}). But in the case of observable drift which is usually adapted to filtration generated by the observation process itself, it is rigorous to obtain explicit expression of $\mathcal{F}_t^{Y_2^{v_2}}$ by combining $Y_2(\cdot)$ and the drift term in the right hand of \eqref{observation l}.

However, if the observation equations are considered in the following case:
\begin{equation}
\begin{aligned}
&dY_1(t)=v_1(t)dt+dW(t),\\
&dY_2(t)=v_2(t)dt+dW(t),
\end{aligned}
\end{equation}
where both the drifts are observable for the follower and the leader. By the same analysis in Remark 2.1, it is easy to verify that $\mathcal{F}_t^{Y_1}=\mathcal{F}_t^W=\mathcal{F}_t^{Y_2}$. In this case, we can understand that the informations that they can observe are just the filtration generated by some Brownian motion in the state equation. It is equivalent to recognize that both of them can directly observe the state equation partially. In particular, if  we consider the case that
\begin{equation}
\begin{aligned}
&dY_1(t)=dW(t),\\
&dY_2(t)=dW(t),
\end{aligned}
\end{equation}
it can also be understood as we analyzed above. The similar situation can be referred to \cite{SWX17}, which is related to the partial information problem. However, it should be stressed that the partial information question is quite different form the partially observed one.

Overall, we conclude that partial information problem can be regraded as a kind of direct partial observation, instead, partially observed problem can be treated as a kind of indirect partial observation.
\end{Remark}

\begin{Remark}
In the observation equations of $Y_1(\cdot)$ and $Y_2(\cdot)$, one can also consider that the driven noises are not the same, that is, it is natural to consider the case:
\begin{equation}\label{observation ff}
\left\{
\begin{aligned}
dY_1^{v_1,v_2}(t)&=\big[f_1(t)x^{v_1,v_2}(t)+g(t)\big]dt+dW(t),\ t\in[0,T],\\
 Y_1^{v_1,v_2}(0)&=0,
\end{aligned}
\right.
\end{equation}
\begin{equation}\label{observation ll}
\left\{
\begin{aligned}
dY_2^{v_2}(t)&=\big[f_2(t)+v_2(t)\big]dt+d\bar{W}(t),\ t\in[0,T],\\
 Y_2^{v_2}(0)&=0.
\end{aligned}
\right.
\end{equation}
Actually, the kind of setting can also be suitable to our next results, and will not have any intrinsic difference, so we omit it and leave it to the interested reader.
\end{Remark}

Let us conclude the above formulation with the following definition, and the problem studied in this paper can be proposed in the following.
\begin{definition}
The pair $(\bar{v}_1(\cdot),\bar{v}_2(\cdot))\in\mathcal{U}_{ad}^f[0,T]\times\mathcal{U}_{ad}^l[0,T]$ is called a Stackelberg equilibrium point to the LQ partially observed Stackelberg stochastic differential game, if it satisfies the following condition:\\
(1)\ For any $v_2(\cdot)\in\mathcal{U}_{ad}^l[0,T]$, there exists a map $\Gamma:\mathcal{U}_{ad}^l[0,T]\rightarrow\mathcal{U}_{ad}^f[0,T]$ such that
\begin{equation*}
J_1(\Gamma(v_2(\cdot)),v_2(\cdot))=\min_{v_1(\cdot)\in\mathcal{U}_{ad}^f[0,T]}J_1(v_1(\cdot),v_2(\cdot)).
\end{equation*}
(2)\ There exists a unique $\bar{v}_2(\cdot)\in\mathcal{U}_{ad}^l[0,T]$ such that
\begin{equation*}
J_2(\Gamma(\bar{v}_2(\cdot)),\bar{v}_2(\cdot))=\min_{v_2(\cdot)\in\mathcal{U}_{ad}^l[0,T]}J_2(\Gamma(v_2(\cdot)),v_2(\cdot)).
\end{equation*}
(3)\ The optimal strategy of the follower is $\bar{v}_1(\cdot)=\Gamma(\bar{v}_2(\cdot))$.
\end{definition}

\section{Main Result}

In this section, we aim to study the LQ partially observed Stackelberg stochastic differential game problem in Definition 2.3.

\subsection{Optimization for The Follower}

First, it is vital to obtain the following lemmas from Definition 2.2, similarly to \cite{WWX15}.
\begin{lemma}\label{FollowerLemma}
Given $v_2(\cdot)\in\mathcal{U}_{ad}^l[0,T]$, for any $v_1(\cdot)\in\mathcal{U}_{ad}^f[0,T]$ and $\mathcal{F}_t^{Y_2^{v_2}}\subseteq(\mathcal{F}_t^{Y_1^{v_1,v_2}}\land\mathcal{F}_t^{Y_1^0})$, then $\mathcal{F}_t^{Y_1^{v_1,v_2}}=\mathcal{F}_t^{Y_1^0}$.
\end{lemma}
\begin{proof}
For any $v_1(\cdot)\in\mathcal{U}_{ad}^f[0,T]$, on one hand, since $v_1(\cdot)$ is $\mathcal{F}_t^{Y_1^0}$-adapted and $v_2(\cdot)$ is $\mathcal{F}_t^{Y_2^{v_2}}(\subseteq\mathcal{F}_t^{Y_1^0})$-adapted, it follows from \eqref{x1} that $x^1(\cdot)$ is $\mathcal{F}_t^{Y_1^0}$-adapted, so is $Y_1^1(\cdot)$. Then $Y_1^{v_1,v_2}(\cdot)=Y_1^0(\cdot)+Y_1^1(\cdot)$ is $\mathcal{F}_t^{Y_1^0}$-adapted, i.e., $\mathcal{F}_t^{Y_1^{v_1,v_2}}\subseteq\mathcal{F}_t^{Y_1^0}$.

On the other hand, since $v_1(\cdot)$ is $\mathcal{F}_t^{Y_1^{v_1,v_2}}$-adapted and $v_2(\cdot)$ is $\mathcal{F}_t^{Y_2^{v_2}}(\subseteq\mathcal{F}_t^{Y_1^{v_1,v_2}})$-adapted, then it follows from \eqref{x1} that $x^1(\cdot)$ is $\mathcal{F}_t^{Y_1^{v_1,v_2}}$-adapted, so is $Y_1^1(\cdot)$. Then $Y_1^0(\cdot)=Y_1^{v_1,v_2}(\cdot)-Y_1^1(\cdot)$ is $\mathcal{F}_t^{Y_1^{v_1,v_2}}$-adapted, i.e., $\mathcal{F}_t^{Y_1^0}\subseteq\mathcal{F}_t^{Y_1^{v_1,v_2}}$. The proof is complete.
\end{proof}

\begin{Remark}
In the theoretical aspect, different from \cite{WWX15}, we add the condition $\mathcal{F}_t^{Y_2^{v_2}}\subseteq(\mathcal{F}_t^{Y_1^{v_1,v_2}}\land\mathcal{F}_t^{Y_1^0})$ in Lemma 3.1, which is nontrivial for the technical need under the leader-follower's framework. In the practical application aspect, the information available to the leader is less than that available to the follower is a common phenomenon in many examples. In fact, the dynamic cooperative advertising problem in Section 4 is a motivation for us to study the LQ partially observed Stackelberg stochastic differential game problem.
\end{Remark}

\begin{lemma}
Under {\bf (A1)} and {\bf (A2)}, for given $v_2(\cdot)\in\mathcal{U}_{ad}^l[0,T]$, we have
\begin{equation}\label{f and f0}
\min_{v_1'(\cdot)\in\mathcal{U}_{ad}^f[0,T]}J_1(v_1'(\cdot),v_2(\cdot))=\min_{v_1(\cdot)\in\mathcal{U}_{ad}^{f0}[0,T]}J_1(v_1(\cdot),v_2(\cdot)).
\end{equation}
\end{lemma}
\begin{proof}
It is similar to Lemma 2.3 of \cite{WWX15} and we leave the details to the interested readers.
\end{proof}

\subsubsection{Necessary and Sufficient Conditions}

We first establish a necessary condition and then a sufficient one for optimality of the follower's problem.
From Lemma 3.2, the follower's problem can be solved in $\mathcal{U}_{ad}^{f0}[0,T]$, where the control process is adapted to a uncontrolled filtration, by the variational method.
\begin{theorem}
Let {\bf (A1)} and  {\bf (A2)} hold and $x_0\in\mathbb{R}$. Giving the leader's strategy $v_2(\cdot)\in\mathcal{U}_{ad}^l[0,T]$, let $\bar{v}_1(\cdot)\in\mathcal{U}_{ad}^f[0,T]$ be an optimal control of the follower and $x^{\bar{v}_1,v_2}(\cdot)$ be the corresponding state trajectory of (\ref{LQsde}). Then we have
\begin{equation}\label{optimal control-follower}
\begin{aligned}
R(t)\bar{v}_1(t)+r(t)+B_1(t)\mathbb{E}\big[p(t)\big|\mathcal{F}_t^{Y_1^{\bar{v}_1,v_2}}\big]=0,\ a.e.t\in[0,T],\ a.s.,
\end{aligned}
\end{equation}
and the adjoint process triple $(p(\cdot),k(\cdot),\bar{k}(\cdot))$ satisfies the following BSDE:
\begin{equation}\label{adjoint eq}
\left\{
\begin{aligned}
-dp(t)&=\big[A(t)p(t)+L(t)x^{\bar{v}_1,v_2}(t)+l(t)\big]dt-k(t)dW(t)-\bar{k}(t)d\bar{W}(t),\ t\in[0,T],\\
  p(T)&=Mx^{\bar{v}_1,v_2}(T)+m.
\end{aligned}
\right.
\end{equation}
\end{theorem}

\begin{proof}
In the following, we use the spike variation technique to optimal control $\bar{v}_1(\cdot)$. For some $\epsilon>0$ and $\tau\in[0,T)$, define
\[
v_1^\epsilon(t)=
\begin{cases}
w_1,&\text{$\tau\leq t\leq\tau+\epsilon$},\\
\bar{v}_1(t),&\text{otherwise},
\end{cases}
\]
where $w_1$ is an arbitrary $\mathbb{R}$-valued, $\mathcal{F}_\tau^{Y_1^{\bar{v}_1,v_2}}$-measurable random variable such that $\sup\limits_{\omega\in\Omega}|w_1(\omega)|<\infty$. Let $x^{v_1^\epsilon}(\cdot)$ be the state trajectory of (\ref{LQsde}) corresponding to $v_1^\epsilon(\cdot)$. From Lemma 3.2, $J_1(\bar{v}_1(\cdot),v_2(\cdot))=\min_{v_1(\cdot)\in\mathcal{U}_{ad}^{f0}[0,T]}J_1(v_1(\cdot),v_2(\cdot))$.

We can choose $w_1\in\mathcal{U}_{ad}^{f0}[0,T]$, and introduce the following variational equation:
\begin{equation}\label{variational eq}
\left\{
\begin{aligned}
dx_1(t)&=\big[A(t)x_1(t)+B_1(t)(w_1-\bar{v}_1(t))1_{[\tau,\tau+\epsilon]}(t)\big]dt,\ t\in[0,T],\\
 x_1(0)&=0,
\end{aligned}
\right.
\end{equation}
and can get the following standard estimates:
\begin{equation}\label{estimate-follower}
\begin{aligned}
&\sup_{0\leq t\leq T}\mathbb{E}|x_1(t)|^2\leq C\epsilon^2,\\
&\sup_{0\leq t\leq T}\mathbb{E}\big[|x^{v_1^{\epsilon},v_2}(t)-x^{\bar{v}_1,v_2}(t)-x_1(t)|^2\big]\leq C_{\epsilon}\epsilon^2,\\
&\sup_{0\leq t\leq T}\mathbb{E}\big[|x^{v_1^{\epsilon},v_2}(t)-x^{\bar{v}_1,v_2}(t)|^2\big]\leq C_1\epsilon^2.
\end{aligned}
\end{equation}
And then we obtain the variational inequality:
\begin{equation}\label{variational ineq}
\begin{aligned}
&\frac{1}{2}\mathbb{E}\bigg\{\int_0^T\Big[2L(t)x^{\bar{v}_1,v_2}(t)x_1(t)+2l(t)x_1(t)+R(t)\big[w_1^2-\bar{v}_1^2(t)\big]1_{[\tau,\tau+\epsilon]}(t)\\
&\qquad\qquad+2r(t)(w_1-\bar{v}_1(t))1_{[\tau,\tau+\epsilon]}(t)\Big]dt+2Mx^{\bar{v}_1,v_2}(T)x_1(T)+2mx_1(T)\bigg\}\geq o(\epsilon).
\end{aligned}
\end{equation}
Applying It\^{o}'s formula to $p(\cdot)x_1(\cdot)$, we have
\begin{equation}\label{px1}
\begin{aligned}
&\mathbb{E}\big[Mx^{\bar{v}_1,v_2}(T)x_1(T)+mx_1(T)\big]\\
&=\mathbb{E}\int_0^T\Big[\big(-L(t)x^{\bar{v}_1,v_2}(t)-l(t)\big)x_1(t)+B_1(t)p(t)(w_1-\bar{v}_1(t))1_{[\tau,\tau+\epsilon]}(t)\Big]dt.
\end{aligned}
\end{equation}
Substituting \eqref{px1} into \eqref{variational ineq}, we get
\begin{equation}\label{variational ineq1}
\begin{aligned}
&\mathbb{E}\bigg\{\int_0^T\bigg[\frac{1}{2}R(t)\big[w_1^2-\bar{v}_1^2(t)\big]+r(t)(w_1-\bar{v}_1(t))\\
&\qquad\qquad+B_1(t)p(t)(w_1-\bar{v}_1(t))\bigg]1_{[\tau,\tau+\epsilon]}(t)dt\bigg\}\geq o(\epsilon).
\end{aligned}
\end{equation}
Hence
\begin{equation}\label{optimal control-follower0}
R(t)\bar{v}_1(t)+r(t)+B_1(t)\mathbb{E}\big[p(t)|\mathcal{F}_t^{Y_1^0}\big]=0,\ a.e.t\in[0,T],\ a.s.
\end{equation}
Moreover, since $\bar{v}_1(\cdot)\in\mathcal{U}_{ad}^f[0,T]$, from Lemma 3.1 we have $\mathcal{F}_t^{Y_1^{\bar{v}_1,v_2}}=\mathcal{F}_t^{Y_1^0}$. Thus we get the desired conclusion.
\end{proof}

\begin{theorem}
Let {\bf (A1)} and {\bf (A2)} hold and $x_0\in\mathbb{R}$. Giving the leader's strategy $v_2(\cdot)\in\mathcal{U}_{ad}^l[0,T]$, let $\bar{v}_1(\cdot)\in\mathcal{U}_{ad}^f[0,T]$ satisfy \eqref{optimal control-follower} where $(p(\cdot),k(\cdot),\bar{k}(\cdot))$ is the solution to \eqref{adjoint eq}. Then $\bar{v}_1(\cdot)$ is an optimal control of the follower.
\end{theorem}
\begin{proof}
For any $v_1(\cdot)\in\mathcal{U}_{ad}^f[0,T]$, from \eqref{LQ cost functional}, we have
\begin{equation*}
J_1(v_1(\cdot),v_2(\cdot))-J_1(\bar{v}_1(\cdot),v_2(\cdot))=\textrm{I}+\textrm{II},
\end{equation*}
where
\begin{equation}\label{I}
\begin{aligned}
\textrm{I}=&\ \frac{1}{2}\mathbb{E}\bigg\{\int_0^T\big[L(t)(x^{v_1,v_2}(t)-x^{\bar{v}_1,v_2}(t))^2+R(t)(v_1(t)-\bar{v}_1(t))^2\big]dt\\
&\qquad+M(x^{v_1,v_2}(T)-x^{\bar{v}_1,v_2}(T))^2\bigg\},
\end{aligned}
\end{equation}
\begin{equation}\label{II}
\begin{aligned}
\textrm{II}=&\ \mathbb{E}\bigg\{\int_0^T\big[\big(l(t)+L(t)x^{\bar{v}_1,v_2}(t)\big)\big(x^{v_1,v_2}(t)-x^{\bar{v}_1,v_2}(t)\big)+\big(R(t)\bar{v}_1+r(t)\big)\\
&\qquad\times\big(v_1(t)-\bar{v}_1(t)\big)\big]dt+\big(Mx^{\bar{v}_1,v_2}(T)+m\big)\big(x^{v_1,v_2}(T)-x^{\bar{v}_1,v_2}(T)\big)\bigg\}.
\end{aligned}
\end{equation}

Applying It\^{o}'s formula to $p(\cdot)\big(x^{v_1,v_2}(\cdot)-x^{\bar{v}_1,v_2}(\cdot)\big)$, we get
\begin{equation}\label{px}
\begin{aligned}
&\mathbb{E}\big[(Mx^{\bar{v}_1,v_2}(T)+m)(x^{v_1,v_2}(T)-x^{\bar{v}_1,v_2}(T)\big]=\mathbb{E}\int_0^T\Big[B_1(t)p(t)(v_1(t)-\bar{v}_1(t))\\
&\quad-L(t)x^{\bar{v}_1,v_2}(t)(x^{v_1,v_2}(t)-x^{\bar{v}_1,v_2}(t))-l(t)(x^{v_1,v_2}(t)-x^{\bar{v}_1,v_2}(t))\Big]dt.
\end{aligned}
\end{equation}
Substituting \eqref{optimal control-follower} and \eqref{px} into \eqref{II}, we can obtain
\begin{equation}
\begin{aligned}
\textrm{II}=\mathbb{E}\int_0^T\big[R(t)\bar{v}_1(t)+r(t)+B_1(t)\mathbb{E}[p(t)|\mathcal{F}_t^{Y_1^0}]\big](v_1-\bar{v}_1))dt=0.
\end{aligned}
\end{equation}
It is obvious that $\textrm{I}\geq0$. From Lemma 3.1, we complete the proof.
\end{proof}

\subsubsection{Filtering Equations}

We need the following assumption.\\
{\bf(A3)}
$R^{-1}(\cdot)$ is a uniformly bounded, deterministic function.

Recall that $Y_1^{v_1,v_2}(\cdot)$ governed by \eqref{observation f} is the observation process of the follower. For any $v_1(\cdot)\in\mathcal{U}_{ad}^f[0,T]$, we let
\begin{equation}\label{definition}
\begin{aligned}
&\hat{\eta}(t):=\mathbb{E}[\eta(t)|\mathcal{F}_t^{Y_1^{v_1,v_2}}],\ \mbox{ for\ }\eta(t)=x^0(t), x^{v_1,v_2}(t), p(t), k(t)\mbox{ and\ }\bar{k}(t),\\
&P(t):=\mathbb{E}\big[x^{v_1,v_2}(t)-\hat{x}^{v_1,v_2}(t)\big]^2,\ t\in[0,T],
\end{aligned}
\end{equation}
be the optimal filtering $\eta(\cdot)$, and the mean square error of $\hat{x}^{v_1,v_2}(\cdot)$, respectively.

If {\bf(A1), (A2), (A3)} hold, then from \eqref{optimal control-follower}, $\bar{v}_1(\cdot)$ satisfies
\begin{equation}\label{barv1}
\begin{aligned}
\bar{v}_1(t)&=-R^{-1}(t)\big[B_1(t)\mathbb{E}\big[p(t)|\mathcal{F}_t^{Y_1^{\bar{v}_1,v_2}}\big]+r(t)\big]\\
            &=-R^{-1}(t)\big[B_1(t)\hat{p}(t)+r(t)\big],\ a.e.t\in[0,T],\ a.s.,
\end{aligned}
\end{equation}
which is represented by the filtering estimate of $p(\cdot)$. Thus it is necessary to study the filtering equations of the state and adjoint variables in the following.
\begin{lemma}
Let {\bf(A1)} and {\bf (A2)} hold, giving the control $v_2(\cdot)\in\mathcal{U}_{ad}^l[0,T]$, for any $v_1(\cdot)\in\mathcal{U}_{ad}^f[0,T]$, the optimal filtering estimate $\hat{x}^{v_1,v_2}(\cdot)$ of the solution $x^{v_1,v_2}(\cdot)$ to \eqref{LQsde} with respect to $\mathcal{F}_t^{Y_1^{v_1,v_2}}$ satisfies
\begin{equation}\label{filtering x}
\left\{
\begin{aligned}
d\hat{x}^{v_1,v_2}(t)&=\big[A(t)\hat{x}^{v_1,v_2}(t)+B_1(t)v_1(t)+B_2(t)v_2(t)+\alpha(t)\big]dt\\
                     &\quad+\big[c(t)+f_1(t)P(t)\big]d\widetilde{W}(t),\ t\in[0,T],\\
 \hat{x}^{v_1,v_2}(0)&=x_0,
\end{aligned}
\right.
\end{equation}
where the mean square error $P(\cdot)$ of the filtering estimate $\hat{x}^{v_1,v_2}(\cdot)$ is the unique solution to the following ODE:
\begin{equation}\label{P}
\left\{
\begin{aligned}
&\dot{P}(t)-2A(t)P(t)+\big(c(t)+f(t)P(t)\big)^2-c^2(t)-\bar{c}^2(t)=0,\ t\in[0,T],\\
&P(0)=0,
\end{aligned}
\right.
\end{equation}
and
\begin{equation}\label{innovation}
\begin{aligned}
\widetilde{W}(t)&:=\int_0^tdY_1^{v_1,v_2}(s)-\big(f_1(s)\hat{x}^{v_1,v_2}(s)+g(s)\big)ds\\
&\equiv\int_0^tf_1(s)\big(x^{v_1,v_2}(s)-\hat{x}^{v_1,v_2}(s)\big)ds+W(t)
\end{aligned}
\end{equation}
is a standard $\mathcal{F}_t^{Y_1^{v_1,v_2}}$-Brownian motion with values in $\mathbb{R}$.
\end{lemma}
\begin{proof}
We can easily deduce the result using Theorem 8.1 in Liptser and Shiryayev \cite{LS77} and Theorem 5.7 in Xiong \cite{Xiong08}. The detail is omitted.
\end{proof}

Similar to Lemmas 3.1 and 3.2 of \cite{WWX15}, we have the following result for the filtering equation of the adjoint equation \eqref{adjoint eq}.
\begin{lemma}
Let {\bf(A1), (A2), (A3)} hold, giving the control $v_2(\cdot)\in\mathcal{U}_{ad}^l[0,T]$, the optimal filtering estimate of $(p(\cdot),k(\cdot),\bar{k}(\cdot))$ with respect to $\mathcal{F}_t^{Y_1^{\bar{v}_1,v_2}}$ satisfies
\begin{equation}\label{filtering p}
\left\{
\begin{aligned}
-d\hat{p}(t)&=\big[L(t)\hat{x}^{\bar{v}_1,v_2}(t)+A(t)\hat{p}(t)+l(t)\big]dt-\hat{\mathcal{K}}(t)d\widetilde{W}(t),\ \ t\in[0,T],\\
  \hat{p}(T)&=M\hat{x}^{\bar{v}_1,v_2}(T)+m,
\end{aligned}
\right.
\end{equation}
with
\begin{equation}\label{KKK}
\hat{\mathcal{K}}(t):=\hat{k}(t)+f_1(t)\big[\widehat{x(t)p(t)}-\hat{x}(t)\hat{p}(t)\big],\ t\in[0,T],
\end{equation}
where $\hat{x}^{\bar{v}_1,v_2}(\cdot)$ and $\widetilde{W}(\cdot)$ satisfy \eqref{filtering x} and \eqref{innovation} with $v_1(\cdot)=\bar{v}_1(\cdot)$, respectively.
\end{lemma}

\subsubsection{State Estimate Feedback Representation}

In this subsection, we study the state feedback representation of the optimal control $\bar{v}_1(\cdot)$ of the follower, with the given control $v_2(\cdot)$ of the leader.

Giving the control $v_2(\cdot)$, for any $v_1(\cdot)\in\mathcal{U}_{ad}^f[0,T]$, we know that \eqref{LQsde} and \eqref{adjoint eq} admit unique solutions, respectively. Noting the terminal condition of \eqref{adjoint eq} and the appearance of control variable $v_2(\cdot)$, we set
\begin{equation}\label{p1}
p(t)=\Pi(t)x^{\bar{v}_1,v_2}(t)+\theta(t),\ t\in[0,T],
\end{equation}
where $\Pi(\cdot)$ is a deterministic and differentiable function with $\Pi(T)=M$, and $\theta(\cdot)$ is an $\mathbb{R}$-valued, $\mathcal{F}_t$-adapted process which satisfies the following BSDE:
\begin{equation}\label{theta}
\left\{
\begin{aligned}
d\theta(t)&=\beta(t)dt+\lambda_1(t)dW(t)+\lambda_2(t)d\bar{W}(t),\ t\in[0,T],\\
\theta(T)&=m.
\end{aligned}
\right.
\end{equation}
In the above equation, the process $\beta(\cdot)$ will be determined later.

Applying It\^{o}'s formula to \eqref{p1}, we obtain
\begin{equation}\label{itop}
\begin{aligned}
dp(t)&=\big[A(t)\Pi(t)x^{\bar{v}_1,v_2}(t)+B_1(t)\Pi(t)\bar{v}_1(t)+B_2(t)\Pi(t)v_2(t)+\alpha(t)\Pi(t)\\
     &\quad+\dot{\Pi}(t)x^{\bar{v}_1,v_2}(t)+\beta(t)\big]dt+\big[c(t)\Pi(t)+\lambda_1(t)\big]dW(t)+\big[\bar{c}(t)\Pi(t)+\lambda_2(t)\big]d\bar{W}(t).
\end{aligned}
\end{equation}
Substituting \eqref{p1} into \eqref{barv1}, we achieve
\begin{equation}\label{Barv1}
\bar{v}_1(t)=-R^{-1}(t)\big[B_1(t)\Pi(t)\hat{x}^{\bar{v}_1,v_2}(t)+B_1(t)\hat{\theta}(t)+r(t)\big],\ a.e.t\in[0,T],\ a.s.
\end{equation}
Then plugging \eqref{Barv1} into \eqref{itop}, we get
\begin{equation}\label{adjoint eq0}
\begin{aligned}
dp(t)&=\big[A(t)\Pi(t)x^{\bar{v}_1,v_2}(t)-B_1^2(t)\Pi^2(t)R^{-1}(t)\hat{x}^{\bar{v}_1,v_2}(t)-B_1^2(t)\Pi(t)R^{-1}(t)\hat{\theta}(t)\\
&\quad-B_1(t)\Pi(t)R^{-1}(t)r(t)+B_2(t)\Pi(t)v_2(t)+\alpha(t)\Pi(t)+\dot{\Pi}(t)x^{\bar{v}_1,v_2}(t)+\beta(t)\big]dt\\
&\quad+\big[c(t)\Pi(t)+\lambda_1(t)\big]dW(t)+\big[\bar{c}(t)\Pi(t)+\lambda_2(t)\big]d\bar{W}(t).
\end{aligned}
\end{equation}
Then comparing the drift term in \eqref{adjoint eq0} with that of \eqref{adjoint eq}, we have
\begin{equation}\label{driftterm}
\begin{aligned}
&\big[\dot{\Pi}(t)+2A(t)\Pi(t)-B_1^2(t)\Pi^2(t)R^{-1}(t)+L(t)\big]x^{\bar{v}_1,v_2}(t)+B_1^2(t)\Pi^2(t)R^{-1}(t)x^{\bar{v}_1,v_2}(t)\\
&-B_1^2(t)\Pi^2(t)R^{-1}(t)\hat{x}^{\bar{v}_1,v_2}(t)+A(t)\theta(t)-B_1^2(t)\Pi(t)R^{-1}(t)\hat{\theta}(t)+B_2(t)\Pi(t)v_2(t)\\
&-B_1(t)\Pi(t)R^{-1}(t)r(t)+\alpha(t)\Pi(t)+l(t)+\beta(t)=0,\ a.e.t\in[0,T].
\end{aligned}
\end{equation}
From the coefficients of $x^{\bar{v}_1,v_2}(\cdot)$ in \eqref{driftterm}, we derive if the following Riccati equation:
\begin{equation}\label{riccatieq}
\left\{
\begin{aligned}
&\dot{\Pi}(t)+2A(t)\Pi(t)-B_1^2(t)\Pi^2(t)R^{-1}(t)+L(t)=0,\ t\in[0,T],\\
&\Pi(T)=M,
\end{aligned}
\right.
\end{equation}
admits a unique solution $\Pi(\cdot)$, then
\begin{equation}\label{beta}
\begin{aligned}
-\beta(t)&=B_1^2(t)\Pi^2(t)R^{-1}(t)x^{\bar{v}_1,v_2}(t)-B_1^2(t)\Pi^2(t)R^{-1}(t)\hat{x}^{\bar{v}_1,v_2}(t)+A(t)\theta(t)\\
&\quad-B_1^2(t)\Pi(t)R^{-1}(t)\hat{\theta}(t)+B_2(t)\Pi(t)v_2(t)-B_1(t)\Pi(t)R^{-1}(t)r(t)\\
&\quad+\alpha(t)\Pi(t)+l(t),\ a.e.t\in[0,T].
\end{aligned}
\end{equation}
In fact, the solvability of \eqref{riccatieq} is standard. Thus the BSDE \eqref{theta} takes the form:
\begin{equation}\label{theta1}
\left\{
\begin{aligned}
-d\theta(t)&=\big[B_1^2(t)\Pi^2(t)R^{-1}(t)x^{\bar{v}_1,v_2}(t)-B_1^2(t)\Pi^2(t)R^{-1}(t)\hat{x}^{\bar{v}_1,v_2}(t)+A(t)\theta(t)\\
&\qquad-B_1^2(t)\Pi(t)R^{-1}(t)\hat{\theta}(t)+B_2(t)\Pi(t)v_2(t)-B_1(t)\Pi(t)R^{-1}(t)r(t)\\
&\qquad+\alpha(t)\Pi(t)+l(t)\big]dt-\lambda_1(t)dW(t)-\lambda_2(t)d\bar{W}(t),\ t\in[0,T],\\
\theta(T)&=m.
\end{aligned}
\right.
\end{equation}
For given $v_2(\cdot)$, plugging the optimal control $\bar{v}_1(\cdot)$ of \eqref{Barv1} into \eqref{filtering x}, we have
\begin{equation}\label{filtering x2}
\left\{
\begin{aligned}
d\hat{x}^{\bar{v}_1,v_2}(t)&=\big\{\big[A(t)-B_1^2(t)\Pi(t)R^{-1}(t)\big]\hat{x}^{\bar{v}_1,v_2}(t)-B_1^2(t)R^{-1}(t)\hat{\theta}(t)\\
&\qquad-B_1(t)R^{-1}(t)r(t)+B_2(t)v_2(t)+\alpha(t)\big\}dt\\
&\quad+\big[c(t)+f_1(t)P(t)\big]d\widetilde{W}(t),\ t\in[0,T],\\
\hat{x}^{v_1,v_2}(0)&=x_0.
\end{aligned}
\right.
\end{equation}
Similar to Lemma 3.4, we can get the filtering estimate $\hat{\theta}(\cdot)$ of the solution $\theta(\cdot)$ to \eqref{theta1} with respect to $\mathcal{F}_t^{Y_1^{\bar{v}_1,v_2}}$ in the following:
\begin{equation}\label{hattheta1}
\left\{
\begin{aligned}
-d\hat{\theta}(t)&=\big\{\big[A(t)-B_1^2(t)\Pi(t)R^{-1}(t)\big]\hat{\theta}(t)+B_2(t)\Pi(t)v_2(t)-B_1(t)\Pi(t)R^{-1}(t)r(t)\\
&\qquad+\alpha(t)\Pi(t)+l(t)\big\}dt-\hat{\Gamma}(t)d\widetilde{W}(t),\ t\in[0,T],\\
\hat{\theta}(T)&=m,
\end{aligned}
\right.
\end{equation}
where
\begin{equation}\label{GammaGamma}
\hat{\Gamma}(t):=\hat{\lambda}_1(t)+f_1(t)\big[\widehat{\theta(t)x(t)}-\hat{\theta}(t)\hat{x}(t)\big],\ t\in[0,T],
\end{equation}
and $\Pi(\cdot)$, $\hat{x}^{\bar{v}_1,v_2}(\cdot)$ and $\widetilde{W}(\cdot)$ satisfy \eqref{riccatieq}, \eqref{filtering x2} and \eqref{innovation} with $v_1(\cdot)=\bar{v}_1(\cdot)$, respectively.

Putting \eqref{filtering x2} and \eqref{hattheta1} together, we get
\begin{equation}\label{hatfbsde}
\left\{
\begin{aligned}
d\hat{x}^{\bar{v}_1,v_2}(t)&=\big\{\big[A(t)-B_1^2(t)\Pi(t)R^{-1}(t)\big]\hat{x}^{\bar{v}_1,v_2}(t)-B_1^2(t)R^{-1}(t)\hat{\theta}(t)-B_1(t)R^{-1}(t)r(t)\\
&\qquad+B_2(t)v_2(t)+\alpha(t)\big\}dt+\big[c(t)+f_1(t)P(t)\big]d\widetilde{W}(t),\\
-d\hat{\theta}(t)&=\big\{\big[A(t)-B_1^2(t)\Pi(t)R^{-1}(t)\big]\hat{\theta}(t)+B_2(t)\Pi(t)v_2(t)-B_1(t)\Pi(t)R^{-1}(t)r(t)\\
&\qquad+\alpha(t)\Pi(t)+l(t)\big\}dt-\hat{\Gamma}(t)d\widetilde{W}(t),\ t\in[0,T],\\
\hat{x}^{v_1,v_2}(0)&=x_0,\ \hat{\theta}(T)=m.
\end{aligned}
\right.
\end{equation}
Note that \eqref{hatfbsde} is an FBSDE, whose solvability can be easily obtained, for given $v_2(\cdot)$.

We summarize the above procedure in the following theorem.
\begin{theorem}
Let {\bf(A1), (A2), (A3)} hold. For any given $x_0\in\mathbb{R}$ and $v_2(\cdot)\in\mathcal{U}_{ad}^l[0,T]$, the follower's problem is solvable with the optimal strategy $\bar{v}_1(\cdot)$ being of a state estimate feedback representation as
\begin{equation}\label{Barv2}
\bar{v}_1(t)=-R^{-1}(t)\big[B_1(t)\Pi(t)\hat{x}^{\bar{v}_1,v_2}(t)+B_1(t)\hat{\theta}(t)+r(t)\big],\ a.e.t\in[0,T],\ a.s.,
\end{equation}
where $\Pi(\cdot)$ is the solution to \eqref{riccatieq}. The filtering estimate of the optimal state trajectory $\hat{x}^{\bar{v}_1,v_2}(\cdot)$ and the process pair $(\hat{\theta}(\cdot),\hat{\Gamma}(\cdot))$ are the unique $\mathcal{F}_t^{Y_1^{\bar{v}_1,v_2}}$-adapted solution to the FBSDFEs \eqref{hatfbsde}.
\end{theorem}

\begin{Remark}
We highlight that the $\hat{\mathcal{K}}(\cdot)$ is a part of the solution $(\hat{p}(\cdot),\hat{\mathcal{K}}(\cdot))$ to the BSDFE \eqref{filtering p} and $\hat{\Gamma}(\cdot)$ is a part of the solution $(\hat{\theta}(\cdot),\hat{\Gamma}(\cdot))$ to the BSDFE \eqref{hattheta1}, and $\widetilde{W}(\cdot)$ satisfies \eqref{innovation} with $v_1(\cdot)=\bar{v}_1(\cdot)$.
\end{Remark}

\begin{proof}
For given $v_2(\cdot)$, let $\Pi(\cdot)$ satisfy \eqref{riccatieq}. By the standard BSDE theory, we can solve \eqref{hattheta1} to obtain $(\hat{\theta}(\cdot),\hat{\Gamma}(\cdot))$, and $P(\cdot)$ is the unique solution to \eqref{P}. Then we can obtain $\hat{x}^{v_1,v_2}(\cdot)$, which is the $\mathcal{F}_t^{Y_1^{\bar{v}_1,v_2}}$-adapted solution to \eqref{filtering x2}. Therefore, the state estimate feedback representation \eqref{Barv2} can be obtained. The proof is complete.
\end{proof}

\subsection{Optimization for The Leader}

\subsubsection{Maximum Principle}
Since $\hat{\theta}(\cdot)$ is involved in the coefficient of \eqref{LQsde} when inserting \eqref{Barv1} into \eqref{LQsde}, \eqref{hattheta1} can be regarded as the new backward ``state" equation of the leader. Then we give the state equation of the leader as follows (Some time variables $t$ will be omitted for simplicity):
\begin{equation}\label{leaderstate}
\left\{
\begin{aligned}
dx^{v_2}(t)&=\big[Ax^{v_2}-B_1^2\Pi R^{-1}\hat{x}^{v_2}-B_1^2R^{-1}\hat{\theta}-B_1R^{-1}r+B_2v_2+\alpha\big]dt\\
&\quad+cdW(t)+\bar{c}d\bar{W}(t),\\
-d\hat{\theta}(t)&=\big\{[A-B_1^2\Pi R^{-1}]\hat{\theta}+B_2\Pi v_2-B_1\Pi R^{-1}r+\alpha\Pi+l\big\}dt-\hat{\Gamma}d\widetilde{W}(t),\ t\in[0,T],\\
x^{v_2}(0)&=x_0,\ \hat{\theta}(T)=m,
\end{aligned}
\right.
\end{equation}
where we set $x^{v_2}(\cdot)\equiv x^{\bar{v}_1,v_2}(\cdot)$. First, we find that, which is vital, the $\mathcal{F}_t^{Y_1}$-adapted Brownian motion $\widetilde{W}(\cdot)$ in \eqref{leaderstate} is independent of the control variable $v_2(\cdot)$. Indeed,
\begin{equation}\label{widetildeW2}
\begin{aligned}
d\widetilde{W}&=dY_1^{v_2}-(f_1\hat{x}^{v_2}+g)dt\\
&=dY_1^0+dY_1^1-[f_1(\hat{x}^0+x^1)+g]dt\\
&=dY_1^0-f_1\hat{x}^0dt,
\end{aligned}
\end{equation}
where $Y_1^0(\cdot)$ and $x^0(\cdot)$ are independent of the control variables. Thus we get the result and it is convenient for us to apply the spike variational method.
\begin{Remark}
Firstly, it is easy to check that $\widetilde{W}(\cdot)$ is an $\mathcal{F}_t^{Y_1}$-adapted Brownian motion independent of the control variable $v_2(\cdot)$ under the expectation $\mathbb{E}$ and $\mathbb{E}\int_0^T\{\cdots\}d\widetilde{W}(t)=0$. Secondly, going back to the observation equation \eqref{observation f}, it is easy to find that $Y_1(\cdot)$ depends on the state $x^{v_2}(\cdot)$ which is driven by the Brownian motions $W(\cdot)$ and $\bar{W}(\cdot)$. Thus the stochastic process $Y_1(\cdot)$ is not independent of $W(\cdot)$ or $\bar{W}(\cdot)$, then $\widetilde{W}(\cdot)$ is not independent of $W(\cdot)$ or $\bar{W}(\cdot)$. Thirdly, since the leader faces an LQ partially observed optimal control problem of {\it conditional mean-field forward-backward stochastic differential equations} (CMFFBSDEs for short), and due to the appearance of the new $\mathcal{F}_t^{Y_1}$-Brownian motion $\widetilde{W}(\cdot)$, a key step in this section is to ensure a reasonable form of adjoint FBSDE and to study its uniqueness and existence of adapted solutions.
\end{Remark}

In the following, we use the spike variation technique to the optimal control $\bar{v}_2(\cdot)$ of the leader. For some $\rho>0$ and $\tau\in[0,T)$, define
\[
v_2^{\rho}(t)=
\begin{cases}
w_2,&\text{$\tau\leq t\leq\tau+\rho$},\\
\bar{v}_2(t),&\text{otherwise},
\end{cases}
\]
where $w_2$ is an arbitrary $\mathbb{R}$-valued, $\mathcal{F}_\tau^{Y_2^{\bar{v}_2}}$-measurable random variable such that $\sup\limits_{\omega\in\Omega}|v_2(\omega)|<\infty$. Let $x^{v_2^{\rho}}(\cdot)$ be the trajectory corresponding to $v_2^{\rho}(\cdot)$.

Similarly in Section 3.1, we introduce the following variational equations:
\begin{equation}\label{variational eq2}
\left\{
\begin{aligned}
d\xi(t)&=\big[A\xi-B_1^2\Pi R^{-1}\hat{\xi}-B_1^2R^{-1}\hat{\theta}_1+B_2(w_2-\bar{v}_2)1_{[\tau,\tau+\rho]}\big]dt,\ t\in[0,T],\\
 \xi(0)&=0,
\end{aligned}
\right.
\end{equation}
\begin{equation}\label{variational eq3}
\left\{
\begin{aligned}
-d\hat{\theta}_1(t)&=\big[(A-B_1^2\Pi R^{-1})\hat{\theta}_1+B_2\Pi(w_2-\bar{v}_2)1_{[\tau,\tau+\rho]}\big]dt-\hat{\Gamma}_1d\widetilde{W}(t),\ t\in[0,T],\\
  \hat{\theta}_1(T)&=0.
\end{aligned}
\right.
\end{equation}
Similar to Lemma 1 and Lemma 2 in Xu \cite{Xu95}, we can get the following estimates.
\begin{lemma}
Suppose {\bf (A1), (A2), (A3)} hold, we have
\begin{equation}\label{estimate-leader}
\begin{aligned}
&\sup_{0\leq t\leq T}\mathbb{E}|\xi(t)|^2\leq C\rho^2,\ \ \sup_{0\leq t\leq T}\mathbb{E}|\xi(t)|^4\leq C\rho^4,\\
&\sup_{0\leq t\leq T}\mathbb{E}|\hat{\xi}(t)|^2\leq C\rho^2,\ \ \sup_{0\leq t\leq T}\mathbb{E}|\hat{\xi}(t)|^4\leq C\rho^4,\\
&\sup_{0\leq t\leq T}\mathbb{E}|\hat{\theta}_1(t)|^2\leq C\rho^2,\ \ \sup_{0\leq t\leq T}\mathbb{E}|\hat{\theta}_1(t)|^4\leq C\rho^4,\\
&\mathbb{E}\int_0^T|\hat{\Gamma}_1(t)|^2dt\leq C\rho^2,\ \ \mathbb{E}\bigg(\int_0^T|\hat{\Gamma}_1(t)|^2dt\bigg)^2\leq C\rho^4,\\
&\sup_{0\leq t\leq T}\mathbb{E}[|x^{v_2^{\rho}}(t)-x^{\bar{v}_2}(t)-\xi(t)|^2]=0,\ \ \sup_{0\leq t\leq T}\mathbb{E}[|\hat{x}^{v_2^{\rho}}(t)-\hat{x}^{\bar{v}_2}(t)-\hat{\xi}(t)|^2]=0,\\
&\sup_{0\leq t\leq T}\mathbb{E}[|\hat{\theta}^{\rho}(t)-\hat{\bar{\theta}}(t)-\hat{\theta}_1(t)|^2]=0,\ \ \mathbb{E}\int_0^T|\hat{\Gamma}^{\rho}(t)-\hat{\bar{\Gamma}}(t)-\hat{\Gamma}_1(t)|^2dt=0.
\end{aligned}
\end{equation}
\end{lemma}
Since $\bar{v}_2(\cdot)$ is an optimal strategy of the leader, we get
\begin{equation*}
J_2(v_2^{\rho}(\cdot))-J_2(\bar{v}_2(\cdot))\geq0.
\end{equation*}
From this, the variational inequality can also be derived
\begin{equation}\label{variational ineq2}
\begin{aligned}
\mathbb{E}\bigg\{\int_0^T\big[(\bar{L}x^{\bar{v}_2}+\bar{l})\xi+(\bar{R}\bar{v}_2+\bar{r})(w_2-\bar{v}_2)1_{[\tau,\tau+\rho]}\big]dt+(\bar{M}x^{\bar{v}_2}(T)+\bar{m})\xi(T)\bigg\}\geq o(\rho).
\end{aligned}
\end{equation}
By means of $\big\langle d\widetilde{W},dW\big\rangle_t=dt$ and $\big\langle d\widetilde{W},d\bar{W}\big\rangle_t=0$, we introduce the adjoint equations as follows:
\begin{equation}\label{lambda}
\left\{
\begin{aligned}
-d\lambda(t)&=\big(\bar{L}x^{\bar{v}_2}+\bar{l}+A\lambda-B_1^2\Pi R^{-1}\hat{\lambda}\big)dt-\eta_1dW(t)-\eta_2d\bar{W}(t)-\eta_3d\widetilde{W}(t),\ t\in[0,T],\\
  \lambda(T)&=\bar{M}x^{\bar{v}_2}(T)+\bar{m},
\end{aligned}
\right.
\end{equation}
\begin{equation}\label{K}
\left\{
\begin{aligned}
dK(t)&=\big[-B_1^2R^{-1}\lambda+(A-B_1^2\Pi R^{-1})K\big]dt,\ t\in[0,T],\\
 K(0)&=0,
\end{aligned}
\right.
\end{equation}
where $(\lambda(\cdot),\eta_1(\cdot),\eta_2(\cdot))$ and $K(\cdot)$ are $\mathcal{F}_t$-adapted processes, while $\eta_3(\cdot)$ is an $\mathcal{F}_t^{Y_1}$-adapted process which will be explained later.

Applying It\^{o}'s formula to $\lambda(\cdot)\xi(\cdot)-K(\cdot)\hat{\theta}_1(\cdot)$, noting \eqref{variational eq2}, \eqref{variational eq3}, \eqref{lambda} and \eqref{K}, the variational inequality \eqref{variational ineq2} becomes
\begin{equation}\label{variational ineq3}
\begin{aligned}
\mathbb{E}\int_0^T\mathbb{E}\big[\bar{R}\bar{v}_2+\bar{r}+B_2\lambda+B_2\Pi K\big|\mathcal{F}_t^W\big](w_2-\bar{v}_2)1_{[\tau,\tau+\rho]}dt\geq o(\rho).
\end{aligned}
\end{equation}
Hence
\begin{equation}\label{maximum condition1}
\bar{R}\bar{v}_2+\bar{r}+B_2\mathbb{E}\big[\lambda|\mathcal{F}_t^W\big]+B_2\Pi\mathbb{E}\big[K|\mathcal{F}_t^W\big]=0,\ a.e.t\in[0,T],\ a.s.
\end{equation}
Since $\bar{v}_2(\cdot)\in\mathcal{U}_{ad}^l[0,T]$, form Remark 2.1, we have $\mathcal{F}_t^{Y_2^{\bar{v}_2}}=\mathcal{F}_t^W$. Then \eqref{maximum condition1} is rewritten as
\begin{equation}\label{maximum condition2}
\bar{R}\bar{v}_2+\bar{r}+B_2\check{\lambda}+B_2\Pi\check{K}=0,\ a.e.t\in[0,T],\ a.s.,
\end{equation}
where $\check{\lambda}(t):=\mathbb{E}\big[\lambda(t)|\mathcal{F}_t^{Y_2^{\bar{v}_2}}\big],\check{K}(t):=\mathbb{E}\big[K(t)|\mathcal{F}_t^{Y_2^{\bar{v}_2}}\big]$.

We summarize the above in the following theorem.

\begin{theorem}
Let {\bf (A1), (A2), (A3)} hold. Suppose that $\bar{v}_2(\cdot)$ is an optimal strategy of the leader, and $(x^{\bar{v}_2}(\cdot),\hat{\bar{\theta}}(\cdot))$ is the corresponding optimal trajectory. Then the maximum condition of $\bar{v}_2(\cdot)$ \eqref{maximum condition2} holds, where the adjoint processes $(\lambda(\cdot),\eta_1(\cdot),\eta_2(\cdot),\eta_3(\cdot))$ and $K(\cdot)$ satisfy \eqref{lambda} and \eqref{K}, respectively.
\end{theorem}

\begin{Remark}
We firstly note that the backward adjoint equation \eqref{lambda} is not only driven by the independent standard Brownian motions $W(\cdot)$ and $\bar{W}(\cdot)$, but also by the innovation process $\widetilde{W}(\cdot)$. However, as we have mentioned before, $\widetilde{W}(\cdot)$ is not independent of $W$ or $\bar{W}$ under the probability measure $P$, which is a nontrivial phenomenon due to the partially observable information. Therefore, secondly, it is necessary to guarantee the unique solvability of \eqref{lambda} and \eqref{K}. We aim to make them clear in the next subsection.
\end{Remark}

We can also obtain the following sufficient conditions for the optimality of the leader.

\begin{theorem}
Let {\bf (A1), (A2), (A3)} hold and let $\bar{v}_2(\cdot)\in\mathcal{U}_{ad}^l[0,T]$ satisfy \eqref{maximum condition2}
where $(\lambda(\cdot),\eta_1(\cdot),\eta_2(\cdot),\eta_3(\cdot))$ is the solution to \eqref{lambda} and $K(\cdot)$ is the solution to \eqref{K}. Then $\bar{v}_2(\cdot)$ is really an optimal control of the leader.
\end{theorem}

\begin{proof}
For any $v_2(\cdot)\in\mathcal{U}_{ad}^l[0,T]$, we rewrite the cost functional \eqref{LQ cost functional2} as
\begin{equation*}
J_2(v_2(\cdot))-J_2(\bar{v}_2(\cdot))=\textrm{I}+\textrm{II},
\end{equation*}
where
\begin{equation}\label{i}
\begin{aligned}
\textrm{I}&=\frac{1}{2}\mathbb{E}\bigg\{\int_0^T\big[\bar{L}(t)(x^{v_2}(t)-x^{\bar{v}_2}(t))^2+\bar{R}(t)(v_2(t)-\bar{v}_2(t))^2\big]dt\\
&\qquad\qquad+\bar{M}(x^{v_2}(T)-x^{\bar{v}_2}(T))^2\bigg\},
\end{aligned}
\end{equation}
\begin{equation}\label{ii}
\begin{aligned}
\textrm{II}&=\mathbb{E}\bigg\{\int_0^T\big[\big(\bar{l}(t)+\bar{L}(t)x^{\bar{v}_2}(t)\big)(x^{v_2}(t)-x^{\bar{v}_2}(t))+\big(\bar{R}(t)\bar{v}_2+\bar{r}(t)\big)(v_2(t)-\bar{v}_2(t))\big]dt\\
&\qquad\qquad+\big(\bar{M}x^{\bar{v}_2}(T)+\bar{m}\big)(x^{v_2}(T)-x^{\bar{v}_2}(T))\bigg\}.
\end{aligned}
\end{equation}
Applying It\^{o}'s formula to $\lambda(\cdot)(x^{v_2}(\cdot)-x^{\bar{v}_2}(\cdot))-K(\cdot)(\hat{\theta}(\cdot)-\hat{\bar{\theta}}(\cdot))$, we have
\begin{equation}\label{terminal}
\begin{aligned}
&\mathbb{E}\big[\big(\bar{M}x^{\bar{v}_2}(T)+\bar{m}\big)\big(x^{v_2}(T)-x^{\bar{v}_2}(T)\big)\big]\\
&=\mathbb{E}\int_0^T\Big\{\lambda\big[A(x^{v_2}-x^{\bar{v}_2})-B_1^2\Pi R^{-1}(\hat{x}^{v_2}-\hat{x}^{\bar{v}_2})-B_1^2R^{-1}(\hat{\theta}-\hat{\bar{\theta}})+B_2(v_2-\bar{v}_2)\big]\\
&\qquad\qquad-\big(\bar{L}x^{\bar{v}_2}+\bar{l}+A\lambda-B_1^2\Pi R^{-1}\hat{\lambda}\big)(x^{v_2}-x^{\bar{v}_2})+K\big[(A-B_1^2\Pi R^{-1})(\hat{\theta}-\hat{\bar{\theta}})\\
&\qquad\qquad+B_2\Pi(v_2-\bar{v}_2)\big]-\big[-B_1^2R^{-1}\lambda+(A-B_1^2\Pi R^{-1})K\big](\hat{\theta}-\hat{\bar{\theta}})\Big\}dt.
\end{aligned}
\end{equation}
Substituting \eqref{terminal} into \eqref{ii}, we can get
\begin{equation*}
\begin{aligned}
\textrm{II}&=\mathbb{E}\int_0^T\Big\{(\bar{r}+\bar{R}\bar{v}_2)(v_2-\bar{v}_2)+\lambda\big[-B_1^2\Pi R^{-1}(\hat{x}^{v_2}-\hat{x}^{\bar{v}_2})-B_1^2R^{-1}(\hat{\theta}-\hat{\bar{\theta}})+B_2(v_2-\bar{v}_2)\big]\\
&\qquad+B_1^2\Pi R^{-1}\hat{\lambda}(x^{v_2}-x^{\bar{v}_2})+K\big[(A-B_1^2\Pi R^{-1})(\hat{\theta}-\hat{\bar{\theta}})+B_2\Pi(v_2-\bar{v}_2)\big]\\
&\qquad-\big[-B_1^2R^{-1}\lambda+(A-B_1^2\Pi R^{-1})K\big](\hat{\theta}-\hat{\bar{\theta}})\Big\}dt\\
&=\mathbb{E}\int_0^T\mathbb{E}\Big[(\bar{r}+\bar{R}\bar{v}_2)(v_2-\bar{v}_2)+\lambda\big[-B_1^2\Pi R^{-1}(\hat{x}^{v_2}-\hat{x}^{\bar{v}_2})-B_1^2R^{-1}(\hat{\theta}-\hat{\bar{\theta}})+B_2(v_2-\bar{v}_2)\big]\\
&\qquad+B_1^2\Pi R^{-1}\hat{\lambda}(x^{v_2}-x^{\bar{v}_2})+K\big[(A-B_1^2\Pi R^{-1})(\hat{\theta}-\hat{\bar{\theta}})+B_2\Pi(v_2-\bar{v}_2)\big]\\
&\qquad-\big[-B_1^2R^{-1}\lambda+(A-B_1^2\Pi R^{-1})K\big](\hat{\theta}-\hat{\bar{\theta}})\Big|\mathcal{F}_t^W\Big]dt\\
\end{aligned}
\end{equation*}
\begin{equation*}
\begin{aligned}
&=\mathbb{E}\int_0^T\mathbb{E}\big[\bar{r}+\bar{R}\bar{v}_2+B_2\lambda+B_2\Pi K\big|\mathcal{F}_t^W\big](v_2-\bar{v}_2)dt\\
&\qquad+\mathbb{E}\int_0^T\mathbb{E}\Big[\mathbb{E}\big[\lambda\big\{-B_1^2\Pi R^{-1}(\hat{x}^{v_2}-\hat{x}^{\bar{v}_2})-B_1^2R^{-1}(\hat{\theta}-\hat{\bar{\theta}})\big\}+B_1^2\Pi R^{-1}\hat{\lambda}(x^{v_2}-x^{\bar{v}_2})\\
&\qquad\quad+(A-B_1^2\Pi R^{-1})(\hat{\theta}-\hat{\bar{\theta}})K-\big\{-B_1^2R^{-1}\lambda+(A-B_1^2\Pi R^{-1})K\big\}(\hat{\theta}-\hat{\bar{\theta}})\big|\mathcal{F}_t^{Y_1^0}\big]\Big|\mathcal{F}_t^W\Big]dt.
\end{aligned}
\end{equation*}
By Lemmas 3.1, Remark 2.1 and \eqref{maximum condition2}, it is obvious that $\textrm{II}=0$. Since $\textrm{I}\geq0$, we complete the proof.
\end{proof}

\subsubsection{Optimal State Filtering and Feedback Form}

We notice that the leader face the LQ partially observed optimal control problem of CMFFBSDEs. Therefore, to obtain the optimal state filtering feedback representation, we put \eqref{leaderstate}, \eqref{lambda} and \eqref{K} together:
\begin{equation}\label{together}
\left\{
\begin{aligned}
      dx^{\bar{v}_2}(t)&=\big[Ax^{\bar{v}_2}-B_1^2\Pi R^{-1}\hat{x}^{\bar{v}_2}-B_1^2R^{-1}\hat{\bar{\theta}}-B_1R^{-1}r+B_2\bar{v}_2+\alpha\big]dt\\
                       &\quad+cdW(t)+\bar{c}d\bar{W}(t),\\
                  dK(t)&=\big[-B_1^2R^{-1}\lambda+(A-B_1^2\Pi R^{-1})K\big]dt,\\
           -d\lambda(t)&=\big(\bar{L}x^{\bar{v}_2}+\bar{l}+A\lambda-B_1^2\Pi R^{-1}\hat{\lambda}\big)dt-\eta_1dW(t)-\eta_2d\bar{W}(t)-\eta_3d\widetilde{W}(t),\\
-d\hat{\bar{\theta}}(t)&=\big[(A-B_1^2\Pi R^{-1})\hat{\bar{\theta}}+B_2\Pi\bar{v}_2-B_1\Pi R^{-1}r+\alpha\Pi+l\big]dt-\hat{\bar{\Gamma}}d\widetilde{W}(t),\ t\in[0,T],\\
       x^{\bar{v}_2}(0)&=x_0,\ K(0)=0,\ \lambda(T)=\bar{M}x^{\bar{v}_2}(T)+\bar{m},\ \hat{\bar{\theta}}(T)=m,
\end{aligned}
\right.
\end{equation}
and let
\begin{equation}\label{notation}
X=
\begin{pmatrix}
x^{\bar{v}_2}\\
K
\end{pmatrix},\ \
Y=
\begin{pmatrix}
\lambda\\
\hat{\bar{\theta}}
\end{pmatrix},\ \
Z_1=
\begin{pmatrix}
\eta_1\\
0
\end{pmatrix},\ \
Z_2=
\begin{pmatrix}
\eta_2\\
0
\end{pmatrix},\ \
Z_3=
\begin{pmatrix}
\eta_3\\
\hat{\bar{\Gamma}}
\end{pmatrix},\ \
\end{equation}
and
\begin{equation*}
\left\{
\begin{aligned}
&\mathcal{A}_1=
\begin{pmatrix}
A&0\\
0&A-B_1^2\Pi R^{-1}
\end{pmatrix},\ \
\mathcal{A}_2=
\begin{pmatrix}
-B_1^2\Pi R^{-1}&0\\
0&0
\end{pmatrix},\ \
\mathcal{A}_3=
\begin{pmatrix}
\bar{L}&0\\
0&0
\end{pmatrix},\\
&\mathcal{B}_1=
\begin{pmatrix}
0&-B_1^2R^{-1}\\
-B_1^2R^{-1}&0
\end{pmatrix},\ \
\mathcal{C}_1=
\begin{pmatrix}
-B_1R^{-1}r+\alpha\\
0
\end{pmatrix},\\
&\mathcal{C}_2=
\begin{pmatrix}
\bar{l}\\
-B_1\Pi R^{-1}r+\alpha\Pi+l
\end{pmatrix},\ \
\mathcal{D}_1=
\begin{pmatrix}
B_2\\
0
\end{pmatrix},\ \
\mathcal{D}_2=
\begin{pmatrix}
0\\
B_2\Pi
\end{pmatrix},\ \
\Sigma_1=
\begin{pmatrix}
c\\
0
\end{pmatrix},\\
&\Sigma_2=
\begin{pmatrix}
\bar{c}\\
0
\end{pmatrix},\ \
X_0=
\begin{pmatrix}
x_0\\
0
\end{pmatrix},\ \
\mathcal{M}=
\begin{pmatrix}
\bar{M}&0\\
0&0
\end{pmatrix},\ \
\mathcal{M}_T=
\begin{pmatrix}
\bar{m}\\
m
\end{pmatrix}.
\end{aligned}
\right.
\end{equation*}
Then \eqref{together} can be rewritten as:
\begin{equation}\label{together2}
\left\{
\begin{aligned}
 dX(t)&=\big[\mathcal{A}_1X+\mathcal{A}_2\hat{X}+\mathcal{B}_1Y+\mathcal{D}_1\bar{v}_2+\mathcal{C}_1\big]dt+\Sigma_1dW(t)+\Sigma_2d\bar{W}(t),\\
-dY(t)&=\big[\mathcal{A}_3X+\mathcal{A}_1Y+\mathcal{A}_2\hat{Y}+\mathcal{D}_2\bar{v}_2+\mathcal{C}_2\big]dt-Z_1dW(t)-Z_2d\bar{W}(t)\\
      &\quad-Z_3d\widetilde{W}(t),\ t\in[0,T],\\
  X(0)&=X_0,\ Y(T)=\mathcal{M}X(T)+\mathcal{M}_T.
\end{aligned}
\right.
\end{equation}
Meanwhile, the optimality condition of $\bar{v}_2(\cdot)$ \eqref{maximum condition2} becomes
\begin{equation}\label{maximum condition3}
\bar{R}\bar{v}_2+\bar{r}+\mathcal{D}_1^\top\check{Y}+\mathcal{D}_2^\top\check{X}=0,\ a.e.t\in[0,T],\ a.s.
\end{equation}

If we assume the following\\
{\bf(A4)}
$\bar{R}^{-1}(\cdot)$ is a uniformly bounded, deterministic function.\\
Then substituting \eqref{maximum condition3} into \eqref{together2}, we have
\begin{equation}\label{together3}
\left\{
\begin{aligned}
 dX(t)&=\big[\mathcal{A}_1X+\mathcal{A}_2\hat{X}+\mathcal{B}_1Y-\mathcal{D}_1\bar{R}^{-1}\mathcal{D}_1^\top\check{Y}-\mathcal{D}_1\bar{R}^{-1}\mathcal{D}_2^\top\check{X}-\mathcal{D}_1\bar{R}^{-1}\bar{r}+\mathcal{C}_1\big]dt\\
      &\quad+\Sigma_1dW(t)+\Sigma_2d\bar{W}(t),\\
-dY(t)&=\big[\mathcal{A}_3X+\mathcal{A}_1Y+\mathcal{A}_2\hat{Y}-\mathcal{D}_2\bar{R}^{-1}\mathcal{D}_1^\top\check{Y}-\mathcal{D}_2\bar{R}^{-1}\mathcal{D}_2^\top\check{X}-\mathcal{D}_2\bar{R}^{-1}\bar{r}+\mathcal{C}_2\big]dt\\
      &\quad-Z_1dW(t)-Z_2d\bar{W}(t)-Z_3d\widetilde{W}(t),\ t\in[0,T],\\
  X(0)&=X_0,\ Y(T)=\mathcal{M}X(T)+\mathcal{M}_T.
\end{aligned}
\right.
\end{equation}

We note that equation \eqref{together3}, different from the classical mean-field FBSDE, is a non-standard coupled FBSDE with two kinds of conditional mean-field forms, which is new and more complicated. And the unique solvability of \eqref{together3} is equivalent to that of \eqref{together}. Before giving the solvability of \eqref{together3}, we first discuss a special FBSDE with filtering terms and prove the existence and uniqueness of its solutions.

We consider a fully coupled FBSDE with filtering terms as follows:
\begin{equation}\label{FBSDEF}
\left\{
\begin{aligned}
 dX(t)&=b(t,X,Y,Z_1,Z_2,\hat{X},\hat{Y},\hat{Z}_1,\hat{Z}_2)dt+\sigma(t,X,Y,Z_1,Z_2,\hat{X},\hat{Y},\hat{Z}_1,\hat{Z}_2)dW(t)\\
      &\quad+\bar{\sigma}(t,X,Y,Z_1,Z_2,\hat{X},\hat{Y},\hat{Z}_1,\hat{Z}_2)d\bar{W}(t),\\
-dY(t)&=g(t,X,Y,Z_1,Z_2,\hat{X},\hat{Y},\hat{Z}_1,\hat{Z}_2)dt-Z_1dW(t)-Z_2d\bar{W}(t),\quad t\in[0,T],\\
  X(0)&=X_0,\ Y(T)=\Phi(X(T)),
\end{aligned}
\right.
\end{equation}
where $\hat{X}=\mathbb{E}[X|\mathcal{F}_t^{W}],\hat{Y}=\mathbb{E}[Y|\mathcal{F}_t^{W}],\hat{Z}_1=\mathbb{E}[Z_1|\mathcal{F}_t^{W}]$ and $\hat{Z}_2=\mathbb{E}[Z_2|\mathcal{F}_t^{W}]$. $b:[0,T]\times \mathbb{R}^2\times \mathbb{R}^2\times \mathbb{R}^2\times \mathbb{R}^2\times \mathbb{R}^2\times \mathbb{R}^2\times \mathbb{R}^2\times \mathbb{R}^2\rightarrow \mathbb{R}^2,\sigma:[0,T]\times \mathbb{R}^2\times \mathbb{R}^2\times \mathbb{R}^2\times \mathbb{R}^2\times \mathbb{R}^2\times \mathbb{R}^2\times \mathbb{R}^2\times \mathbb{R}^2\rightarrow \mathbb{R}^2,\bar{\sigma}:[0,T]\times \mathbb{R}^2\times \mathbb{R}^2\times \mathbb{R}^2\times \mathbb{R}^2\times \mathbb{R}^2\times \mathbb{R}^2\times \mathbb{R}^2\times \mathbb{R}^2\rightarrow \mathbb{R}^2$ and $\Phi:\mathbb{R}^2\rightarrow \mathbb{R}^2$ are given functions.
We are given a $m\times n$ full rank matrix $G$ and use the notations:
\begin{equation}
u=\begin{pmatrix}
X\\
Y\\
Z_1\\
Z_2
\end{pmatrix},\ \hat{u}=\begin{pmatrix}
\hat{X}\\
\hat{Y}\\
\hat{Z}_1\\
\hat{Z}_2
\end{pmatrix},\ A(t,u,\hat{u})=\begin{pmatrix}
-G^\top g\\
Gb\\
G\sigma\\
G\bar{\sigma}
\end{pmatrix}(t,u,\hat{u}),
\end{equation}
where we have $m=n=2$ and $G=\mathbb{I}_2$. We assume that

(i) $A(t,u,\hat{u})$ is uniformly Lipschitz with respect to $u,\hat{u}$, that is, there exists a constant $C$ such that
$|A(t,u,\hat{u})-A(t,u',\hat{u}')|\leq C(|u-u'|+|\hat{u}-\hat{u}'|)$,

(ii) for each $u,\hat{u}$, $A(\cdot,u,\hat{u})$ is in $L^2_{\mathcal{F}}(0,T;\mathbb{R}^2\times \mathbb{R}^2\times \mathbb{R}^2\times \mathbb{R}^2\times \mathbb{R}^2\times \mathbb{R}^2\times \mathbb{R}^2\times \mathbb{R}^2)$,

(iii) $\Phi(x)$ is uniformly Lipschitz with respect to $x\in \mathbb{R}^2$, that is, there exists a constant $C_1$ such that $\Phi(x)-\Phi(x')\leq C_1|x-x'|$,

(iv) for each $x$, $\Phi(x)$ is in $L^2_{\mathcal{F}_T}(\Omega,\mathbb{R}^2)$.

The following monotonic conditions are our main assumptions.
For any $\hat{u}$, we have
\begin{equation}
\begin{aligned}
&\langle A(t,u,\hat{u})-A(t,u',\hat{u}),u-u'\rangle\leq-\beta|u-u'|^2,\\
&\langle \Phi(X)-\Phi(X'),X-X'\rangle\geq\mu_1|\tilde{X}|^2,\\
&\forall\ u=(X,Y,Z_1,Z_2),\ u'=(X',Y',Z'_1,Z'_2),\\
&\quad \tilde{X}=X-X',\ \tilde{Y}=Y-Y',\ \tilde{Z}_1=Z_1-Z'_1,\ \tilde{Z}_2=Z_2-Z'_2,
\end{aligned}
\end{equation}
where $\mu_1,\beta$ are given positive constants.

For any $(x(\cdot),y(\cdot),z_1(\cdot),z_2(\cdot))\in L^2_{\mathcal{F}}(0,T;\mathbb{R}^{2+2+2+2})$, the following fully coupled FBSDE
\begin{equation}
\left\{
\begin{aligned}
 dX(t)&=b(t,X,Y,Z_1,Z_2,\hat{x},\hat{y},\hat{z}_1,\hat{z}_2)dt+\sigma(t,X,Y,Z_1,Z_2,\hat{x},\hat{y},\hat{z}_1,\hat{z}_2)dW(t)\\
      &\quad+\bar{\sigma}(t,X,Y,Z_1,Z_2,\hat{x},\hat{y},\hat{z}_1,\hat{z}_2)d\bar{W}(t),\\
-dY(t)&=g(t,X,Y,Z_1,Z_2,\hat{x},\hat{y},\hat{z}_1,\hat{z}_2)dt-Z_1dW(t)-Z_2d\bar{W}(t),\quad t\in[0,T],\\
  X(0)&=X_0,\ Y(T)=\Phi(X(T)),
\end{aligned}
\right.
\end{equation}
admits unique solution $(X(\cdot),Y(\cdot),Z_1(\cdot),Z_2(\cdot))\in L^2_{\mathcal{F}}(0,T;\mathbb{R}^{2+2+2+2})$(see \cite{PW99}). We then introduce a mapping and denote $U=(X,Y,Z_1,Z_2)=I(u),U'=(X',Y',Z'_1,Z'_2)=I(u'),(\tilde{x},\tilde{y},
\tilde{z}_1,\tilde{z}_2)=(x-x',y-y',z_1-z'_2,z_2-z'_2)$ and $(\tilde{X},\tilde{Y},
\tilde{Z}_1,\tilde{Z}_2)=(X-X',Y-Y',Z_1-Z'_2,Z_2-Z'_2)$.

Applying It\^{o}'s formula to $\langle\tilde{X}(\cdot),\tilde{Y}(\cdot)\rangle$, we have
\begin{equation}
\begin{aligned}
\mathbb{E}&\big[\langle \tilde{X}_T, \Phi(X_T)-\Phi(X'_T)\rangle\big]=\mathbb{E}\int_0^T\big\langle A(t,U,\hat{u})-A(t,U',\hat{u}'),\tilde{U}\big\rangle dt\\
&=\mathbb{E}\int_0^T\big\langle A(t,U,\hat{u})-A(t,U,\hat{u}')+A(t,U,\hat{u}')-A(t,U',\hat{u}'),\tilde{U}\big\rangle dt\\
&\leq\mathbb{E}\int_0^T\big\langle C|\hat{u}-\hat{u}'|,\tilde{U}\rangle+\langle A(t,U,\hat{u}')-A(t,U',\hat{u}'),\tilde{U}\big\rangle dt\\
&\leq\mathbb{E}\int_0^T\bigg[\frac{\delta}{2}|\tilde{U}|^2+\frac{C^2}{2\delta}|\hat{u}-\hat{u}'|^2-\beta|\tilde{U}|^2\bigg]dt.
\end{aligned}
\end{equation}
By the monotonic assumption, we have
\begin{equation}
\mathbb{E}[\mu_1|\tilde{X}_T|^2]\leq\mathbb{E}\int_0^T\bigg[\frac{\delta}{2}|\tilde{U}|^2+\frac{C^2}{2\delta}|\hat{u}-\hat{u}'|^2-\beta|\tilde{U}|^2\bigg]dt,
\end{equation}
and
\begin{equation}
\begin{aligned}
\mathbb{E}\bigg[(\beta-\frac{\delta}{2})\mathbb{E}\bigg[\int_0^T|\tilde{U}|^2dt|\mathcal{F}_t\bigg]\bigg]&\leq\mathbb{E}\bigg[\frac{C^2}{2\delta}\mathbb{E}\bigg[\int_0^T|\hat{\tilde{u}}|^2dt|\mathcal{F}_t\bigg]\bigg]\\
&\leq\mathbb{E}\bigg[\frac{C^2}{2\delta}\mathbb{E}\bigg[\int_0^T|\tilde{u}|^2dt|\mathcal{F}_t\bigg]\bigg].
\end{aligned}
\end{equation}
If $\delta\in[\beta-\sqrt{\beta^2-2C^2},\beta+\sqrt{\beta^2-2C^2}]\cap(0,2\beta)$ and $\beta\geq\sqrt{2}C$, then we have
\begin{equation}
\mathbb{E}\int_0^T|\tilde{U}|^2dt\leq\frac{1}{2}\mathbb{E}\int_0^T|\tilde{u}|^2dt.
\end{equation}
Therefore, we have the mapping $I$ is a contraction on $L^2_{\mathcal{F}}(0,T;\mathbb{R}^{2+2+2+2})$. Due to the fixed point theorem, there is a unique fixed point $(X(\cdot),Y(\cdot),Z_1(\cdot),Z_2(\cdot))\in L^2_{\mathcal{F}}(0,T;\mathbb{R}^{2+2+2+2})$ such that $I(X,Y,Z_1,Z_2)=(X,Y,Z_1,Z_2)$. Then we prove $(X(\cdot),Y(\cdot),Z_1(\cdot),Z_2(\cdot))$ satisfies the fully-coupled FBSDE with filtering \eqref{FBSDEF}.

\begin{Remark}
Noting \eqref{together3} and \eqref{FBSDEF}, we can regard \eqref{FBSDEF} as a special case of \eqref{together3} from the perspective of filtering. \eqref{FBSDEF}, whose solvability has been guaranteed, can be derived by considering the following case: $\mathcal{F}_t^{Y_1}=\mathcal{F}_t^{Y_2}=\mathcal{F}_t^W$, that is, both the follower and the leader can only observe the same partial information, where the innovation process $\widetilde{W}$ disappears, and the only noise sources come from two independent $\mathcal{F}_t$-Brownian motions $W,\bar{W}$. However, the new Brownian motion $\widetilde{W}$ cause the intrinsic difficult for obtaining the existence and uniqueness of the solution to \eqref{together3} with the similar proof in the above. So we give an another way to treat the problem in the following.
\end{Remark}

\begin{Remark}
As we have analyzed, the equation \eqref{together3} contains not only $(X(\cdot),Y(\cdot))$ itself but also its two kinds of filtering form, $(\hat{X}(\cdot),\hat{Y}(\cdot))$ and $(\check{X}(\cdot),\check{Y}(\cdot))$, therefore, we fail to obtain the existence and uniqueness of the solution to it directly because of its complex coupling. Fortunately, we can utilize an indirect method to obtain its unique solvability of \eqref{together3}. We can concisely formulate the new combined idea by considering the following two steps: decoupling \eqref{together3} for its existence and then prove its uniqueness.
\end{Remark}

To obtain the existence of solution to \eqref{together3}, we decouple it first by setting
\begin{equation}\label{Y3X}
Y(t)=\Pi_1X(t)+\Pi_2\hat{X}(t)+\Pi_3\check{X}(t)+\Phi(t),\ t\in[0,T],
\end{equation}
where $\Phi(\cdot)$ is an $\mathcal{F}_t^{Y_1}$-adapted semi-martingale to be determined later, satisfying $\Phi(T)=\mathcal{M}_T$.

\begin{Remark}
It is worthy to notice that the setting \eqref{Y3X} is different from the existing literatures. Indeed, $Y(\cdot)$ is not only represented by $X(\cdot),\hat{X}(\cdot)$ and $\check{X}(\cdot)$ linearly, but also contains $\Phi(\cdot)$ which is an $\mathcal{F}_t^{Y_1}$-adapted process instead of a deterministic function. Actually, inspired by Li et al. \cite{LFCM19}, $\Phi(\cdot)$ is an $\mathcal{F}_t^{Y_1}$-adapted semi-martingale. Compared to Yong \cite{Yong02}, we consider a partially observed case with deterministic coefficients, which results in the appearance of $\hat{X}(\cdot),\check{X}(\cdot)$ and $\Phi(\cdot)$. Compared to Shi et al. \cite{SWX20}, we consider the case that $\mathcal{F}_t^{Y_2}\subset\mathcal{F}_t^{Y_1}$ and no overlapping information exists. Therefore, different from the relation (100) in \cite{SWX20}, the fourth term $\check{\hat{X}}$ disappears naturally in our case. And the information filtrations in \cite{SWX20} are generated by independent standard Brownian motions, however, our filtrations are generated by observation processes which introduce the innovation process $\widetilde{W}(\cdot)$ and they build a strong relationship with the semi-martingale $\Phi(\cdot)$.
\end{Remark}

For obtaining the filtering equations of $\hat{X}(\cdot)$ and $\check{X}(\cdot)$, it is necessary to study the information filtrations $\mathcal{F}_t^{Y_1},\mathcal{F}_t^{Y_2}$ deeply and the corresponding conditional expectations $\mathbb{E}[\cdot|\mathcal{F}_t^{Y_1}],\mathbb{E}[\cdot|\mathcal{F}_t^{Y_2}]$, respectively. Meanwhile, the FBSDE \eqref{together3} is driven by a innovation process $\widetilde{W}(\cdot)$. We set
$$\mathcal{I}_t=\sigma\{\widetilde{W}(s)|0\leq s\leq t\}$$
and give the following lemma.
\begin{lemma}
For $\bar{v}_1(\cdot)\in\mathcal{U}_{ad}^f[0,T]$, $\bar{v}_2(\cdot)\in\mathcal{U}_{ad}^l[0,T]$ and $\mathcal{F}_t^{Y_2^{\bar{v}_2}}\subseteq(\mathcal{F}_t^{Y_1^{\bar{v}_1,\bar{v}_2}}\land\mathcal{F}_t^{Y_1^0})$, then $\mathcal{I}_t=\mathcal{F}_t^{Y_1^{\bar{v}_1,\bar{v}_2}}$.
\end{lemma}
\begin{proof}
From \eqref{x0} and \eqref{Y0}, $x^0(\cdot)$ and $Y_1^0(\cdot)$ are independent of the control variables, then
\begin{equation}\label{filtering x0}
\begin{aligned}
d\hat{x}^0&=A\hat{x}^0dt+(c+f_1P)(dY_1^0-f_1\hat{x}^0dt),\\
&=A\hat{x}^0dt+(c+f_1P)d\widetilde{W}.
\end{aligned}
\end{equation}
Hence, we derive that $\hat{x}^0(\cdot)$ is clearly adapted to the filtration $\mathcal{I}_t$, and noting \eqref{filtering x0} and \eqref{widetildeW2}, $\mathcal{F}_t^{Y_1^0}\subset\mathcal{I}_t$. Then for $\bar{v}_1(\cdot)\in\mathcal{U}_{ad}^f[0,T],\bar{v}_2(\cdot)\in\mathcal{U}_{ad}^l[0,T]$ and $\mathcal{F}_t^{Y_2^{\bar{v}_2}}\subseteq(\mathcal{F}_t^{Y_1^{\bar{v}_1,\bar{v}_2}}\land\mathcal{F}_t^{Y_1^0})$, we have $\mathcal{F}_t^{Y_1^{\bar{v}_1,\bar{v}_2}}=\mathcal{F}_t^{Y_1^0}\subset\mathcal{I}_t$ by Lemma 3.1. Moreover, the reverse is obviously true. Therefore, we get
$\mathcal{I}_t=\mathcal{F}_t^{Y_1^{\bar{v}_1,\bar{v}_2}}$. The proof is complete.
\end{proof}
By the above lemma, we have the relation that $\mathbb{E}[\cdot|\mathcal{F}_t^{Y_1}]=\mathbb{E}[\cdot|\mathcal{I}_t]$, where we omit the superscript $(\bar{v}_1,\bar{v}_2)$ of $Y_1$ for simplicity.

For obtaining the filtering equation of $\hat{X}(\cdot)$ by Theorem 8.1 in \cite{LS77}, firstly, noting \eqref{notation}, we rewrite the observation equation \eqref{observation f} as:
\begin{equation}\label{observation f2}
\left\{
\begin{aligned}
dY_1(t)&=\big[F^\top(t)X(t)+g(t)\big]dt+dW(t),\ t\in[0,T],\\
 Y_1(0)&=0,
\end{aligned}
\right.
\end{equation}
where
\begin{equation}\label{F}
F(t)=
\begin{pmatrix}
f_1(t)\\
0
\end{pmatrix}.
\end{equation}

\begin{lemma}
Let {\bf (A1)-(A4)} hold, Then $\hat{X}(\cdot)$ is the solution to the following equation
\begin{equation}\label{hatX}
\left\{
\begin{aligned}
d\hat{X}(t)&=\big[(\mathcal{A}_1+\mathcal{A}_2)\hat{X}+\mathcal{B}_1\hat{Y}-\mathcal{D}_1\bar{R}^{-1}\mathcal{D}_1^{\top}\check{Y}-\mathcal{D}_1\bar{R}^{-1}\mathcal{D}_2^{\top}\check{X}
            -\mathcal{D}_1\bar{R}^{-1}\bar{r}+\mathcal{C}_1\big]dt\\
           &\quad+(\Sigma_1+\mathcal{P}F)d\widetilde{W}(t),\ t\in[0,T],\\
 \hat{X}(0)&=X_0,
\end{aligned}
\right.
\end{equation}
where the mean square error $\mathcal{P}(\cdot)\in R^{2\times2}$ of the filtering estimate $\hat{X}(\cdot)$ satisfy the ordinary differential equation (ODE for short):
\begin{equation}\label{mathcalP}
\left\{
\begin{aligned}
&\dot{\mathcal{P}}-(\mathcal{A}_1+\mathcal{B}_1\Pi_1)\mathcal{P}-\mathcal{P}(\mathcal{A}_1+\Pi_1\mathcal{B}_1)+\Sigma_1F^\top\mathcal{P}\\
&\quad+\mathcal{P}F\Sigma_1^\top+\mathcal{P}FF^{\top}\mathcal{P}-\Sigma_2\Sigma_2^{\top}=0,\ t\in[0,T],\\
&\mathcal{P}(0)=\textbf{0},
\end{aligned}
\right.
\end{equation}
where \textbf{0} is the $R^{2\times2}$ matrix with all the components being 0.
\end{lemma}
\begin{proof}
The equation \eqref{hatX} can be obtained by the Theorem 8.1 in \cite{LS77}, and we omit the details. We prove \eqref{mathcalP} in the following. Denote $\delta(\cdot):=X(\cdot)-\hat{X}(\cdot)$, and apply It\^{o}'s formula, we get
\begin{equation*}
\left\{
\begin{aligned}
d\delta(t)&=\big[\mathcal{A}_1\delta(t)+\mathcal{B}_1\Pi_1\delta(t)-(\Sigma_1+\mathcal{P}F)F^{\top}\delta(t)\big]dt\\
          &\quad-\mathcal{P}FdW(t)+\Sigma_2d\bar{W}(t),\ t\in[0,T],\\
 \delta(0)&=0.
\end{aligned}
\right.
\end{equation*}
Then applying again It\^{o}'s formula to $\delta(\cdot)\delta^\top(\cdot)$, we have
\begin{equation*}
\left\{
\begin{aligned}
d[\delta(t)\delta^\top(t)]&=\big[(\mathcal{A}_1+\mathcal{B}_1\Pi_1)\delta\delta^\top+\delta\delta^\top(\mathcal{A}_1+\Pi_1\mathcal{B}_1)-(\Sigma_1+\mathcal{P}F)F^\top\delta\delta^\top\\
&\quad-\delta\delta^\top F(\Sigma_1+\mathcal{P}F)^\top+\mathcal{P}FF^\top\mathcal{P}+\Sigma_2\Sigma_2^\top\big]dt-\big\langle\mathcal{P}FdW(t),\delta(t)\big\rangle\\
&\quad-\big\langle\delta(t),\mathcal{P}FdW(t)\big\rangle+\left\langle\Sigma_2d\bar{W}(t),\delta(t)\right\rangle+\left\langle\delta(t),\Sigma_2d\bar{W}(t)\right\rangle\\
\delta(0)\delta^{\top}(0)&=0.
\end{aligned}
\right.
\end{equation*}
Taking expectation on both sides, we derive \eqref{mathcalP}. The proof is complete.
\end{proof}

From Lemma 3.5, we have (the superscript $v_2$ of $Y_2$ is omitted for simplicity)
\begin{equation}\label{equality}
\mathbb{E}[X|\mathcal{F}_t^{Y_2}]=\mathbb{E}[X|\mathcal{F}_t^{Y_2^0}]=\mathbb{E}[X|\mathcal{F}_t^{W}].
\end{equation}
Applying the Lemma 5.4 in Xiong \cite{Xiong08}, we get the equation of $\check{X}(\cdot)$ as follows
\begin{equation}\label{checkX}
\left\{
\begin{aligned}
d\check{X}(t)&=\big[(\mathcal{A}_1+\mathcal{A}_2-\mathcal{D}_1\bar{R}^{-1}\mathcal{D}_2^{\top})\check{X}+(\mathcal{B}_1-\mathcal{D}_1\bar{R}^{-1}\mathcal{D}_1^{\top})\check{Y}
              -\mathcal{D}_1\bar{R}^{-1}\bar{r}+\mathcal{C}_1\big]dt\\
             &\quad+\Sigma_1d{W}(t),\ t\in[0,T],\\
 \check{X}(0)&=X_0.
\end{aligned}
\right.
\end{equation}
Applying It\^{o}'s formula to \eqref{Y3X}, and comparing it with \eqref{together3}, it yields
\begin{equation}\label{YY}
\begin{aligned}
dY(t)&=\big\{\dot{\Pi}_1X+\dot{\Pi}_2\hat{X}+\dot{\Pi}_3\check{X}+\Pi_1\mathcal{A}_1X+\Pi_1\mathcal{A}_2\hat{X}+\Pi_1\mathcal{B}_1Y-\Pi_1\mathcal{D}_1\bar{R}^{-1}\mathcal{D}_1^\top\check{Y}\\
&\qquad-\Pi_1\mathcal{D}_1\bar{R}^{-1}\mathcal{D}_2^\top\check{X}-\Pi_1\mathcal{D}_1\bar{R}^{-1}\bar{r}+\Pi_1\mathcal{C}_1+\Pi_2(\mathcal{A}_1+\mathcal{A}_2)\hat{X}+\Pi_2\mathcal{B}_1\hat{Y}\\
&\qquad-\Pi_2\mathcal{D}_1\bar{R}^{-1}\mathcal{D}_1^\top\check{Y}-\Pi_2\mathcal{D}_1\bar{R}^{-1}\mathcal{D}_2^{\top}\check{X}-\Pi_2\mathcal{D}_1\bar{R}^{-1}\bar{r}+\Pi_2\mathcal{C}_1+\Pi_3(\mathcal{A}_1\\
&\qquad+\mathcal{A}_2-\mathcal{D}_1\bar{R}^{-1}\mathcal{D}_2^\top)\check{X}+\Pi_3(\mathcal{B}_1-\mathcal{D}_1\bar{R}^{-1}\mathcal{D}_1^\top)\check{Y}-\Pi_3\mathcal{D}_1\bar{R}^{-1}\bar{r}+\Pi_3\mathcal{C}_1\big\}dt\\
&\qquad+d\Phi(t)+(\Pi_1+\Pi_3)\Sigma_1dW(t)+\Pi_1\Sigma_2d\bar{W}(t)+\Pi_2(\Sigma_1+\mathcal{P}F)d\widetilde{W}(t),\\
&=-\big[\mathcal{A}_3X+\mathcal{A}_1Y+\mathcal{A}_2\hat{Y}-\mathcal{D}_2\bar{R}^{-1}\mathcal{D}_1^\top\check{Y}-\mathcal{D}_2\bar{R}^{-1}\mathcal{D}_2^\top\check{X}-\mathcal{D}_2\bar{R}^{-1}\bar{r}+\mathcal{C}_2\big]dt\\
&\quad+Z_1dW(t)+Z_2d\bar{W}(t)+Z_3d\widetilde{W}(t).
\end{aligned}
\end{equation}
Noting \eqref{Y3X}, we have
\begin{equation}\label{zero-zero}
\begin{aligned}
&\Big\{(\dot{\Pi}_1+\Pi_1\mathcal{A}_1+\Pi_1\mathcal{B}_1\Pi_1+\mathcal{A}_3+\mathcal{A}_1\Pi_1)X+\big[\dot{\Pi}_2+\Pi_1\mathcal{A}_2+\Pi_1\mathcal{B}_1\Pi_2+\Pi_2(\mathcal{A}_1+\mathcal{A}_2)\\
&\quad+\Pi_2\mathcal{B}_1\Pi_1+\Pi_2\mathcal{B}_1\Pi_2+(\mathcal{A}_1+\mathcal{A}_2)\Pi_2+\mathcal{A}_2\Pi_1\big]\hat{X}+\big\{\dot{\Pi}_3+[(\Pi_1+\Pi_2)\mathcal{B}_1+\mathcal{A}_1\\
&\quad+\mathcal{A}_2]\Pi_3-(\Pi_1+\Pi_2+\Pi_3)\mathcal{D}_1\bar{R}^{-1}\mathcal{D}_1^\top(\Pi_1+\Pi_2+\Pi_3)-(\Pi_1+\Pi_2+\Pi_3)\mathcal{D}_1\bar{R}^{-1}\mathcal{D}_2^\top\\
&\quad+\Pi_3\big[\mathcal{A}_1+\mathcal{A}_2+\mathcal{B}_1(\Pi_1+\Pi_2)\big]+\Pi_3\mathcal{B}_1\Pi_3-\mathcal{D}_2\bar{R}^{-1}\mathcal{D}_1^\top(\Pi_1+\Pi_2+\Pi_3)\\
&\quad-\mathcal{D}_2\bar{R}^{-1}\mathcal{D}_2^\top\big\}\check{X}+\big[(\Pi_1+\Pi_2)\mathcal{B}_1\Phi+\Pi_3\mathcal{B}_1\check{\Phi}-(\Pi_1+\Pi_2+\Pi_3)\mathcal{D}_1\bar{R}^{-1}\mathcal{D}_1^\top\check{\Phi}\\
&\quad-(\Pi_1+\Pi_2+\Pi_3)\mathcal{D}_1\bar{R}^{-1}\bar{r}+(\Pi_1+\Pi_2+\Pi_3)\mathcal{C}_1+\mathcal{A}_1\Phi+\mathcal{A}_2\Phi-\mathcal{D}_2\bar{R}^{-1}\mathcal{D}_1^\top\check{\Phi}\\
&\quad-\mathcal{D}_2\bar{R}^{-1}\bar{r}+\mathcal{C}_2\big]\Big\}dt+d\Phi(t)+\big[(\Pi_1+\Pi_3)\Sigma_1-Z_1\big]dW(t)+(\Pi_1\Sigma_2-Z_2)d\bar{W}(t)\\
&+\big[\Pi_2(\Sigma_1+\mathcal{P}F)-Z_3\big]d\widetilde{W}(t)=0.
\end{aligned}
\end{equation}
Let the coefficients of $X,\hat{X},\check{X},dW$ and $d\bar{W}$ in \eqref{zero-zero} be zero, respectively, we derive the three Riccati equations of $\Pi_1(\cdot),\Pi_2(\cdot)$ and $\Pi_3(\cdot)$:
\begin{equation}\label{Pi1}
\left\{
\begin{aligned}
&\dot{\Pi}_1+\Pi_1\mathcal{A}_1+\mathcal{A}_1\Pi_1+\Pi_1\mathcal{B}_1\Pi_1+\mathcal{A}_3=0,\ \ t\in[0,T],\\
&\Pi_1(T)=\mathcal{M},
\end{aligned}
\right.
\end{equation}
\begin{equation}\label{Pi2}
\left\{
\begin{aligned}
&\dot{\Pi}_2+(\mathcal{A}_1+\mathcal{A}_2+\Pi_1\mathcal{B}_1)\Pi_2+\Pi_2(\mathcal{A}_1+\mathcal{A}_2+\mathcal{B}_1\Pi_1)\\
&\quad+\Pi_2\mathcal{B}_1\Pi_2+\Pi_1\mathcal{A}_2+\mathcal{A}_2\Pi_1=0,\ \ t\in[0,T],\\
&\Pi_2(T)=0,
\end{aligned}
\right.
\end{equation}
\begin{equation}\label{Pi3}
\left\{
\begin{aligned}
&\dot{\Pi}_3+\big[(\Pi_1+\Pi_2)(\mathcal{B}_1-\mathcal{D}_1\bar{R}^{-1}\mathcal{D}_1^\top)+\mathcal{A}_1+\mathcal{A}_2-\mathcal{D}_2\bar{R}^{-1}\mathcal{D}_1^\top\big]\Pi_3\\
&\quad+\Pi_3\big[\mathcal{A}_1+\mathcal{A}_2-\mathcal{D}_1\bar{R}^{-1}\mathcal{D}_2^\top+(\mathcal{B}_1-\mathcal{D}_1\bar{R}^{-1}\mathcal{D}_1^\top)(\Pi_1+\Pi_2)\big]\\
&\quad+\Pi_3(\mathcal{B}_1-\mathcal{D}_1\bar{R}^{-1}\mathcal{D}_1^\top)\Pi_3-(\Pi_1+\Pi_2)\mathcal{D}_1\bar{R}^{-1}\mathcal{D}_1^\top(\Pi_1+\Pi_2)-\mathcal{D}_2\bar{R}^{-1}\mathcal{D}_1^\top(\Pi_1+\Pi_2)\\
&\quad-(\Pi_1+\Pi_2)\mathcal{D}_1\bar{R}^{-1}\mathcal{D}_2^\top-\mathcal{D}_2\bar{R}^{-1}\mathcal{D}_2^\top=0,\ \ t\in[0,T],\\
&\Pi_3(T)=0,
\end{aligned}
\right.
\end{equation}
respectively; meanwhile, it yields
\begin{equation}
\left\{
\begin{aligned}
&Z_1=(\Pi_1+\Pi_3)\Sigma_1,\quad a.e.t\in[0,T],\ a.s.,\\
&Z_2=\Pi_1\Sigma_2,\quad a.e.t\in[0,T],\ a.s.
\end{aligned}
\right.
\end{equation}
It remains
\begin{equation}\label{Phi}
\left\{
\begin{aligned}
-d\Phi(t)&=\big\{\big[(\Pi_1+\Pi_2)\mathcal{B}_1+\mathcal{A}_1+\mathcal{A}_2\big]\Phi+\big[-(\Pi_1+\Pi_2+\Pi_3)\mathcal{D}_1\bar{R}^{-1}\mathcal{D}_1^\top\\
         &\qquad+\Pi_3\mathcal{B}_1-\mathcal{D}_2\bar{R}^{-1}\mathcal{D}_1^\top\big]\check{\Phi}-(\Pi_1+\Pi_2+\Pi_3)\mathcal{D}_1\bar{R}^{-1}\bar{r}-\mathcal{D}_2\bar{R}^{-1}\bar{r}\\
         &\qquad+(\Pi_1+\Pi_2+\Pi_3)\mathcal{C}_1+\mathcal{C}_2\big\}dt-\Theta d\widetilde{W}(t),\ t\in[0,T],\\
  \Phi(T)&=\mathcal{M}_T,
\end{aligned}
\right.
\end{equation}
where $\Theta(\cdot)\equiv Z_3(\cdot)-\Pi_2(\Sigma_1+\mathcal{P}F)$ is $\mathcal{F}_t^{Y_1}(\mathcal{I}_t)$-adapted.
\begin{Remark}
Noticing that \eqref{Phi} is a special BSDE for process pair $(\Phi(\cdot),\Theta(\cdot))$, which is only driven by the $\mathcal{F}_t^{Y_1}$-adapted innovation process $\widetilde{W}(\cdot)$. And we know, by the above results, that $\widetilde{W}$ is an $\mathcal{F}_t^{Y_1}$-Brownian motion under the probability measure $P$. Therefore, $\widetilde{W}(\cdot)$ is a $P$-martingale with respect to $\mathcal{F}_t^{Y_1}$, then the martingale representation theorem holds for this innovation process. The second term on the right hand of the equation \eqref{Phi} $\int_0^T\Theta(t)d\widetilde{W}(t)$ is a martingale, and is also a locally martingale. Moreover, the first term on the right hand of \eqref{Phi} is obviously a bounded variation process, thus $\Phi(\cdot)$ is an $\mathcal{F}_t^{Y_1}$-adapted semi-martingale.
\end{Remark}

Applying Theorem 8.1 in \cite{LS77}, we achieve
\begin{equation}\label{checkPhi}
\left\{
\begin{aligned}
-d\check{\Phi}(t)&=\big\{\big[(\Pi_1+\Pi_2+\Pi_3)\mathcal{B}_1-(\Pi_1+\Pi_2+\Pi_3)\mathcal{D}_1\bar{R}^{-1}\mathcal{D}_1^\top+\mathcal{A}_1+\mathcal{A}_2\\
                 &\qquad-\mathcal{D}_2\bar{R}^{-1}\mathcal{D}_1^\top\big]\check{\Phi}-(\Pi_1+\Pi_2+\Pi_3)\mathcal{D}_1\bar{R}^{-1}\bar{r}-\mathcal{D}_2\bar{R}^{-1}\bar{r}\\
                 &\qquad+(\Pi_1+\Pi_2+\Pi_3)\mathcal{C}_1+\mathcal{C}_2\big\}dt-\check{\Theta}dW(t),\ t\in[0,T],\\
  \check{\Phi}(T)&=\mathcal{M}_T,
\end{aligned}
\right.
\end{equation}
where $\check{\Theta}(\cdot)\equiv\check{Z}_3(\cdot)-\Pi_2(\Sigma_1+\mathcal{P}F)$ is $\mathcal{F}_t^{Y_2}$-adapted.

\begin{Remark}
Recalling the condition that $\eta_3(\cdot)$ is $\mathcal{F}_t^{Y_1}$-adapted in \eqref{lambda}, we can now give the explanation by the equation \eqref{Phi}. Since $\Theta(\cdot)=Z_3(\cdot)-\Pi_2(\Sigma_1+\mathcal{P}F)$, and
\begin{equation*}
Z_3=
\begin{pmatrix}
\eta_3\\
\hat{\bar{\Gamma}}
\end{pmatrix}\ \
\end{equation*}
depends on the two components $\eta_3(\cdot)$ and $\hat{\bar{\Gamma}}(\cdot)$, we know $\hat{\bar{\Gamma}}(\cdot)$ is $\mathcal{F}_t^{Y_1}$-adapted. If $\eta_3(\cdot)$ is $\mathcal{F}_t$-adapted, then $Z_3(\cdot)$ is $\mathcal{F}_t$-adapted and $\Theta(\cdot)$ is $\mathcal{F}_t$-adapted, which cause that the $\Phi(\cdot)$ is not $\mathcal{F}_t^{Y_1}$-adapted. Therefore, we need the condition that $\eta_3(\cdot)$ has to be $\mathcal{F}_t^{Y_1}$-adapted.
\end{Remark}

\begin{Remark}
We still make an observation for $\Theta(\cdot)$ in \eqref{Phi}, there is an another way to look at it. We can separate it into three parts: $\eta_3(\cdot),\hat{\bar{\Gamma}}(\cdot)$ and $\Pi_2(\Sigma_1+\mathcal{P}F)$, which come from the $d\widetilde{W}$ term of the leader's adjoint equation \eqref{lambda}, state equation \eqref{leaderstate} and filtering equation \eqref{hatX}, respectively. On one hand, the semi-martingale $\Phi(\cdot)$ depends on the innovation process $\widetilde{W}(\cdot)$. On the other hand, $\widetilde{W}(\cdot)$ can be regarded to be absorbed by $\Phi(\cdot)$, in order to guarantee that $\Phi(\cdot)$ is a semi-martingale process. Therefore, it is reasonable to get the form of \eqref{lambda} with three Brownian motions although $\widetilde{W}(\cdot)$ is not independent of $W(\cdot)$ or $\bar{W}(\cdot)$. Moreover, the existence of the semi-martingale $\Phi(\cdot)$ explains the existence of $\widetilde{W}(\cdot)$ in some equations.
\end{Remark}

\begin{Remark}
Noting that the conjecture \eqref{Y3X} deduces three coupled $2\times2$ Riccati equations whose solvability can be guaranteed sequentially in the following discussion. In fact, the three kinds of equations are not easy to give their explicit solutions, even under some special case or as some degeneration of ordinary differential equations. Therefore, we can consider another equivalent relation between $Y(\cdot)$ and $X(\cdot),\hat{X}(\cdot),\check{X}(\cdot)$, which can deduce three independent Riccati equations. In fact, we let
\begin{equation}\label{Y3X2}
Y(t)=\hat{\Pi}_1(X(t)-\hat{X}(t))+\hat{\Pi}_2(\hat{X}(t)-\check{X}(t))+\hat{\Pi}_3\check{X}(t)+\tilde{\Phi}(t),\ t\in[0,T],
\end{equation}
and we can get the equation of $X(\cdot)-\hat{X}(\cdot)$, $\hat{X}(\cdot)-\check{X}(\cdot)$ by \eqref{together3}, \eqref{hatX} and \eqref{checkX}. Applying It\^{o}'s formula to \eqref{Y3X2} and following the similar procedure, we get three new $2\times2$ Riccati equations as follows:
\begin{equation}\label{hatPi1}
\left\{
\begin{aligned}
&\dot{\hat{\Pi}}_1+\hat{\Pi}_1\mathcal{A}_1+\mathcal{A}_1\hat{\Pi}_1+\hat{\Pi}_1\mathcal{B}_1\hat{\Pi}_1+\mathcal{A}_3=0,\ \ t\in[0,T],\\
&\hat{\Pi}_1(T)=\mathcal{M},
\end{aligned}
\right.
\end{equation}
\begin{equation}\label{hatPi2}
\left\{
\begin{aligned}
&\dot{\hat{\Pi}}_2+(\mathcal{A}_1+\mathcal{A}_2)\hat{\Pi}_2+\hat{\Pi}_2(\mathcal{A}_1+\mathcal{A}_2)+\hat{\Pi}_2\mathcal{B}_1\hat{\Pi}_2+\mathcal{A}_3=0,\ \ t\in[0,T],\\
&\hat{\Pi}_2(T)=\mathcal{M},
\end{aligned}
\right.
\end{equation}
\begin{equation}\label{hatPi3}
\left\{
\begin{aligned}
&\dot{\hat{\Pi}}_3+(\mathcal{A}_1+\mathcal{A}_2-\mathcal{D}_2\bar{R}^{-1}\mathcal{D}_1^\top)\hat{\Pi}_3+\hat{\Pi}_3(\mathcal{A}_1+\mathcal{A}_2-\mathcal{D}_1\bar{R}^{-1}\mathcal{D}_2^\top)\\
&\quad+\hat{\Pi}_3(\mathcal{B}_1-\mathcal{D}_1\bar{R}^{-1}\mathcal{D}_1^\top)\hat{\Pi}_3+\mathcal{A}_3-\mathcal{D}_2\bar{R}^{-1}\mathcal{D}_2^\top=0,\ \ t\in[0,T],\\
&\hat{\Pi}_3(T)=\mathcal{M},
\end{aligned}
\right.
\end{equation}
respectively, and $\tilde{\Phi}(\cdot)$ satisfies
\begin{equation}\label{tildePhi}
\left\{
\begin{aligned}
-d\tilde{\Phi}(t)&=\big\{\big[\hat{\Pi}_2\mathcal{B}_1+\mathcal{A}_1+\mathcal{A}_2\big]\tilde{\Phi}-\big[\hat{\Pi}_2\mathcal{B}_1-\hat{\Pi}_3(\mathcal{B}_1-\mathcal{D}_1\bar{R}^{-1}\mathcal{D}_1^\top)+\mathcal{D}_2\bar{R}^{-1}\mathcal{D}_1^\top\big]\check{\tilde{\Phi}}\\
&\qquad-\hat{\Pi}_3\mathcal{D}_1\bar{R}^{-1}\bar{r}-\mathcal{D}_2\bar{R}^{-1}\bar{r}+\hat{\Pi}_3\mathcal{C}_1+\mathcal{C}_2\big\}dt-\tilde{\Theta}d\widetilde{W}(t),\ t\in[0,T],\\
  \tilde{\Phi}(T)&=\mathcal{M}_T,
\end{aligned}
\right.
\end{equation}
where $\tilde{\Theta}=Z_3-(\Sigma_1+\mathcal{P}F)(\hat{\Pi}_2-\hat{\Pi}_1)$. In fact, we can obtain the following relations between $\Pi_1(\cdot),\Pi_2(\cdot),\Pi_3(\cdot)$ and $\hat{\Pi}_1(\cdot),\hat{\Pi}_2(\cdot),\hat{\Pi}_3(\cdot)$:
  \begin{equation}\label{hatPiPi}
  \begin{aligned}
  &\hat{\Pi}_1(\cdot)=\Pi_1(\cdot);\\
  &\hat{\Pi}_2(\cdot)-\hat{\Pi}_1(\cdot)=\Pi_2(\cdot);\\
  &\hat{\Pi}_3(\cdot)-\hat{\Pi}_2(\cdot)=\Pi_3(\cdot);
 \end{aligned}
  \end{equation}
and we can also check that $(\tilde{\Phi}(\cdot),\tilde{\Theta}(\cdot))=(\Phi(\cdot),\Theta(\cdot))$. The above analysis tells us the conjecture \eqref{Y3X} are equivalent to \eqref{Y3X2} in the sense of \eqref{hatPiPi}. The state estimate feedbacks of optimal control corresponding to \eqref{Y3X} and \eqref{Y3X2}, which although exist some differences, could also be transformed equivalently, Therefore, we mainly derive the state estimate feedback of optimal control utilizing the former \eqref{Y3X}, and give the explicit expression of solution under some special condition utilizing the latter \eqref{Y3X2} in the discussion below.
\end{Remark}

In the following, we will discuss the solvability of Riccati equations \eqref{Pi1}, \eqref{Pi2} and \eqref{Pi3} by giving some useful lemmas. We first discuss the solvability of Riccati equation \eqref{Pi1}.
Introducing the following Riccati equation:
\begin{equation}\label{Pi11}
\left\{
\begin{aligned}
&\dot{\Pi}_{1,1}+\Pi_{1,1}(\mathcal{A}_1+\mathcal{B}_1\mathcal{M})+(\mathcal{A}_1+\mathcal{M}\mathcal{B}_1)\Pi_{1,1}+\Pi_{1,1}\mathcal{B}_1\Pi_{1,1}+\mathcal{M}\mathcal{B}_1\mathcal{M}\\
&\quad+\mathcal{M}\mathcal{A}_1+\mathcal{A}_1\mathcal{M}+\mathcal{A}_3=0,\ \ t\in[0,T],\\
&\Pi_{1,1}(T)=0,
\end{aligned}
\right.
\end{equation}
It is easy to check that the solution $\Pi_{1,1}(\cdot)$ to \eqref{Pi11} and the solution $\Pi_1(\cdot)$ to \eqref{Pi1} have the relation as follows:
\begin{equation}\label{Pi111}
\Pi_1(t)=\mathcal{M}+\Pi_{1,1}(t),\ t\in[0,T].
\end{equation}
Then we give the following lemma to guarantee the solvability of Riccati equation of $\Pi_1(\cdot)$ by Theorem 5.3 of Yong \cite{Yong99}.
\begin{lemma}
 Let {\bf (A1)-(A4)} hold and ${\rm det}\bigg\{\begin{pmatrix}
0&I
\end{pmatrix}e^{\mathcal{A}t}\begin{pmatrix}
0\\
I
\end{pmatrix}
\bigg\}>0,\ t\in[0,T]$ hold. Then \eqref{Pi11} admits a unique solution $\Pi_{1,1}(\cdot)$ which has the following representation
\begin{equation}
\Pi_{1,1}(t)=-\Bigg [\begin{pmatrix}0&I
\end{pmatrix}e^{\mathcal{A}(T-t)}\begin{pmatrix}
0\\
I
\end{pmatrix}\Bigg]^{-1}\begin{pmatrix}0&I
\end{pmatrix}e^{\mathcal{A}(T-t)}\begin{pmatrix}
I\\
0
\end{pmatrix},
\end{equation}
where
\begin{equation}
\mathcal{A}=\begin{pmatrix}
\mathcal{A}_1+\mathcal{B}_1\mathcal{M}&\mathcal{B}_1\\
-(\mathcal{M}\mathcal{B}_1\mathcal{M}+\mathcal{M}\mathcal{A}_1+\mathcal{A}_1\mathcal{M}+\mathcal{A}_3)&-(\mathcal{A}_1+\mathcal{M}\mathcal{B}_1)
\end{pmatrix}.
\end{equation}
Moreover, \eqref{Pi111} gives the solution $\Pi_1(\cdot)$ to the Riccati equation \eqref{Pi1}.
\end{lemma}

Next, we will get the sufficient conditions to guarantee the solvability of two Riccati equations of $\Pi_2(\cdot)$ and $\Pi_3(\cdot)$, respectively, by the same method as discussed above.
\begin{lemma}
Let {\bf (A1)-(A4)} hold and ${\rm det}\bigg\{\begin{pmatrix}
0&I
\end{pmatrix}e^{\mathcal{B}t}\begin{pmatrix}
0\\
I
\end{pmatrix}
\bigg\}>0,\ t\in[0,T]$ hold, and \eqref{Pi1} admits a unique solution $\Pi_1(\cdot)$. Then \eqref{Pi2} admits a unique solution $\Pi_2(\cdot)$ which has the following representation
\begin{equation}
\Pi_2(t)=-\Bigg [\begin{pmatrix}0&I
\end{pmatrix}e^{\mathcal{B}(T-t)}\begin{pmatrix}
0\\
I
\end{pmatrix}\Bigg]^{-1}\begin{pmatrix}0&I
\end{pmatrix}e^{\mathcal{B}(T-t)}\begin{pmatrix}
I\\
0
\end{pmatrix},
\end{equation}
where
\begin{equation}
\mathcal{B}=\begin{pmatrix}
\mathcal{A}_1+\mathcal{A}_2+\mathcal{B}_1\Pi_1&\mathcal{B}_1\\
-(\Pi_1\mathcal{A}_2+\mathcal{A}_2\Pi_1)&-(\mathcal{A}_1+\mathcal{A}_2+\Pi_1\mathcal{B}_1)
\end{pmatrix}.
\end{equation}
\end{lemma}

\begin{lemma}
Let {\bf (A1)-(A4)} hold and ${\rm det}\bigg\{\begin{pmatrix}
0&I
\end{pmatrix}e^{\mathcal{C}t}\begin{pmatrix}
0\\
I
\end{pmatrix}
\bigg\}>0,\ t\in[0,T]$ hold, and \eqref{Pi1} admits a unique solution $\Pi_1(\cdot)$, \eqref{Pi2} admits a unique solution $\Pi_2(\cdot)$. Then \eqref{Pi3} admits a unique solution $\Pi_3(\cdot)$ which has the following representation
\begin{equation}
\Pi_3(t)=-\Bigg [\begin{pmatrix}0&I
\end{pmatrix}e^{\mathcal{C}(T-t)}\begin{pmatrix}
0\\
I
\end{pmatrix}\Bigg]^{-1}\begin{pmatrix}0&I
\end{pmatrix}e^{\mathcal{C}(T-t)}\begin{pmatrix}
I\\
0
\end{pmatrix},
\end{equation}
where
\begin{equation}
\left\{
\begin{aligned}
&\kappa=\mathcal{A}_1+\mathcal{A}_2-\mathcal{D}_1\bar{R}^{-1}\mathcal{D}_2^\top+(\mathcal{B}_1-\mathcal{D}_1\bar{R}^{-1}\mathcal{D}_1^\top)(\Pi_1+\Pi_2),\\
&\varphi=(\Pi_1+\Pi_2)\mathcal{D}_1\bar{R}^{-1}\mathcal{D}_1^\top(\Pi_1+\Pi_2)+\mathcal{D}_2\bar{R}^{-1}\mathcal{D}_1^\top(\Pi_1+\Pi_2)\\
&\quad+(\Pi_1+\Pi_2)\mathcal{D}_1\bar{R}^{-1}\mathcal{D}_2^\top+\mathcal{D}_2\bar{R}^{-1}\mathcal{D}_2^\top\\
&\mathcal{B}=\begin{pmatrix}
\kappa&\mathcal{B}_1-\mathcal{D}_1\bar{R}^{-1}\mathcal{D}_1^\top\\
\varphi&-\kappa^\top
\end{pmatrix}.
\end{aligned}
\right.
\end{equation}
\end{lemma}

We summarize the analysis above in the following theorem.
\begin{theorem}
Let {\bf (A1)-(A4)} hold, $\Pi_1(\cdot),\Pi_2(\cdot)$ and $\Pi_3(\cdot)$ satisfy \eqref{Pi1}, \eqref{Pi2} and \eqref{Pi3}, respectively. Let $\Phi(\cdot)$ be the solution to \eqref{Phi}, $\check{\Phi}(\cdot)$ be the solution to \eqref{checkPhi}, $\check{X}(\cdot)$ be the $\mathcal{F}_t^{Y_2}$-adapted solution to
\begin{equation}\label{checkXX}
\left\{
\begin{aligned}
d\check{X}(t)&=\big\{\big[\mathcal{A}_1+\mathcal{A}_2-\mathcal{D}_1\bar{R}^{-1}\mathcal{D}_2^\top+(\mathcal{B}_1-\mathcal{D}_1\bar{R}^{-1}\mathcal{D}_1^\top)(\Pi_1+\Pi_2+\Pi_3)\big]\check{X}\\
             &\qquad+(\mathcal{B}_1-\mathcal{D}_1\bar{R}^{-1}\mathcal{D}_1^\top)\check{\Phi}-\mathcal{D}_1\bar{R}^{-1}\bar{r}+\mathcal{C}_1\big\}dt+\Sigma_1d{W}(t),\ t\in[0,T],\\
 \check{X}(0)&=X_0,
\end{aligned}
\right.
\end{equation}
and $\hat{X}(\cdot)$ be the $\mathcal{F}_t^{Y_1}$-adapted solution to
\begin{equation}\label{hatXX}
\left\{
\begin{aligned}
d\hat{X}(t)&=\big\{\big[\mathcal{A}_1+\mathcal{A}_2+\mathcal{B}_1(\Pi_1+\Pi_2)\big]\hat{X}+\big[\mathcal{B}_1\Pi_3-\mathcal{D}_1\bar{R}^{-1}\mathcal{D}_1^\top(\Pi_1+\Pi_2+\Pi_3)\\
           &\qquad-\mathcal{D}_1\bar{R}^{-1}\mathcal{D}_2^\top\big]\check{X}+\mathcal{B}_1\Phi-\mathcal{D}_1\bar{R}^{-1}\mathcal{D}_1^\top\check{\Phi}-\mathcal{D}_1\bar{R}^{-1}\bar{r}+\mathcal{C}_1\big\}dt\\
           &\quad+(\Sigma_1+\mathcal{P}F)d\widetilde{W}(t),\ t\in[0,T],\\
 \hat{X}(0)&=X_0.
\end{aligned}
\right.
\end{equation}
Then, due to the relation \eqref{Y3X}, the FBSDE \eqref{together3} admit a unique solution $(X(\cdot),Y(\cdot),Z_1(\cdot),\\Z_2(\cdot),Z_3(\cdot))$ and the state estimate feedback representation of optimal control $\bar{v}_2(\cdot)$ of the leader can be given by
\begin{equation}\label{barv22}
\bar{v}_2(t)=-\bar{R}^{-1}\big[(\mathcal{D}_1^\top(\Pi_1+\Pi_2+\Pi_3)+\mathcal{D}_2^\top)\check{X}+\mathcal{D}_1^\top\check{\Phi}+\bar{r}\big],\ a.e.t\in[0,T],\ a.s.
\end{equation}
\end{theorem}

\begin{proof}
The Riccati equation of $\Pi_1(\cdot),\Pi_2(\cdot)$ and $\Pi_3(\cdot)$ are solvable sequentially. Then, by the standard BSDE theory, \eqref{checkPhi} admits a unique $\mathcal{F}_t^{Y_2}$-adapted solution $(\check{\Phi}(\cdot),\check{\Theta}(\cdot))$. Similarly, \eqref{Phi} admits a unique $\mathcal{F}_t^{Y_1}$-adapted solution $(\Phi(\cdot),\Theta(\cdot))$. By the standard SDE theory, $\check{X}(\cdot)$ and $\hat{X}(\cdot)$ are the $\mathcal{F}_t^{Y_2}$-adapted and $\mathcal{F}_t^{Y_1}$-adapted solutions to \eqref{checkXX} and \eqref{hatXX}, respectively. Noting \eqref{Y3X}, the non-standard FBSDE \eqref{together3} admit a solution $(X(\cdot),Y(\cdot),Z_1(\cdot),Z_2(\cdot),Z_3(\cdot))$ and finally from \eqref{maximum condition3}, $\bar{v}_2(\cdot)$ can be given by \eqref{barv22}. The proof is complete.
\end{proof}

\begin{Remark}
By Theorem 3.6, we can prove the existence of the solutions to the conditional mean-field FBSDE \eqref{together3} via decoupling technique. More specific, $(\check{\Phi}(\cdot),\check{\Theta}(\cdot))$ is first given by \eqref{checkPhi} and then $(\Phi(\cdot),\Theta(\cdot))$ by \eqref{Phi}. Thus $\check{X}(\cdot)$ and $\hat{X}(\cdot)$ are given by \eqref{checkXX} and \eqref{hatXX}, respectively. Next, by \eqref{Y3X}, $\check{Y}(\cdot)$ and $\hat{Y}(\cdot)$ are obtained and the FBSDFE \eqref{together3} admit a solution $(X(\cdot),Y(\cdot),Z_1(\cdot),Z_2(\cdot),Z_3(\cdot))$. Next, we will prove the uniqueness of solutions to \eqref{together3} inspired by the method of \cite{LZ01}.
\end{Remark}

Noticing that there are two different kinds of filtering of $X$ and $Y$, $(\hat{X}(\cdot),\hat{Y}(\cdot))$ and $(\check{X}(\cdot),\check{Y}(\cdot))$, therefore, we first solve them and return to think about the equation of $(X(\cdot),Y(\cdot),Z_1(\cdot),Z_2(\cdot),\\Z_3(\cdot))$. Because of the case, $\mathcal{F}_t^{Y_2}\subset\mathcal{F}_t^{Y_1}$, we then applying Lemma 5.4 in \cite{Xiong08} and Theorem 8.1 in \cite{LS77} to the forward and backward SDEs in \eqref{together3}, respectively,
\begin{equation}\label{together4}
\left\{
\begin{aligned}
 d\check{X}(t)&=\big[(\mathcal{A}_1+\mathcal{A}_2-\mathcal{D}_1\bar{R}^{-1}\mathcal{D}_2^\top)\check{X}+(\mathcal{B}_1-\mathcal{D}_1\bar{R}^{-1}\mathcal{D}_1^\top)\check{Y}-\mathcal{D}_1\bar{R}^{-1}\bar{r}+\mathcal{C}_1\big]dt\\
      &\quad+\Sigma_1dW(t),\\
-d\check{Y}(t)&=\big[(\mathcal{A}_3-\mathcal{D}_2\bar{R}^{-1}\mathcal{D}_2^\top)\check{X}+(\mathcal{A}_1+\mathcal{A}_2-\mathcal{D}_2\bar{R}^{-1}\mathcal{D}_1^\top)\check{Y}-\mathcal{D}_2\bar{R}^{-1}\bar{r}+\mathcal{C}_2\big]dt\\
      &\quad-\check{Z}_{13}dW(t),\quad t\in[0,T],\\
  \check{X}(0)&=X_0,\ \check{Y}(T)=\mathcal{M}\check{X}(T)+\mathcal{M}_T.
\end{aligned}
\right.
\end{equation}
where $\check{Z}_{13}(\cdot):=\check{Z}_1(\cdot)+\check{Z}_3(\cdot)$. Let $(X^1(\cdot),Y^1(\cdot),Z_1^1(\cdot),Z_2^1(\cdot),Z_3^1(\cdot))$ and $(X^2(\cdot),Y^2(\cdot),Z_1^2(\cdot),Z_2^2(\cdot),\\Z_3^2(\cdot))$ be two solutions to \eqref{together3}, $(\Phi^1(\cdot),\Theta^1(\cdot))$ and $(\Phi^2(\cdot),\Theta^2(\cdot))$ be two solutions to \eqref{Phi}, then $(\check{X}^1(\cdot),\check{Y}^1(\cdot),\check{Z}_{13}^1(\cdot))$ and $(\check{X}^2(\cdot),\check{Y}^2(\cdot),\check{Z}_{13}^2(\cdot))$ are two solutions to \eqref{together4}.
If we set
\begin{equation}
\begin{aligned}
&(X^{12}(\cdot),Y^{12}(\cdot),Z_1^{12}(\cdot),Z_2^{12}(\cdot),Z_3^{12}(\cdot))\\
&\qquad=(X^1(\cdot)-X^2(\cdot),Y^1(\cdot)-Y^2(\cdot),Z_1^1(\cdot)-Z_1^2(\cdot),Z_2^1(\cdot)-Z_2^2(\cdot),Z_3^1(\cdot)-Z_3^2(\cdot)),
\end{aligned}
\end{equation}
thus, $(\check{X}^{12}(\cdot),\check{Y}^{12}(\cdot),\check{Z}_{13}^{12}(\cdot))$ satisfies the following equation:
\begin{equation}\label{together5}
\left\{
\begin{aligned}
 d\check{X}^{12}(t)&=\big[(\mathcal{A}_1+\mathcal{A}_2-\mathcal{D}_1\bar{R}^{-1}\mathcal{D}_2^\top)\check{X}^{12}+(\mathcal{B}_1-\mathcal{D}_1\bar{R}^{-1}\mathcal{D}_1^\top)\check{Y}^{12}\big]dt,\\
-d\check{Y}^{12}(t)&=\big[(\mathcal{A}_3-\mathcal{D}_2\bar{R}^{-1}\mathcal{D}_2^\top)\check{X}^{12}+(\mathcal{A}_1+\mathcal{A}_2-\mathcal{D}_2\bar{R}^{-1}\mathcal{D}_1^\top)\check{Y}^{12}\big]dt\\
      &\quad-\check{Z}_{13}^{12}dW(t),\quad t\in[0,T],\\
  \check{X}^{12}(0)&=0,\ \check{Y}^{12}(T)=\mathcal{M}\check{X}^{12}(T).
\end{aligned}
\right.
\end{equation}
Applying It\^{o}'s formula to $\langle \check{X}^{12}(\cdot),\check{Y}^{12}(\cdot)\rangle$, we have
\begin{equation}\label{checkXY}
\begin{aligned}
\mathbb{E}&\langle\check{X}^{12}(T),\mathcal{M}\check{X}^{12}(T)\rangle\\
&=-\mathbb{E}\int_0^T\big[\langle(\mathcal{A}_3-\mathcal{D}_2\bar{R}^{-1}\mathcal{D}_2^\top)\check{X}^{12},\check{X}^{12}\rangle
+\langle(\mathcal{D}_1\bar{R}^{-1}\mathcal{D}_1^\top-\mathcal{B}_1)\check{Y}^{12},\check{Y}^{12}\rangle\big]dt\geq0,
\end{aligned}
\end{equation}
and based on the proof of the existence of solutions, the relation \eqref{Y3X} keeps holding in the process. Then we have $\check{Y}(\cdot)=(\Pi_1+\Pi_2+\Pi_3)\check{X}(\cdot)+\check{\Phi}(\cdot)$ and can further obtain $\check{Y}^{12}(\cdot)=(\Pi_1+\Pi_2+\Pi_3)\check{X}^{12}(\cdot)+\check{\Phi}^{12}(\cdot)$, where, in fact, $\check{\Phi}^{12}(\cdot):=\check{\Phi}^1(\cdot)-\check{\Phi}^2(\cdot)=0$ due to the existence and uniqueness of solution to standard linear BSDE \eqref{checkPhi}. Then \eqref{checkXY} becomes
\begin{equation}\label{checkXY2}
\begin{aligned}
\mathbb{E}&\langle\check{X}^{12}(T),\mathcal{M}\check{X}^{12}(T)\rangle=-\mathbb{E}\int_0^T\Big\{\big\langle[\mathcal{A}_3-\mathcal{D}_2\bar{R}^{-1}\mathcal{D}_2^\top\\
&+(\Pi_1+\Pi_2+\Pi_3)(\mathcal{D}_1\bar{R}^{-1}\mathcal{D}_1^\top-\mathcal{B}_1)(\Pi_1+\Pi_2+\Pi_3)]\check{X}^{12},\check{X}^{12}\big\rangle\Big\}dt\geq0.
\end{aligned}
\end{equation}
{\bf (A5)} Assuming that $\mathcal{A}_3-\mathcal{D}_2\bar{R}^{-1}\mathcal{D}_2^\top+(\Pi_1+\Pi_2+\Pi_3)(\mathcal{D}_1\bar{R}^{-1}\mathcal{D}_1^\top-\mathcal{B}_1)(\Pi_1+\Pi_2+\Pi_3)>0$,
then we get
\begin{equation}
\begin{aligned}
\mathbb{E}\int_0^T\Big\{\langle[\mathcal{A}_3-\mathcal{D}_2\bar{R}^{-1}\mathcal{D}_2^\top+(\Pi_1+\Pi_2+\Pi_3)(\mathcal{D}_1\bar{R}^{-1}\mathcal{D}_1^\top
-\mathcal{B}_1)(\Pi_1+\Pi_2+\Pi_3)]\check{X}^{12},\check{X}^{12}\rangle\Big\}dt=0.
\end{aligned}
\end{equation}
Thus, we have
\begin{equation}\label{checkX0}
\check{X}^{12}(t)=0,\ a.e. t\in[0,T],\ a.s.
\end{equation}
Substituting \eqref{checkX0} into \eqref{together5} and combining with the case of $\check{Y}^{12}(\cdot)=(\Pi_1+\Pi_2+\Pi_3)\check{X}^{12}(\cdot)$, we can verify that $\check{X}^{12}(\cdot)=\check{Y}^{12}(\cdot)=\check{Z}_{13}^{12}(\cdot)=0$ by the uniqueness of solutions to linear SDE and BSDE. Then we obtain the uniqueness of the solution to \eqref{together4}.

Next, applying Theorem 8.1 in \cite{LS77} to \eqref{together3}, we have
 \begin{equation}\label{together6}
\left\{
\begin{aligned}
 d\hat{X}(t)&=\big[(\mathcal{A}_1+\mathcal{A}_2)\hat{X}+\mathcal{B}_1\hat{Y}-\mathcal{D}_1\bar{R}^{-1}\mathcal{D}_1^\top\check{Y}-\mathcal{D}_1\bar{R}^{-1}\mathcal{D}_2^\top\check{X}-\mathcal{D}_1\bar{R}^{-1}\bar{r}+\mathcal{C}_1\big]dt\\
      &\quad+(\Sigma_1+\mathcal{P}F)d\widetilde{W}(t),\\
-d\hat{Y}(t)&=\big[\mathcal{A}_3\hat{X}+(\mathcal{A}_1+\mathcal{A}_2)\hat{Y}-\mathcal{D}_2\bar{R}^{-1}\mathcal{D}_1^\top\check{Y}-\mathcal{D}_2\bar{R}^{-1}\mathcal{D}_2^\top\check{X}-\mathcal{D}_2\bar{R}^{-1}\bar{r}+\mathcal{C}_2\big]dt\\
      &\quad-\hat{\tilde{Z}}(t)d\widetilde{W}(t),\quad t\in[0,T],\\
  \hat{X}(0)&=X_0,\ \hat{Y}(T)=\mathcal{M}\hat{X}(T)+\mathcal{M}_T,
\end{aligned}
\right.
\end{equation}
where $\hat{\tilde{Z}}(\cdot):=\hat{Z}_{1}(\cdot)+Z_3(\cdot)+(\widehat{XY^\top}(\cdot)-\hat{X}(\cdot)\hat{Y}^\top(\cdot))F$. It is easy to verify that $(\hat{X}^{12}(\cdot),\hat{Y}^{12}(\cdot),\hat{\tilde{Z}}^{12}(\cdot))$ satisfies the following equation:
 \begin{equation}\label{together7}
\left\{
\begin{aligned}
 d\hat{X}^{12}(t)&=\big[(\mathcal{A}_1+\mathcal{A}_2)\hat{X}^{12}+\mathcal{B}_1\hat{Y}^{12}\big]dt,\\
-d\hat{Y}^{12}(t)&=\big[\mathcal{A}_3\hat{X}^{12}+(\mathcal{A}_1+\mathcal{A}_2)\hat{Y}^{12}\big]dt-\hat{\tilde{Z}}^{12}(t)d\widetilde{W}(t),\ t\in[0,T],\\
  \hat{X}^{12}(0)&=0,\ \hat{Y}^{12}(T)=\mathcal{M}\hat{X}^{12}(T).
\end{aligned}
\right.
\end{equation}
Applying It\^{o}'s formula to $\langle \hat{X}^{12}(\cdot),\hat{Y}^{12}(\cdot)\rangle$, we have
\begin{equation}\label{hatXY}
\begin{aligned}
\mathbb{E}\langle\hat{X}^{12}(T),\mathcal{M}\hat{X}^{12}(T)\rangle=-\mathbb{E}\int_0^T[\langle\mathcal{A}_3\hat{X}^{12},\hat{X}^{12}\rangle-\langle\mathcal{B}_1\hat{Y}^{12},\hat{Y}^{12}\rangle]dt\geq0.
\end{aligned}
\end{equation}
Similarly, due to the relation \eqref{Y3X}, we have $\hat{Y}(\cdot)=(\Pi_1+\Pi_2)\hat{X}(\cdot)+\Pi_3\check{X}(\cdot)+\Phi(\cdot)$. By \eqref{checkX0} and the uniqueness of the solution to \eqref{Phi}, we can deduce that $\hat{Y}^{12}(\cdot)=(\Pi_1+\Pi_2)\hat{X}^{12}(\cdot)$, thus \eqref{hatXY} becomes
\begin{equation}
\begin{aligned}
\mathbb{E}\langle\hat{X}^{12}(T),\mathcal{M}\hat{X}^{12}(T)\rangle=-\mathbb{E}\int_0^T\big\langle[\mathcal{A}_3-(\Pi_1+\Pi_2)\mathcal{B}_1(\Pi_1+\Pi_2)]\hat{X}^{12},\hat{X}^{12}\big\rangle dt\geq0.
\end{aligned}
\end{equation}
{\bf (A6)} Suppose that $\mathcal{A}_3-(\Pi_1+\Pi_2)\mathcal{B}_1(\Pi_1+\Pi_2)>0$,
then we have
\begin{equation}
\begin{aligned}
\mathbb{E}\int_0^T\big\langle[\mathcal{A}_3-(\Pi_1+\Pi_2)\mathcal{B}_1(\Pi_1+\Pi_2)]\hat{X}^{12},\hat{X}^{12}\big\rangle dt=0,
\end{aligned}
\end{equation}
and it is obvious that $\hat{X}^{12}(t)=0,\ a.e. t\in[0,T],\ a.s.$ Substituting it into \eqref{together7} and combining with $\hat{Y}^{12}(\cdot)=(\Pi_1+\Pi_2)\hat{X}^{12}(\cdot)$, we obtain $\hat{X}^{12}(\cdot)=\hat{Y}^{12}(\cdot)=\hat{\tilde{Z}}^{12}(\cdot)=0$ by the uniqueness of solutions to linear FBSDE \eqref{together7}. Then we obtain the uniqueness of solutions to \eqref{together6}.

Finally, we return to the original equation \eqref{together3}. It is easy to verify that $(X^{12}(\cdot),Y^{12}(\cdot),Z_1^{12}(\cdot),\\Z_2^{12}(\cdot),Z_3^{12}(\cdot))$ are the solution to the following equation
 \begin{equation}\label{together8}
\left\{
\begin{aligned}
 dX^{12}(t)&=\big[\mathcal{A}_1X^{12}+\mathcal{B}_1Y^{12}\big]dt,\\
-dY^{12}(t)&=\big[\mathcal{A}_3X^{12}+\mathcal{A}_1Y^{12}\big]dt-Z_1^{12}dW(t)-Z_2^{12}d\bar{W}(t)-Z_3^{12}d\widetilde{W}(t),\ t\in[0,T],\\
  X^{12}(0)&=0,\ Y^{12}(T)=\mathcal{M}X^{12}(T).
\end{aligned}
\right.
\end{equation}
Applying It\^{o}'s formula to $\langle X^{12}(\cdot),Y^{12}(\cdot)\rangle$, we have
\begin{equation}
\mathbb{E}\langle X^{12}(T),\mathcal{M}X^{12}(T)\rangle=-\mathbb{E}\int_0^T[\langle\mathcal{A}_3 X^{12},X^{12}\rangle-\langle\mathcal{B}_1 Y^{12},Y^{12}\rangle]dt\geq0.
\end{equation}
We can obtain $Y^{12}(\cdot)=\Pi_1X^{12}(\cdot)$ and
\begin{equation}
\mathbb{E}\big\langle X^{12}(T),\mathcal{M}X^{12}(T)\big\rangle=-\mathbb{E}\int_0^T\big\langle(\mathcal{A}_3-\Pi_1\mathcal{B}_1\Pi_1) X^{12},X^{12}\big\rangle]dt\geq0.
\end{equation}
{\bf (A7)} Suppose that $\mathcal{A}_3-\Pi_1\mathcal{B}_1\Pi_1>0$,
then we have
\begin{equation}
\mathbb{E}\int_0^T\big\langle(\mathcal{A}_3-\Pi_1\mathcal{B}_1\Pi_1) X^{12},X^{12}\big\rangle dt=0,
\end{equation}
which means that $X^{12}(t)=0,\ a.e.t\in[0,T],\ a.s.$ Then we have $X^{12}(\cdot)=Y^{12}(\cdot)=Z_1^{12}(\cdot)=Z_2^{12}(\cdot)=Z_3^{12}(\cdot)=0$ with the same analysis in the above. Consequently, we obtain the uniqueness of solutions to \eqref{together3}. In conclusion, we have proved the existence and uniqueness of solutions to \eqref{together3} under assumption {\bf (A1)-(A7)}.

Note that, by the reasonable analysis, we get the unique solvability of \eqref{together3}, which is equivalent to the existence and uniqueness of solutions to \eqref{together}. And naturally, the equations \eqref{lambda} and \eqref{K} admit unique solutions $(\lambda(\cdot),\eta_1(\cdot),\eta_2(\cdot),\eta_3(\cdot))$ and $K(\cdot)$, respectively.

Before the end of this section, by \eqref{Barv2}, \eqref{notation} and \eqref{Y3X}, we derive
\begin{equation}\label{barv11}
\begin{aligned}
\bar{v}_1(t)&=-R^{-1}\big[B_1\Pi\hat{x}^{\bar{v}_2}+B_1\hat{\bar{\theta}}+r\big]\\
&=-R^{-1}\Big\{\Big[\begin{pmatrix}
B_1\Pi&0
\end{pmatrix}+\begin{pmatrix}
0&B_1
\end{pmatrix}(\Pi_1+\Pi_2)\Big]\hat{X}\\
&\qquad\qquad+\begin{pmatrix}
0&B_1
\end{pmatrix}\Pi_3\check{X}+\begin{pmatrix}
0&B_1
\end{pmatrix}\Phi+r\Big\},\ a.e.t\in[0,T],\ a.s.
\end{aligned}
\end{equation}

Up to now, the Stackelberg equilibrium point $(\bar{v}_1(\cdot),\bar{v}_2(\cdot))$ of our problem is obtained.

\section{The Explicit Solutions for Some Special Case}

In this section, we focus on some special case and give some explicit expression of Stackelber equilibrium points for the follower and the leader. The main assumption in this section is that $f_1(\cdot)=0$, that is, $\mathcal{F}_t^{Y_1}=\mathcal{F}_t^{Y_2}=\mathcal{F}_t^W$. In this case, both the follower and the leader can only observe the state system partially, but the known information is same. Moreover, as we mentioned in Remark 2.3, the filtrations becomes $\mathcal{F}_t^{Y_1}=\mathcal{F}_t^W$ and $\mathcal{F}_t^{Y_2}=\mathcal{F}_t^{\bar{W}}$ when $f_1(\cdot)=0$, that is, they also observe the system partially, but the information filtrations are independent. In this later case, it only adds some complexity of the calculation and has no influence on the representation of the explicit solutions. So we omit this case and only discuss the former case, $\mathcal{F}_t^{Y_1}=\mathcal{F}_t^{Y_2}=\mathcal{F}_t^W$.

In this case, we also assume that $M(\cdot)=L(\cdot)=\bar{L}(\cdot)=0$. The main difference is that the observation processes of the follower and the leader become
\begin{equation}\label{sobservation f}
\left\{
\begin{aligned}
dY_1(t)&=g(t)dt+dW(t),\ t\in[0,T],\\
 Y_1(0)&=0,
\end{aligned}
\right.
\end{equation}
\begin{equation}\label{sobservation l}
\left\{
\begin{aligned}
dY_2^{v_2}(t)&=\big[f_2(t)+v_2(t)\big]dt+dW(t),\ t\in[0,T],\\
 Y_2^{v_2}(0)&=0,
\end{aligned}
\right.
\end{equation}
respectively, and the cost functionals become
\begin{equation}\label{scost functional}
\begin{aligned}
J_1(v_1(\cdot),v_2(\cdot))&=\frac{1}{2}\mathbb{E}\bigg\{\int_0^T\big[R(t)v_1^2(t)+2l(t)x^{v_1,v_2}(t)+2r(t)v_1(t)\big]dt+2mx^{v_1,v_2}(T)\bigg\},
\end{aligned}
\end{equation}
\begin{equation}\label{LQ cost functional2}
\begin{aligned}
J_2(v_1(\cdot),v_2(\cdot))&=\frac{1}{2}\mathbb{E}\bigg\{\int_0^T\big[\bar{R}(t)v_2^2(t)+2\bar{l}(t)x^{v_2}(t)+2\bar{r}(t)v_2(t)\big]dt\\
&+\bar{M}(x^{v_2}(T))^2+2\bar{m}x^{v_2}(T)\bigg\}.
\end{aligned}
\end{equation}

\subsection{The Special Case for The Follower}

By Theorem 3.3 in Section 3, we have the optimal control of the follower:
\begin{equation}\label{sBarv2}
\bar{v}_1(t)=-R^{-1}(t)\big[B_1(t)\Pi(t)\hat{\bar{x}}^{\bar{v}_1,v_2}(t)+B_1(t)\hat{\bar{\theta}}(t)+r(t)\big],\ a.e.t\in[0,T],\ a.s.,
\end{equation}
where $\Pi(\cdot)$ satisfy the following special ODE:
\begin{equation}\label{sriccatieq}
\left\{
\begin{aligned}
&\dot{\Pi}(t)+2A(t)\Pi(t)-B_1^2(t)\Pi^2(t)R^{-1}(t)=0,\ t\in[0,T],\\
&\Pi(T)=0.
\end{aligned}
\right.
\end{equation}
Due to the existence and uniqueness of solutions to ODE, we can solve \eqref{sriccatieq} and it has the unique solution $\Pi(t)=0,\ t\in[0,T]$. Then \eqref{sBarv2} reduces to
\begin{equation}\label{ssBarv1}
\bar{v}_1(t)=-R^{-1}(t)\big[B_1(t)\hat{\bar{\theta}}(t)+r(t)\big],\ a.e.t\in[0,T],\ a.s.,
\end{equation}
where $\hat{\bar{\theta}}(\cdot)$ satisfies
\begin{equation}
\left\{
\begin{aligned}
-d\hat{\bar{\theta}}(t)&=(A\hat{\bar{\theta}}+l)dt-\hat{\bar{\lambda}}_1dW(t),\ t\in[0,T],\\
\hat{\bar{\theta}}(T)&=m.
\end{aligned}
\right.
\end{equation}
By the BSDE theory with deterministic terminal value, we can get its explicit solution:
\begin{equation}
(\hat{\bar{\theta}},0)=\bigg(\bigg[\int_0^Tl(s)e^{-2\int_0^sA(r)dr}ds+m\bigg]e^{\int_0^tA(s)ds},0\bigg).
\end{equation}

\subsection{The Special Case for The Leader}

When considering the special condition, $L(\cdot)=\bar{L}(\cdot)=M(\cdot)\equiv0$ and knowing the follower would take the optimal control \eqref{ssBarv1}, the leader's optimal control would reduce to:
\begin{equation}\label{ssBarv2}
\bar{R}(t)\bar{v}_2(t)+\bar{r}(t)+\mathcal{D}_1^\top(t)\hat{Y}(t)=0,\ a.e.t\in[0,T],\ a.s.
\end{equation}
where $(X(\cdot),Y(\cdot))$ satisfies
\begin{equation}\label{stogether}
\left\{
\begin{aligned}
 dX(t)&=\big[\mathcal{A}_1X+\mathcal{B}_1Y+\mathcal{D}_1\bar{v}_2+\mathcal{C}_1\big]dt+\Sigma_1dW(t)+\Sigma_2d\bar{W}(t),\\
-dY(t)&=\big[\mathcal{A}_1Y+\mathcal{C}_2\big]dt-Z_1dW(t)-Z_2d\bar{W}(t),\ t\in[0,T],\\
  X(0)&=X_0,\ Y(T)=\mathcal{M}X(T)+\mathcal{M}_T.
\end{aligned}
\right.
\end{equation}
Let
\begin{equation}\label{notation}
X=
\begin{pmatrix}
x^{\bar{v}_2}\\
K
\end{pmatrix},\ \
Y=
\begin{pmatrix}
\lambda\\
\hat{\bar{\theta}}
\end{pmatrix},\ \
Z_1=
\begin{pmatrix}
\eta_1\\
\hat{\bar{\lambda}}_1
\end{pmatrix},\ \
Z_2=
\begin{pmatrix}
\eta_2\\
0
\end{pmatrix},\ \
\end{equation}
and
\begin{equation*}
\left\{
\begin{aligned}
&\mathcal{A}_1=
\begin{pmatrix}
A&0\\
0&A
\end{pmatrix},\ \
\mathcal{A}_2=
\begin{pmatrix}
0&0\\
0&0
\end{pmatrix},\ \
\mathcal{A}_3=
\begin{pmatrix}
0&0\\
0&0
\end{pmatrix},\\
&\mathcal{B}_1=
\begin{pmatrix}
0&-B_1^2R^{-1}\\
-B_1^2R^{-1}&0
\end{pmatrix},\ \
\mathcal{C}_1=
\begin{pmatrix}
-B_1R^{-1}r+\alpha\\
0
\end{pmatrix},\\
&\mathcal{C}_2=
\begin{pmatrix}
\bar{l}\\
l
\end{pmatrix},\ \
\mathcal{D}_1=
\begin{pmatrix}
B_2\\
0
\end{pmatrix},\ \
\mathcal{D}_2=
\begin{pmatrix}
0\\
0
\end{pmatrix},\ \
\Sigma_1=
\begin{pmatrix}
c\\
0
\end{pmatrix},\\
&\Sigma_2=
\begin{pmatrix}
\bar{c}\\
0
\end{pmatrix},\ \
X_0=
\begin{pmatrix}
x_0\\
0
\end{pmatrix},\ \
\mathcal{M}=
\begin{pmatrix}
\bar{M}&0\\
0&0
\end{pmatrix},\ \
\mathcal{M}_T=
\begin{pmatrix}
\bar{m}\\
m
\end{pmatrix}.
\end{aligned}
\right.
\end{equation*}
Then putting \eqref{ssBarv2} into \eqref{stogether}, we get
\begin{equation}\label{stogether2}
\left\{
\begin{aligned}
 dX(t)&=\big[\mathcal{A}_1X+\mathcal{B}_1Y-\mathcal{D}_1\bar{R}^{-1}\mathcal{D}_1^{\top}\hat{Y}-\mathcal{D}_1\bar{R}^{-1}\bar{r}+\mathcal{C}_1\big]dt+\Sigma_1dW(t)+\Sigma_2d\bar{W}(t),\\
-dY(t)&=\big[\mathcal{A}_1Y+\mathcal{C}_2\big]dt-Z_1dW(t)-Z_2d\bar{W}(t),\ t\in[0,T],\\
  X(0)&=X_0,\ Y(T)=\mathcal{M}X(T)+\mathcal{M}_T.
\end{aligned}
\right.
\end{equation}
According to the analysis in Remark 3.12, we build the relation between $Y(\cdot)$ and $X(\cdot),\hat{X}(\cdot)$ as follows:
\begin{equation}\label{sY2X}
Y(t)=\bar{\Pi}_1(X(t)-\hat{X}(t))+\bar{\Pi}_2(t)\hat{X}(t)+\Phi(t),\ t\in[0,T],
\end{equation}
where $\bar{\Pi}_1(\cdot)$ and $\bar{\Pi}_2(\cdot)$ satisfy the following ODEs, respectively:
\begin{equation}\label{sode1}
\left\{
\begin{aligned}
&\dot{\bar{\Pi}}_1+\bar{\Pi}_1\mathcal{A}_1+\mathcal{A}_1\bar{\Pi}_1+\bar{\Pi}_1\mathcal{B}_1\bar{\Pi}_1=0,\ t\in[0,T],\\
&\bar{\Pi}_1(T)=\mathcal{M},
\end{aligned}
\right.
\end{equation}
\begin{equation}\label{sode2}
\left\{
\begin{aligned}
&\dot{\bar{\Pi}}_2+\bar{\Pi}_2\mathcal{A}_1+\mathcal{A}_1\bar{\Pi}_2+\bar{\Pi}_2(\mathcal{B}_1-\mathcal{D}_1\bar{R}^{-1}\mathcal{D}_1^\top)\bar{\Pi}_2=0,\ t\in[0,T],\\
&\bar{\Pi}_2(T)=\mathcal{M},
\end{aligned}
\right.
\end{equation}
We note that $\bar{\Pi}(\cdot)$ and $\bar{\Pi}(\cdot)$ are both $2\times2$ matrix-value functions, and we can get their solutions as follows:
\begin{equation}
\bar{\Pi}_1=\begin{pmatrix}
\bar{M}e^{\int_0^t2A(s)ds}&0\\
0&0\\
\end{pmatrix},\ \
\bar{\Pi}_2=\begin{pmatrix}
\frac{\bar{M}e^{\int_0^t2A(s)}ds}{\bar{M}\int_0^tB_2^2\bar{R}^{-1}e^{\int_0^s2A(r)dr}ds+1}&0\\
0&0\\
\end{pmatrix}.
\end{equation}
Meanwhile, $\Phi(\cdot),\hat{\Phi}(\cdot)$ are the solutions to the following BSDEs, respectively:
\begin{equation}
\left\{
\begin{aligned}
-d\Phi(t)&=\big[\bar{\Pi}_1\mathcal{B}_1(\Phi-\hat{\Phi})+\bar{\Pi}_2(\mathcal{B}_1-\mathcal{D}_1\bar{R}^{-1}\mathcal{D}_1^\top)\hat{\Phi}-\bar{\Pi}_2\mathcal{D}_1\bar{R}^{-1}\bar{r}
          +\bar{\Pi}_2\mathcal{C}_1+\mathcal{A}_1\Phi+\mathcal{C}_2\big]dt\\
         &\quad-\Theta_1dW(t)-\Theta_2d\bar{W}(t),\ t\in[0,T],\\
  \Phi(T)&=\mathcal{M}_T,
\end{aligned}
\right.
\end{equation}
\begin{equation}
\left\{
\begin{aligned}
-d\hat{\Phi}(t)&=\big\{[\bar{\Pi}_2(\mathcal{B}_1-\mathcal{D}_1\bar{R}^{-1}\mathcal{D}_1^\top)+\mathcal{A}_1]\hat{\Phi}-\bar{\Pi}_2\mathcal{D}_1\bar{R}^{-1}\bar{r}+\bar{\Pi}_2\mathcal{C}_1+\mathcal{C}_2\big\}dt\\
               &\quad-\hat{\Theta}_1dW(t),\ t\in[0,T],\\
  \hat{\Phi}(T)&=\mathcal{M}_T.
\end{aligned}
\right.
\end{equation}
By \eqref{sY2X}, the optimal control \eqref{ssBarv2} has the following explicit representation
\begin{equation}\label{SBarv2}
\begin{aligned}
\bar{v}_2&=-\bar{R}^{-1}(\bar{r}+\mathcal{D}_1^\top\bar{\Pi}_2\hat{X}+\mathcal{D}_1^\top\hat{\Phi})\\
&=-\bar{R}^{-1}(\bar{r}+B_2\bar{\Pi}_{211}\hat{\bar{x}}+B_2\hat{\Phi}_1),\ a.e.t\in[0,T],\ a.s.
\end{aligned}
\end{equation}
where $\hat{\bar{x}}(\cdot)=\hat{x}^{\bar{v}_2}(\cdot)$, $\bar{\Pi}_{211}$ denotes the value of the 1 row and the 1 column of $\bar{\Pi}_2$. And $\hat{\Phi}_1(\cdot)$ denote the first component of $\hat{\Phi}(\cdot)$ which satisfies the following BSDE:
\begin{equation}
\left\{
\begin{aligned}
-d\hat{\Phi}_1(t)&=\big[(A-\bar{\Pi}_{211}B_2^2\bar{R}^{-1})\hat{\Phi}_1-\bar{\Pi}_{211}B_1^2R^{-1}\hat{\Phi}_2-\bar{\Pi}_{211}B_2\bar{R}^{-1}\bar{r}-\bar{\Pi}_{211}B_1R^{-1}r\\
                 &\quad+\bar{\Pi}_{211}\alpha+\bar{l}\big]dt-\hat{\Theta}_{11}dW(t),\ t\in[0,T],\\
  \hat{\Phi}_1(T)&=\bar{m},
\end{aligned}
\right.
\end{equation}
which admits a unique solution
\begin{equation}
\begin{aligned}
(\hat{\Phi}_1,0)=
\bigg(\bigg[\int_0^t\mathcal{H}(\bar{\Pi}_{211},\hat{\Phi}_2)e^{\int_0^s-2(A-\bar{\Pi}_{211}B_2^2\bar{R}^{-1})(r)dr}ds+\bar{m}\bigg]e^{\int_0^t(A-\bar{\Pi}_{211}B_2^2\bar{R}^{-1})(s)ds},0\bigg),
\end{aligned}
\end{equation}
where $\mathcal{H}(\bar{\Pi}_{211},\hat{\Phi}_2)=-\bar{\Pi}_{211}B_1^2R^{-1}\hat{\Phi}_2-\bar{\Pi}_{211}B_2\bar{R}^{-1}\bar{r}-\bar{\Pi}_{211}B_1R^{-1}r+\bar{\Pi}_{211}\alpha+\bar{l}$, and $\hat{\Phi}_2(\cdot)$, as the second component of $\hat{\Phi}(\cdot)$, satisfies the following BSDE:
\begin{equation}
\left\{
\begin{aligned}
-d\hat{\Phi}_2(t)&=(A\hat{\Phi}_2+l)dt-\hat{\Theta}_{12}dW(t),\ t\in[0,T],\\
\hat{\Phi}_2(T)&=m,
\end{aligned}
\right.
\end{equation}
which admits a unique solution:
\begin{equation}
\begin{aligned}
(\hat{\Phi}_2,0)=
\bigg(\bigg[\int_0^tl(s)e^{-2\int_0^sA(r)dr}ds+m\bigg]e^{\int_0^tA(s)ds},0\bigg).
\end{aligned}
\end{equation}
Finally, $\hat{\bar{x}}(\cdot)$ is the solution to the following SDE:
\begin{equation}
\left\{
\begin{aligned}
d\hat{\bar{x}}(t)&=[(A-\bar{\Pi}_{211}B_2^2\bar{R}^{-1})\hat{\bar{x}}-B_2^2\bar{R}^{-1}\hat{\Phi}_1-B_1^2R^{-1}\hat{\Phi}_2\\
                 &\quad-B_2\bar{R}^{-1}\bar{r}-B_1R^{-1}r+\alpha]dt+cdW(t),\ t\in[0,T],\\
 \hat{\bar{x}}(0)&=x_0,
\end{aligned}
\right.
\end{equation}
and we can give its explicit expression in the following:
\begin{equation}
\begin{aligned}
\hat{\bar{x}}(t)&=x_0\Psi^{-1}(t)+\Psi^{-1}(t)\bigg[\int_0^t(-B_2^2\bar{R}^{-1}\hat{\Phi}_1-B_1^2R^{-1}\hat{\Phi}_2-B_2\bar{R}^{-1}\bar{r}\\
                &\qquad-B_1R^{-1}r+\alpha)\Psi(s)ds+\int_0^tc\Psi(s)dW(s)\bigg],
\end{aligned}
\end{equation}
where $\Psi(t)=e^{-\int_0^t(A-\bar{\Pi}_{211}B_2^2\bar{R}^{-1})(s)ds}$.

Therefore, we can give the explicit expressions of optimal control corresponding to the follower and the leader as \eqref{ssBarv1} and \eqref{SBarv2}.

\section{Application to Dynamic Cooperative Advertising with Asymmetric Information}

In this section, we aim to demonstrate the effectiveness of the theoretical results in the previous sections, by a motivated dynamic cooperative advertising problem which is described as a Stackelberg differential game with a manufacturer and a retailer with asymmetric information.

\subsection{Dynamic Cooperative Advertising Problem}

We consider a dynamic cooperative advertising problem in the marketing channel, which contains two participants, the retailer (label 1) and the manufacturer (label 2). For increasing the brand image, the manufacturer controls the rate of advertising efforts in the national media, and the retailer controls the local promotional activities for the brand. In Jorgrnsen et al. \cite{JSZ00}, they consider the case the retailer's promotion enhances current sales of the product and also has a positive impact on the brand image. In Jorgrnsen et al. \cite{JTZ01}, the retailer's promotion is benefit to the current sale of the product, however, the impact of promotions on the brand is not considered. In Jorgrnsen et al. \cite{JTZ03}, the retailer's promotion, although has the positive impact on the sale revenue, is assumed to erode the brand image. Moreover, we should note that all of them consider that the manufacturer can support the retailer's advertising promotional effort by paying part of the cost incurred by a retailer. In our paper, we extend the framework above into the stochastic case with asymmetric information, which is not studied before.

Under our partial information background, in this section, we consider that the manufacturer and the retailer have their own positive impact on the sale of the product, $Q_2$ and $Q_1$, respectively. We study in our example the case where the retailer's promotion effort has a negative impact on the brand image. More differently, we assume that the manufacturer dose not support the retailer's advertising cost, we can explain the reason with our asymmetric information feature. In our setting, we consider the case that no one can observe all the informations completely in a practical situation, thus we assume that the manufacturer and the retailer can only acquire the partial information by their own observation equation:
\begin{equation*}
\left\{
\begin{aligned}
dY_1^{v_1,v_2}(t)&=\big[f_1(t)x^{v_1,v_2}(t)+g(t)\big]dt+dW(t),\ t\in[0,T],\\
  Y_1^{v_1,v_2}(0)&=0,
\end{aligned}
\right.
\end{equation*}
\begin{equation*}
\left\{
\begin{aligned}
dY_2^{v_2}(t)&=[f_2(t)+v_2(t)]dt+dW(t),\ t\in[0,T],\\
  Y_2^{v_2}(0)&=0,
\end{aligned}
\right.
\end{equation*}
respectively, and it is assumed that the information available to the manufacturer (the leader) is less than that of the retailer (the follower).

Combining the interesting phenomenon, we can make our analysis that both the manufacturer and the retailer can not observe the dynamic market influenced by the uncertain factor such as the competition of the marketing shares among the different rivals, the limitation to the local government's policy and the information delay due to the distribution of the channels. Moreover, for a manufacturer, he/she has a lot of sub-products in production, thus he/she is unable to pay all the attention to just one kind of product. But for the retailer, he/she makes a local promotional activities, compared to the manufacturer. Firstly, the retailer is often already experienced in dealing with the needs and preferences of potential adopters as the retailer may provide the potential adopters with other related products and services. Secondly, the retailer may have a long-term and stable relationship with the target customers, which has more advantages to reach potential customers faster and more conveniently. Thirdly, he/she is located in the same geographical region as his/her potential customer base, which is more efficient in providing the after-sale service and getting the information feedback in time. The retailer is natural to build his/her own linkedin and sale network which provide more information that the manufacturer does not know and can get the considerable profit. Therefore, we assume that the retailer grasps more information than the manufacturer does. Based on the above analysis, it is reasonable to assume that the manufacturer have no need to support the retailer's promotion advertising effort. However, at the same time, the retailer's promotion can damage the brand image because the consumers come to believe that frequent promotions are used as a ``cover up" for insufficient quality and they think that high quality products do not worry about their sales and need little or no promotion. More details can be referred in Davis et al. \cite{DIM92}, Papatla and Krishnamurthi \cite{PK96}, Yoo et al. \cite{YDL00}.

Next, we describe the Stackelberg differential game model in detail. Suppose that the manufacturer controls the advertising rate $v_2(\cdot)$ in national media (e.g., television, radio) and the retailer controls the local promotion effort rate $v_1(\cdot)$ (e.g., in-store display, advertising in local newspapers). The manufacturer's advertising effort has a positive impact on the brand image, on the contrary, the retailer's local promotion effort has a negative impact on it. Moreover, different from the deterministic case in \cite{JTZ03}, we take the market uncertainty into account. Therefore, the dynamic of the brand image $x(\cdot)$ is characterized by an SDE:
\begin{equation}\label{brand image}
\left\{
\begin{aligned}
dx^{v_1,v_2}(t)&=\big[\beta_2 v_2(t)-\beta_1 v_1(t)-\delta x^{v_1,v_2}(t)\big]dt+\sigma dW(t)+\bar{\sigma}d\bar{W}(t),\ t\in[0,T],\\
 x^{v_1,v_2}(0)&=x_0,
\end{aligned}
\right.
\end{equation}
where $\beta_2$ and $\beta_1$ are positive parameters measuring the impact of the manufacturer's advertising and the retailer's promotion on the brand image, respectively. $\delta>0$ is the decay rate of the brand image. $W(\cdot)$ and $\bar{W}(\cdot)$ are two independent Brownian motions (random noises) which interfere the brand image. $\sigma,\bar{\sigma}$ are the constant volatility coefficients.

Another assumption different form the existing literatures is that the manufacturer's contribution to the sales revenue rate of the product is given by
\begin{equation}\label{c2}
c_2(v_2(t),x(t))=\gamma_2 v_2(t)+\theta_2 x(t)+\kappa_2 x^2(t),
\end{equation}
and the retailer's contribution to the sales revenue rate of the product is given by
\begin{equation}\label{c1}
c_1(v_1(t),x(t))=\gamma_1 v_1(t)+\theta_1 x(t)+\kappa_1 x^2(t),
\end{equation}
where $\gamma_1,\gamma_2$ are positive constants measuring the effects of the retailer's promotion and manufacturer's advertising effort on the current sales revenue, respectively; $\theta_1,\theta_2$ and $\kappa_1,\kappa_2$ are positive constants representing the effects of brand image on current sales revenue. From \eqref{c2}, we can see both advertising effort and brand image have a positive impact on the current sales revenue which is the main objective of the manufacturer. From \eqref{c1}, we can see the retailer's trade-off: on one hand, the promotion increases the current sales revenue, on the other hand, from \eqref{brand image}, the promotion erodes the brand image which has a positive impact on the current sales revenue, and to some degree, it hinders the brand image's benefits to the current sales revenue.

Then we aim to minimize the retailer's objective functional
\begin{equation}\label{rcf}
\begin{aligned}
J_1(v_1(\cdot))=\mathbb{E}\bigg\{\int_0^T\Big[-\big(\gamma_1 v_1(t)+\theta_1 x(t)+\kappa_1 x^2(t)\big)+\frac{\mu_1}{2}v_1^2(t)\Big]dt+M_1(x^{v_1,v_2}(T))^2\bigg\},
\end{aligned}
\end{equation}
and the manufacturer's
\begin{equation}\label{mcf}
\begin{aligned}
J_2(v_2(\cdot))=\mathbb{E}\bigg\{\int_0^T\Big[-\big(\gamma_2 v_2(t)+\theta_2 x(t)+\kappa_2 x^2(t)\big)+\frac{\mu_2}{2}v_2^2(t)\Big]dt+M_2(x^{v_1,v_2}(T))^2\bigg\},
\end{aligned}
\end{equation}
respectively.
It is reasonable to formulate the objective functional as three parts. The first part indicate the sales revenue would be maximized. The second part is to represent the quadratic cost of the promotion $v_1(\cdot)$ and advertising effort $v_2(\cdot)$, respectively, which need to be minimized, and $\mu_1>0$ and $\mu_2>0$ are constant cost coefficients of $v_1(\cdot)$ and $v_2(\cdot)$, respectively. The third part is to imply the cost of maintaining a pretty brand image in the terminal time which aims to be minimized, $M_1$ and $M_2$ are the corresponding positive cost coefficients.

\begin{Remark}
A difference could be noticed in the objective functionals of the follower and leader, the weight coefficients of $x^2$ are negative when $\kappa_1,\kappa_2$ are assumed to be positive, which seems to be a contradiction to the assumption {\bf (A2)}. In fact, {\bf (A2)} includes a similar condition as standard assumption (4.23) in Chapter 6 of \cite{YZ99}, which can be regarded as a sufficient condition to guarantee the solvability of Riccati equations. But as we can see in the next subsection, we can plot the trajectories of these solutions to Riccati equations in the Figures 1-4 under some particular coefficients while the assumption {\bf (A2)} is not satisfied, which means that the Riccati equations \eqref{riccatieq---}, \eqref{Pi1---}, \eqref{Pi2---} and \eqref{Pi3---} admit unique solutions, respectively. In other word, the assumption {\bf (A2)} is a sufficient condition for the solvability but not a necessary one.
\end{Remark}

\subsection{Numerical Simulation}

In this subsection, we solve the above dynamic cooperative advertising problem by our main results in Section 3. We fail to give the explicit expression of optimal controls of follower and leader because of the complex and coupled form of Riccati equations and  FBSDEs in section 3. Therefore we will give the numerical simulation and draw some figures of Riccati equations $\Pi(\cdot),\Pi_1(\cdot),\Pi_2(\cdot),\Pi_3(\cdot),\mathcal{P}(\cdot)$, the brand image $x(\cdot)$ and its filtering $\check{x}(\cdot),\hat{x}(\cdot)$, the manufacturer's advertising effort $v_2(\cdot)$ and the retailer's promotion $v_1(\cdot)$, to give the reasonable analysis for clarifying the effectiveness of the theoretical results.

Compared to \eqref{LQsde}, \eqref{LQ cost functional} and \eqref{LQ cost functional2}, we have $A(t)=-\delta, B_1(t)=-\beta_1, B_2(t)=\beta_2, \alpha=0, c(t)=\sigma, \bar{c}(t)=\bar{\sigma}, L(t)=-2\kappa_1, R(t)=\mu_1, l(t)=-\theta_1, r(t)=-\gamma_1, M=2M_1, m=0, \bar{L}(t)=-2\kappa_2, \bar{R}(t)=\mu_2, \bar{l}(t)=-\theta_2, \bar{r}(t)=-\gamma_2, \bar{M}=2M_2$ and $\bar{m}=0$. And we put
\begin{equation*}
\left\{
\begin{aligned}
&\mathcal{A}_1(t)=
\begin{pmatrix}
-\delta&0\\
0&-\delta-\beta_1^2\Pi(t)\mu_1^{-1}
\end{pmatrix},\ \
\mathcal{A}_2(t)=
\begin{pmatrix}
-\beta_1^2\Pi(t)\mu_1^{-1}&0\\
0&0
\end{pmatrix},\\
&\mathcal{A}_3=
\begin{pmatrix}
-2\kappa_2&0\\
0&0
\end{pmatrix},\ \
\mathcal{B}_1=
\begin{pmatrix}
0&-\beta_1^2\mu_1^{-1}\\
-\beta_1^2\mu_1^{-1}&0
\end{pmatrix},\ \
\mathcal{C}_1=
\begin{pmatrix}
-\beta_1\gamma_1\mu_1^{-1}\\
0
\end{pmatrix},\\
&\mathcal{C}_2(t)=
\begin{pmatrix}
-\theta_2\\
-\beta_1\Pi(t)\gamma_1\mu_1^{-1}-\theta_1
\end{pmatrix},\ \
\mathcal{D}_1=
\begin{pmatrix}
\beta_2\\
0
\end{pmatrix},\ \
\mathcal{D}_2(t)=
\begin{pmatrix}
0\\
\beta_2\Pi(t)
\end{pmatrix},\\
&\Sigma_1=
\begin{pmatrix}
\sigma\\
0
\end{pmatrix},\ \
\Sigma_2=
\begin{pmatrix}
\bar{\sigma}\\
0
\end{pmatrix},\ \
X_0=
\begin{pmatrix}
x_0\\
0
\end{pmatrix},\ \
\mathcal{M}=
\begin{pmatrix}
2M_2&0\\
0&0
\end{pmatrix},\ \
\mathcal{M}_T=
\begin{pmatrix}
0\\
0
\end{pmatrix}.
\end{aligned}
\right.
\end{equation*}
Applying the theoretical results in Section 3, we can get the optimal solution to the retailer's promotion effort and manufacturer's advertising effort, $\bar v_1(\cdot)$ and $\bar v_2(\cdot)$, respectively:
\begin{equation}
\left\{
\begin{aligned}
\bar{v}_1(t)=&-\mu_1^{-1}\Big\{\Big[\begin{pmatrix}
-\beta_1\Pi(t)&0
\end{pmatrix}+\begin{pmatrix}
0&-\beta_1
\end{pmatrix}(\Pi_1(t)+\Pi_2(t))\Big]\hat{X}(t)\\
&\qquad\qquad+\begin{pmatrix}
0&-\beta_1
\end{pmatrix}\Pi_3(t)\check{X}(t)+\begin{pmatrix}
0&-\beta_1
\end{pmatrix}\Phi(t)-\gamma_1\Big\},\ \ a.e.t\in[0,T],\ a.s.,\\
\bar{v}_2(t)=&-\mu_2^{-1}\Big\{\Big[\begin{pmatrix}
\beta_2&0
\end{pmatrix}(\Pi_1(t)+\Pi_2(t)+\Pi_3(t))+\begin{pmatrix}
0&\beta_2\Pi(t)
\end{pmatrix}\Big]\check{X}(t)\\
&\qquad\qquad+\begin{pmatrix}
\beta_2&0
\end{pmatrix}\check{\Phi}(t)-\gamma_2\Big\},\ \ a.e.t\in[0,T],\ a.s.,
\end{aligned}
\right.
\end{equation}
where $\check{\Phi}(\cdot),\Phi(\cdot)$ satisfy
\begin{equation}\label{solvecheckPhi}
\left\{
\begin{aligned}
d\check{\Phi}(t)&=-\Bigg\{\bigg[(\Pi_1(t)+\Pi_2(t)+\Pi_3(t))\begin{pmatrix}
-\beta_2^2\mu_2^{-1}&-\beta_1^2\mu_1^{-1}\\
-\beta_1^2\mu_1^{-1}&0
\end{pmatrix}\\
&\qquad+\begin{pmatrix}
-\delta-\beta_1^2\Pi(t)\mu_1^{-1}&0\\
-\beta_2^2\Pi(t)\mu_2^{-1}&-\delta-\beta_1^2\Pi(t)\mu_1^{-1}
\end{pmatrix}\bigg]\check{\Phi}(t)\\
&\qquad-(\Pi_1(t)+\Pi_2(t)+\Pi_3(t))\begin{pmatrix}
\beta_1\gamma_1\mu_1^{-1}-\beta_2\gamma_2\mu_2^{-1}\\
0
\end{pmatrix}\\
&\qquad+\begin{pmatrix}
-\theta_2\\
\beta_2\Pi(t)\gamma_2\mu_2^{-1}-\beta_1\Pi(t)\gamma_1\mu_1^{-1}-\theta_1
\end{pmatrix}\Bigg\}dt+\check{\Theta}(t)dW(t),\ t\in[0,T],\\
\check{\Phi}(T)&=0,
\end{aligned}
\right.
\end{equation}
and
\begin{equation}\label{solvePhi}
\left\{
\begin{aligned}
d\Phi(t)&=-\Bigg\{\bigg[(\Pi_1(t)+\Pi_2(t))\begin{pmatrix}
0&-\beta_1^2\mu_1^{-1}\\
-\beta_1^2\mu_1^{-1}&0
\end{pmatrix}\\
&\qquad+\begin{pmatrix}
-\delta-\beta_1^2\Pi(t)\mu_1^{-1}&0\\
0&-\delta-\beta_1^2\Pi(t)\mu_1^{-1}
\end{pmatrix}\bigg]\Phi(t)\\
&\qquad+\bigg[-(\Pi_1(t)+\Pi_2(t)+\Pi_3(t))\begin{pmatrix}
\beta_2^2\mu_2^{-1}&0\\
0&0
\end{pmatrix}\\
&\qquad+\Pi_3(t)\begin{pmatrix}
0&-\beta_1^2\mu_1^{-1}\\
-\beta_1^2\mu_1^{-1}&0
\end{pmatrix}-\begin{pmatrix}
0&0\\
\beta_2^2\Pi(t)\mu_2^{-1}&0
\end{pmatrix}\Bigg]\check{\Phi}(t)\\
&\qquad-(\Pi_1(t)+\Pi_2(t)+\Pi_3(t))\begin{pmatrix}
\beta_1\gamma_1\mu_1^{-1}-\beta_2\gamma_2\mu_2^{-1}\\
0
\end{pmatrix}\\
&\qquad+\begin{pmatrix}
-\theta_2\\
\beta_2\Pi(t)\gamma_2\mu_2^{-1}-\beta_1\Pi(t)\gamma_1\mu_1^{-1}-\theta_1
\end{pmatrix}\Bigg\}dt+\Theta(t)d\widetilde{W}(t),\ t\in[0,T],\\
\Phi(T)&=0,
\end{aligned}
\right.
\end{equation}
respectively, and $\Pi(\cdot), \Pi_1(\cdot), \Pi_2(\cdot), \Pi_3(\cdot)$ satisfy Riccati equations:
\begin{equation}\label{riccatieq---}
\left\{
\begin{aligned}
&\dot{\Pi}(t)-2\delta\Pi(t)-\beta_1^2\mu_1^{-1}\Pi^2(t)-2\kappa_1=0,\ t\in[0,T],\\
&\Pi(T)=2M_1,
\end{aligned}
\right.
\end{equation}
\begin{equation}\label{Pi1---}
\left\{
\begin{aligned}
&\dot{\Pi}_1(t)+\Pi_1(t)\begin{pmatrix}
-\delta&0\\
0&-\delta-\beta_1^2\Pi(t)\mu_1^{-1}
\end{pmatrix}+\begin{pmatrix}
-\delta&0\\
0&-\delta-\beta_1^2\Pi(t)\mu_1^{-1}
\end{pmatrix}\Pi_1(t)\\
&\quad+\Pi_1(t)\begin{pmatrix}
0&-\beta_1^2\mu_1^{-1}\\
-\beta_1^2\mu_1^{-1}&0
\end{pmatrix}\Pi_1(t)+\begin{pmatrix}
-2\kappa_2&0\\
0&0
\end{pmatrix}=0,\ \ t\in[0,T],\\
&\Pi_1(T)=\begin{pmatrix}
2M_2&0\\
0&0
\end{pmatrix},
\end{aligned}
\right.
\end{equation}
{\small\begin{equation}\label{Pi2---}
\left\{
\begin{aligned}
&\dot{\Pi}_2(t)+\Bigg[\begin{pmatrix}
-\delta-\beta_1^2\Pi(t)\mu_1^{-1}&0\\
0&-\delta-\beta_1^2\Pi(t)\mu_1^{-1}
\end{pmatrix}+\Pi_1(t)\begin{pmatrix}
0&-\beta_1^2\mu_1^{-1}\\
-\beta_1^2\mu_1^{-1}&0
\end{pmatrix}\Bigg]\Pi_2(t)\\
&\quad+\Pi_2(t)\Bigg[\begin{pmatrix}
-\delta-\beta_1^2\Pi(t)\mu_1^{-1}&0\\
0&-\delta-\beta_1^2\Pi(t)\mu_1^{-1}
\end{pmatrix}+\begin{pmatrix}
0&-\beta_1^2\mu_1^{-1}\\
-\beta_1^2\mu_1^{-1}&0
\end{pmatrix}\Pi_1(t)\Bigg]\\
&\quad+\Pi_2(t)\begin{pmatrix}
0&-\beta_1^2\mu_1^{-1}\\
-\beta_1^2\mu_1^{-1}&0
\end{pmatrix}\Pi_2(t)+\Pi_1(t)\begin{pmatrix}
-\beta_1^2\Pi(t)\mu_1^{-1}&0\\
0&0
\end{pmatrix}\\
&\quad+\begin{pmatrix}
-\beta_1^2\Pi(t)\mu_1^{-1}&0\\
0&0
\end{pmatrix}\Pi_1(t)=0,\ \ t\in[0,T],\\
&\Pi_2(T)=0,
\end{aligned}
\right.
\end{equation}
\begin{equation}\label{Pi3---}
\left\{
\begin{aligned}
&\dot{\Pi}_3(t)+\Bigg[(\Pi_1(t)+\Pi_2(t))\begin{pmatrix}
0&-\beta_1^2\mu_1^{-1}\\
-\beta_1^2\mu_1^{-1}&0
\end{pmatrix}+\begin{pmatrix}
-\delta-\beta_1^2\Pi(t)\mu_1^{-1}&0\\
0&-\delta-\beta_1^2\Pi(t)\mu_1^{-1}
\end{pmatrix}\Bigg]\Pi_3(t)\\
&\quad+\Pi_3(t)\Bigg[\begin{pmatrix}
-\delta-\beta_1^2\Pi(t)\mu_1^{-1}&0\\
0&-\delta-\beta_1^2\Pi(t)\mu_1^{-1}
\end{pmatrix}+\begin{pmatrix}
0&-\beta_1^2\mu_1^{-1}\\
-\beta_1^2\mu_1^{-1}&0
\end{pmatrix}(\Pi_1(t)+\Pi_2(t))\Bigg]\\
&\quad-(\Pi_1(t)+\Pi_2(t)+\Pi_3(t))\begin{pmatrix}
\beta_2^2\mu_2^{-1}&0\\
0&0
\end{pmatrix}(\Pi_1(t)+\Pi_2(t)+\Pi_3(t))\\
&\quad-\begin{pmatrix}
0&0\\
\beta_2^2\Pi(t)\mu_2^{-1}&0
\end{pmatrix}(\Pi_1(t)+\Pi_2(t)+\Pi_3(t))-(\Pi_1(t)+\Pi_2(t)+\Pi_3(t))\begin{pmatrix}
0&\beta_2^2\Pi(t)\mu_2^{-1}\\
0&0
\end{pmatrix}\\
&\quad+\Pi_3(t)\begin{pmatrix}
0&-\beta_1^2\mu_1^{-1}\\
-\beta_1^2\mu_1^{-1}&0
\end{pmatrix}\Pi_3(t)-\begin{pmatrix}
0&0\\
0&\beta_2^2\Pi^2(t)\mu_2^{-1}
\end{pmatrix}=0,\ \ t\in[0,T],\\
&\Pi_3(T)=0,
\end{aligned}
\right.
\end{equation}}
respectively. Solving \eqref{solvecheckPhi}, we can get the explicit solutions $(\check{\Phi}_1,0)$ and $(\check{\Phi}_2,0)$ as follows:
{\small\begin{equation}
\begin{aligned}
\check{\Phi}_1(t)&=\int_t^T\exp\bigg\{\int_t^s\Big[(\Pi_1(1,1)+\Pi_2(1,1)+\Pi_3(1,1))(-\beta_2^2\mu_2^{-1})+(\Pi_1(1,2)+\Pi_2(1,2)\\
                 &\qquad+\Pi_3(1,2))(-\beta_1^2\mu_1^{-1})-\delta-\beta_1^2\Pi\mu_1^{-1}\Big]dr\bigg\}\Big[(\Pi_1(1,1)+\Pi_2(1,1)+\Pi_3(1,1))\\
                 &\qquad\times(-\beta_1^2\mu_1^{-1})\check{\Phi}_2+(\Pi_1(1,1)+\Pi_2(1,1)+\Pi_3(1,1))(-\beta_1\mu_1^{-1}\gamma_1+\beta_2\mu_2^{-1}\gamma_2)+\bar{l}\Big]ds,\\
\check{\Phi}_2(t)&=\int_t^T\exp\bigg\{\int_t^s\Big[(\Pi_1(2,1)+\Pi_2(2,1)+\Pi_3(2,1))(-\beta_1^2\mu_1^{-1})-\delta-\beta_1^2\Pi\mu_1^{-1}\Big]dr\bigg\}\\
                 &\qquad\times\Big[\big[(\Pi_1(2,1)+\Pi_2(2,1)+\Pi_3(2,1))(-\beta_2^2\mu_2^{-1})+(\Pi_1(2,2)+\Pi_2(2,2)\\
                 &\qquad+\Pi_3(2,2))(-\beta_1^2\mu_1^{-1})-\beta_2^2\Pi\mu_2^{-1}\big]\check{\Phi}_1+(\Pi_1(2,1)+\Pi_2(2,1)+\Pi_3(2,1))\\
                 &\qquad\times(-\beta_1\mu_1^{-1}\gamma_1+\beta_2\mu_2^{-1}\gamma_2)-\beta_1\Pi\mu_1^{-1}\gamma_1+\beta_2\Pi\mu_2^{-1}\gamma_2-\theta_1\Big]ds.
\end{aligned}
\end{equation}}
and solving \eqref{solvePhi}, we can get the explicit solution $(\Phi_1,0)$ and $(\Phi_2,0)$ as follows:
{\small\begin{equation}
\begin{aligned}
\Phi_1(t)&=\int_t^T\exp\bigg\{\int_t^s\Big[(\Pi_1(1,2)+\Pi_2(1,2))(-\beta_1^2\mu_1^{-1})-\delta-\beta_1^2\Pi\mu_1^{-1}\Big]dr\bigg\}\Big[(\Pi_1(1,1)+\Pi_2(1,1))\\
         &\quad\times(-\beta_1^2\mu_1^{-1})\Phi_2+\big[-(\Pi_1(1,1)+\Pi_2(1,1)+\Pi_3(1,1))\beta_2^2\mu_2^{-1}+\Pi_3(1,2)(-\beta_1^2\mu_1^{-1})\big]\check{\Phi}_1\\
         &\quad+\Pi_3(1,1)(-\beta_1^2\mu_1^{-1})\check{\Phi}_2+(\Pi_1(1,1)+\Pi_2(1,1)+\Pi_3(1,1))(-\beta_1\mu_1^{-1}\gamma_1+\beta_2\mu_2^{-1}\gamma_2)-\theta_2\Big]ds,\\
\Phi_2(t)&=\int_t^T\exp\bigg\{\int_t^s\Big[(\Pi_1(2,1)+\Pi_2(2,1))(-\beta_1^2\mu_1^{-1})-\delta-\beta_1^2\Pi\mu_1^{-1}\Big]dr\bigg\}\Big[(\Pi_1(2,2)\\
         &\qquad+\Pi_2(2,2))(-\beta_1^2\mu_1^{-1})\Phi_1+\big[-(\Pi_1(2,1)+\Pi_2(2,1)+\Pi_3(2,1))\beta_2^2\mu_2^{-1}\\
         &\qquad+\Pi_3(2,2)(-\beta_1^2\mu_1^{-1})-\beta_2^2\Pi\mu_2^{-1}\big]\check{\Phi}_1+\Pi_3(2,1)(-\beta_1^2\mu_1^{-1})\check{\Phi}_2+(\Pi_1(2,1)+\Pi_2(2,1)\\
         &\qquad+\Pi_3(2,1))(-\beta_1\mu_1^{-1}\gamma_1+\beta_2\mu_2^{-1}\gamma_2)-\beta_1\Pi\mu_1^{-1}\gamma_1+\beta_2\Pi\mu_2^{-1}\gamma_2-\theta_1\Big]ds,
\end{aligned}
\end{equation}}
where $\check{\Phi}(\cdot)\equiv\begin{pmatrix}
\check{\Phi}_1(\cdot)\\
\check{\Phi}_2(\cdot)
\end{pmatrix},\Phi(\cdot)\equiv\begin{pmatrix}
\Phi_1(\cdot)\\
\Phi_2(\cdot)
\end{pmatrix}$ and $\Pi_{i}\equiv\begin{pmatrix}
\Pi_{i}(1,1)&\Pi_{i}(1,2)\\
\Pi_{i}(2,1)&\Pi_{i}(2,2)
\end{pmatrix}$, $i=1,2,3$.
Finally, $\check{X}(\cdot)$ and $\hat{X}(\cdot)$ satisfy
\begin{equation}
\left\{
\begin{aligned}
d\check{X}(t)&=\Bigg\{\Bigg[\begin{pmatrix}
-\delta-\beta_1^2\Pi(t)\mu_1^{-1}&-\beta_2^2\Pi(t)\mu_2^{-1}\\
0&-\delta-\beta_1^2\Pi(t)\mu_1^{-1}
\end{pmatrix}+\begin{pmatrix}
-\beta_2^2\mu_2^{-1}&-\beta_1^2\mu_1^{-1}\\
-\beta_1^2\mu_1^{-1}&0
\end{pmatrix}\\
&\qquad\times(\Pi_1(t)+\Pi_2(t)+\Pi_3(t))\Bigg]\check{X}(t)+\begin{pmatrix}
-\beta_2^2\mu_2^{-1}&-\beta_1^2\mu_1^{-1}\\
-\beta_1^2\mu_1^{-1}&0
\end{pmatrix}\check{\Phi}(t)\\
&\qquad+\begin{pmatrix}
\beta_2\gamma_2\mu_2^{-1}-\beta_1\gamma_1\mu_1^{-1}\\
0
\end{pmatrix}\Bigg\}dt+\begin{pmatrix}
\sigma\\
0
\end{pmatrix}d{W}(t),\ t\in[0,T],\\
\check{X}(0)&=\begin{pmatrix}
x_0\\
0
\end{pmatrix},
\end{aligned}
\right.
\end{equation}
and
\begin{equation}
\left\{
\begin{aligned}
d\hat{X}(t)&=\Bigg\{\Bigg[\begin{pmatrix}
-\delta-\beta_1^2\Pi(t)\mu_1^{-1}&0\\
0&-\delta-\beta_1^2\Pi(t)\mu_1^{-1}
\end{pmatrix}\\
&\qquad+\begin{pmatrix}
0&-\beta_1^2\mu_1^{-1}\\
-\beta_1^2\mu_1^{-1}&0
\end{pmatrix}(\Pi_1(t)+\Pi_2(t))\Bigg]\hat{X}(t)\\
&\qquad+\Bigg[\begin{pmatrix}
0&-\beta_1^2\mu_1^{-1}\\
-\beta_1^2\mu_1^{-1}&0
\end{pmatrix}\Pi_3(t)-\begin{pmatrix}
\beta_2^2\mu_2^{-1}&0\\
0&0
\end{pmatrix}(\Pi_1(t)+\Pi_2(t)+\Pi_3(t))\\
&\qquad-\begin{pmatrix}
0&\beta_2^2\Pi(t)\mu_2^{-1}\\
0&0
\end{pmatrix}\Bigg]\check{X}(t)+\begin{pmatrix}
0&-\beta_1^2\mu_1^{-1}\\
-\beta_1^2\mu_1^{-1}&0
\end{pmatrix}\Phi(t)\\
&\qquad-\begin{pmatrix}
\beta_2^2\mu_2^{-1}&0\\
0&0
\end{pmatrix}\check{\Phi}(t)+\begin{pmatrix}
\beta_2\gamma_2\mu_2^{-1}-\beta_1\gamma_1\mu_1^{-1}\\
0
\end{pmatrix}\Bigg\}dt\\
&\quad+\Bigg[\begin{pmatrix}
\sigma\\
0
\end{pmatrix}+\mathcal{P}\begin{pmatrix}
f_1(t)\\
0
\end{pmatrix}\Bigg]d\widetilde{W}(t),\ t\in[0,T],\\
\hat{X}(0)&=\begin{pmatrix}
x_0\\
0
\end{pmatrix},
\end{aligned}
\right.
\end{equation}
respectively.

\subsubsection{Explicit Numerical Solutions}

To be more intuitive, we would like to give some numerical simulations with certain particular coefficients and some figures could be plotted to illustrate our reasonable results in reality and verify the effectiveness of our theoretical results. First, let $\beta_1=0.2,\beta_2=0.4,\delta=0.5,\gamma_1=0.6,\gamma_2=0.5,\sigma=0.2,\bar{\sigma}=0.4,\kappa_1=0.6,\kappa_2=0.5,\theta_1=0.4,\theta_2=0.6,\mu_1=0.3,\mu_2=0.5,M_1=0.8,M_2=1,x_0=0.01,f_1=0.6$. Then, applying Euler's method, we derive the numerical solutions of $\Pi(\cdot),\Pi_1(\cdot),\Pi_2(\cdot),\Pi_3(\cdot)$ and $\mathcal{P}(\cdot)$ in Figure 1.

\begin{figure}[H]
\centering
\subfigure[The trajectory of $\Pi(\cdot)$.]
{
\includegraphics[width=3in]{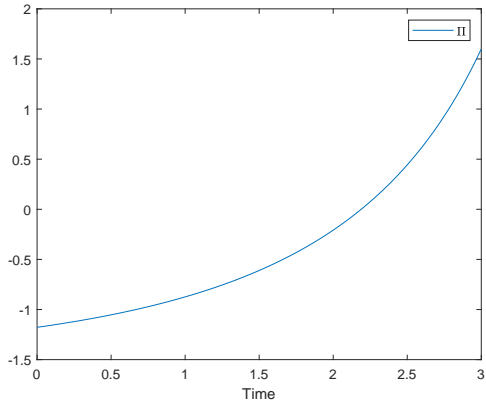}
}
\subfigure[The trajectory of $\Pi_1(\cdot)$.]
{
\includegraphics[width=3in]{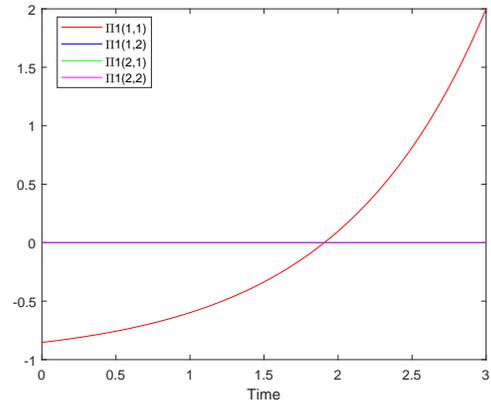}
}
\subfigure[The trajectory of $\Pi_2(\cdot)$.]
{
\includegraphics[width=3in]{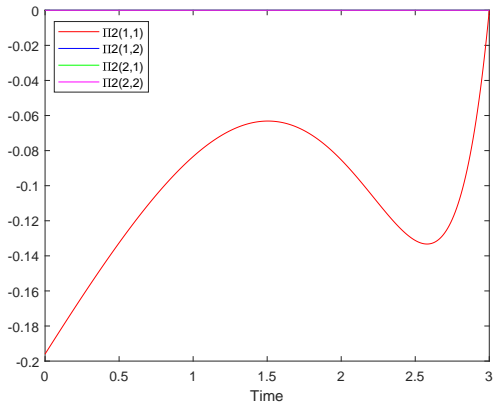}
}
\subfigure[The trajectory of $\Pi_3(\cdot)$.]
{
\includegraphics[width=3in]{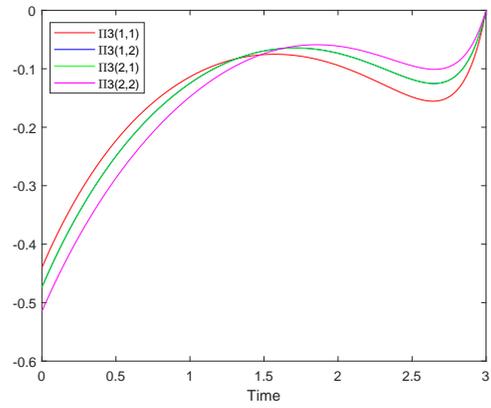}
}
\subfigure[The trajectory of $\mathcal{P}(\cdot)$.]
{
\includegraphics[width=3in]{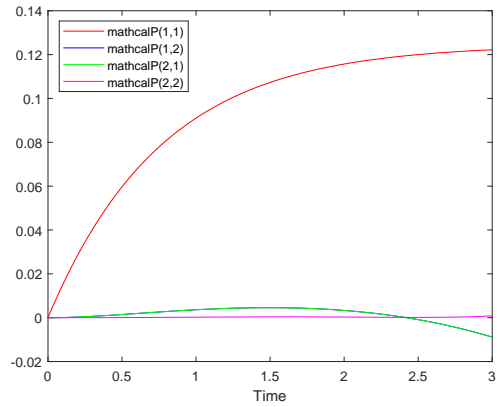}
}
\caption{The trajectories of Riccati equations.}
\end{figure}

In detail, $\Pi(\cdot)$ is one-dimensional function, and $\Pi_1(\cdot), \Pi_2(\cdot), \Pi_3(\cdot)$ and $\mathcal{P}$ are four symmetric matrix-valued functions, and we have $\Pi_{i}(1,2)=\Pi_{i}(2,1)$ and $\mathcal{P}(1,2)=\mathcal{P}(2,1)$, where $\Pi_{i}(1,2)$ and $\mathcal{P}(1,2)$ denote the value of the 1 row and the 2 column of $(\Pi_{i},i=1,2,3)$ and $\mathcal{P}$, respectively. Especially, we have $\Pi_{i}(1,2)=\Pi_{i}(2,1)=\Pi_{i}(2,2)=0, i=1,2$.

Similarly, we plot the trajectories of $\check{X}(\cdot), \hat{X}(\cdot), X(\cdot)$ and $(\bar{v}_1(\cdot),\bar{v}_2(\cdot))$ in Figures 2 and 3, respectively. More concretely, $\check{X}(\cdot), \hat{X}(\cdot), X(\cdot)$ are all $2\times1$ vector-valued processes with the initial value $\check{X}(0)=\hat{X}(0)=X(0)=\begin{pmatrix}
x_0\\
0
\end{pmatrix}$. Then, $\bar{v}_1(\cdot), \bar{v}_2(\cdot)$ are both one-dimensional processes.

\begin{figure}[H]
\centering
\includegraphics[height=7.5cm,width=14cm]{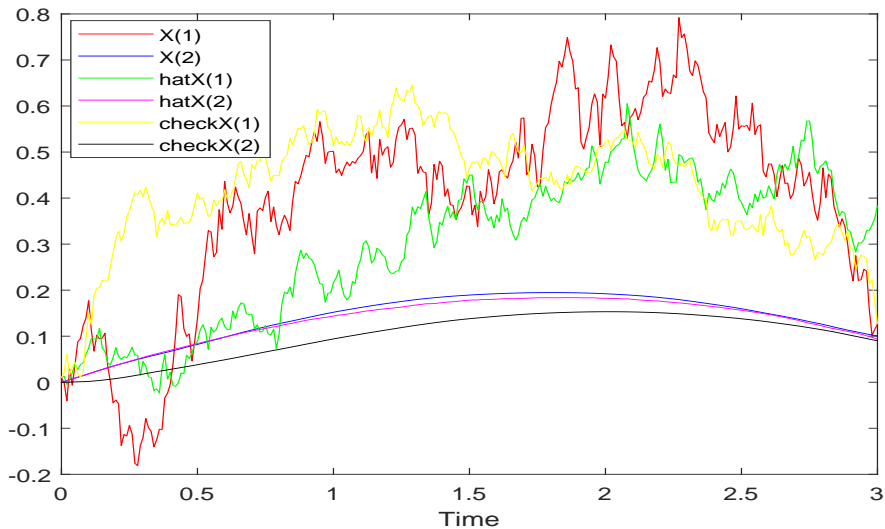}
\caption{The trajectory of $\check{X}(\cdot),\hat{X}(\cdot),X(\cdot)$.}
\end{figure}

In Figure 2, the dynamic curves of the first components of all the state process $X(\cdot)$ and the state filtering processes $\check{X}(\cdot),\hat{X}(\cdot)$: $X(1)$ and $\check{X}(1),\hat{X}(1)$, have tendencies for sharp fluctuations and we give one reasonable explanation. Especially, for the brand image $X(\cdot)$, observing the red line, the decreasing trends and increasing trends appear alternately due the uncertain factor and the asymmetric information in the market. Furthermore, we could illustrate that once the retailer (the follower) does much local promotional efforts depending on his/her more available information at the beginning, which, sometimes, causes the nearly negative value of brand image in the graph. After which, the manufacturer (the leader), although only acquires less information, still notices the negative effect on the brand image, then he/she will strengthen the corresponding advertising effort (such as national advertising and radio) to save the brand image as much as possible for the need of macro control, which leads to a positive value of image. The small fluctuations (the phenomenon of up and down) partially shows the game feature between the manufacturer and the retailer, which, finally, guarantee that the value of brand image at terminal time is positive. The interesting phenomenon, to some degree, can be explained in the following with the optimal Stackelberg equilibrium points $(\bar{v}_1,\bar{v}_2)$.

\begin{figure}[H]
\centering
\includegraphics[height=7.5cm,width=14cm]{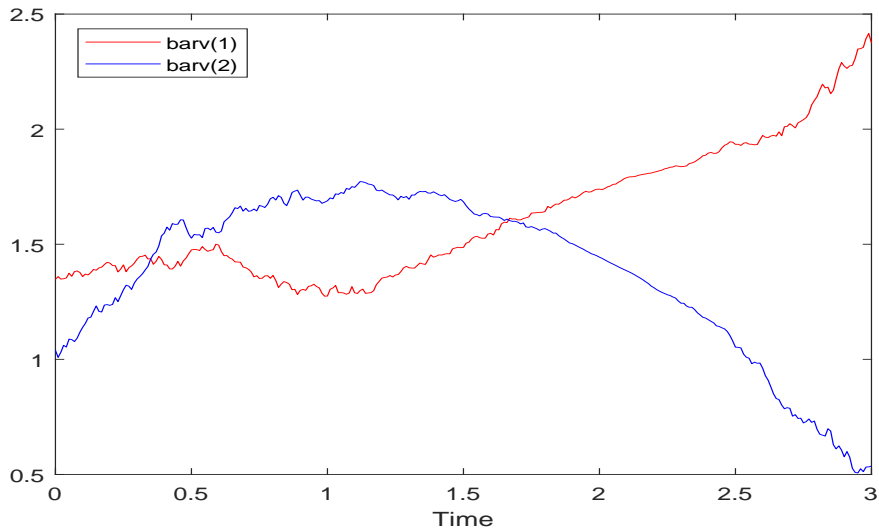}
\caption{The trajectory of $\bar{v}_1(\cdot),\bar{v}_2(\cdot)$.}
\end{figure}

In Figure 3, intuitively, the value of advertising effort of the retailer (the follower) is larger than that of the the manufacturer (the leader) at the beginning, The phenomenon can be explained as follows, compared to the manufacturer (the leader), the retailer (the follower) possesses more information resource (such as the long-term customer relationship, professional selling technique and geographical advantage), naturally, he/she would like to input more energy into the product, and prefer to increase more promotional effort to make advertising. After a period of time, the manufacturer observes that the local advertising effort of the retailer has eroded the brand image, and then takes timely remedy to make up the damage, therefore, the manufacturer, to some extent, keeps strengthening his/her advertising effort which even exceeds the value of the retailer's promotional effort for a period of time.  However, it is impossible for the manufacturer to keep paying too much attentions on just one product, because he/she has various types of products to manufacture and propaganda. Therefore, the manufacture usually pour less energy into the advertising effort of single product. Finally, the value of the manufacture's advertising effort become decreasing, which causes that it is again less than that of the retailer.

Next, we will make some analysis about the effect of some practical coefficients on the brand image $x(\cdot)$, the follower's local promotion effort rate $v_1(\cdot)$ and the leader's advertising rate $v_2(\cdot)$.

In our setting, the manufacturer's advertising effort is helpful for the brand image, however, the retailer's promotion has a negative influence on the brand image, and the positive parameters $\beta_2,\beta_1$ measures the impact of $v_2,v_1$ on $x$, respectively. So we consider the relative effect between $\beta_2$ and $\beta_1$ on the leader's advertising rate $v_2$ and follower's local promotion effort $v_1$, that is, we study the changes of $v_2,v_1$ when $\beta_2$ takes these values in $[0.04,0.14,0.24,0.34,0.44]$.

\begin{figure}[H]
\centering
\subfigure[The impact of $\beta_2$ on leader's advertising rate]
{
\includegraphics[width=3in]{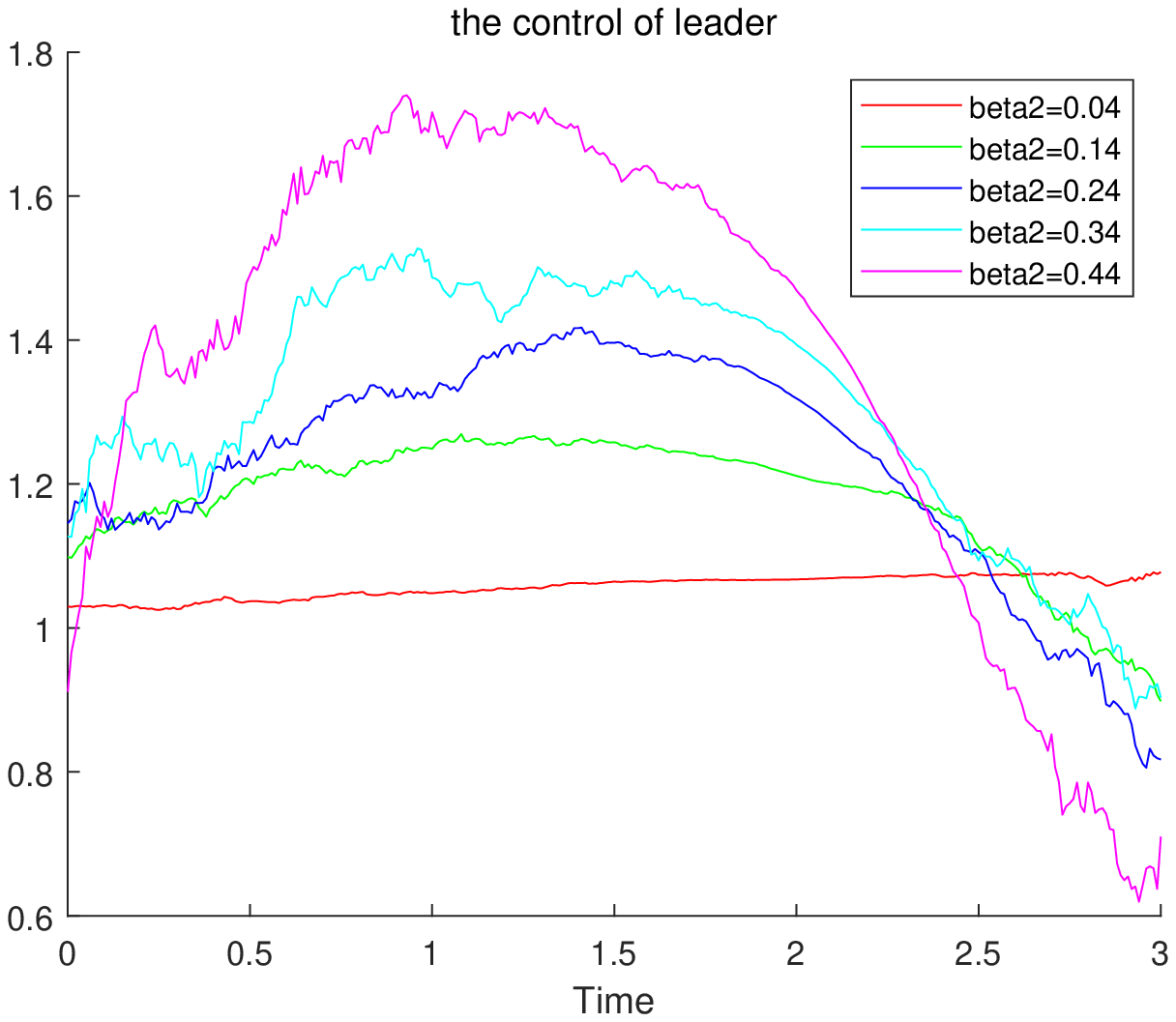}
}
\subfigure[The impact of $\beta_2$ on follower's local promotion effort]
{
\includegraphics[width=3in]{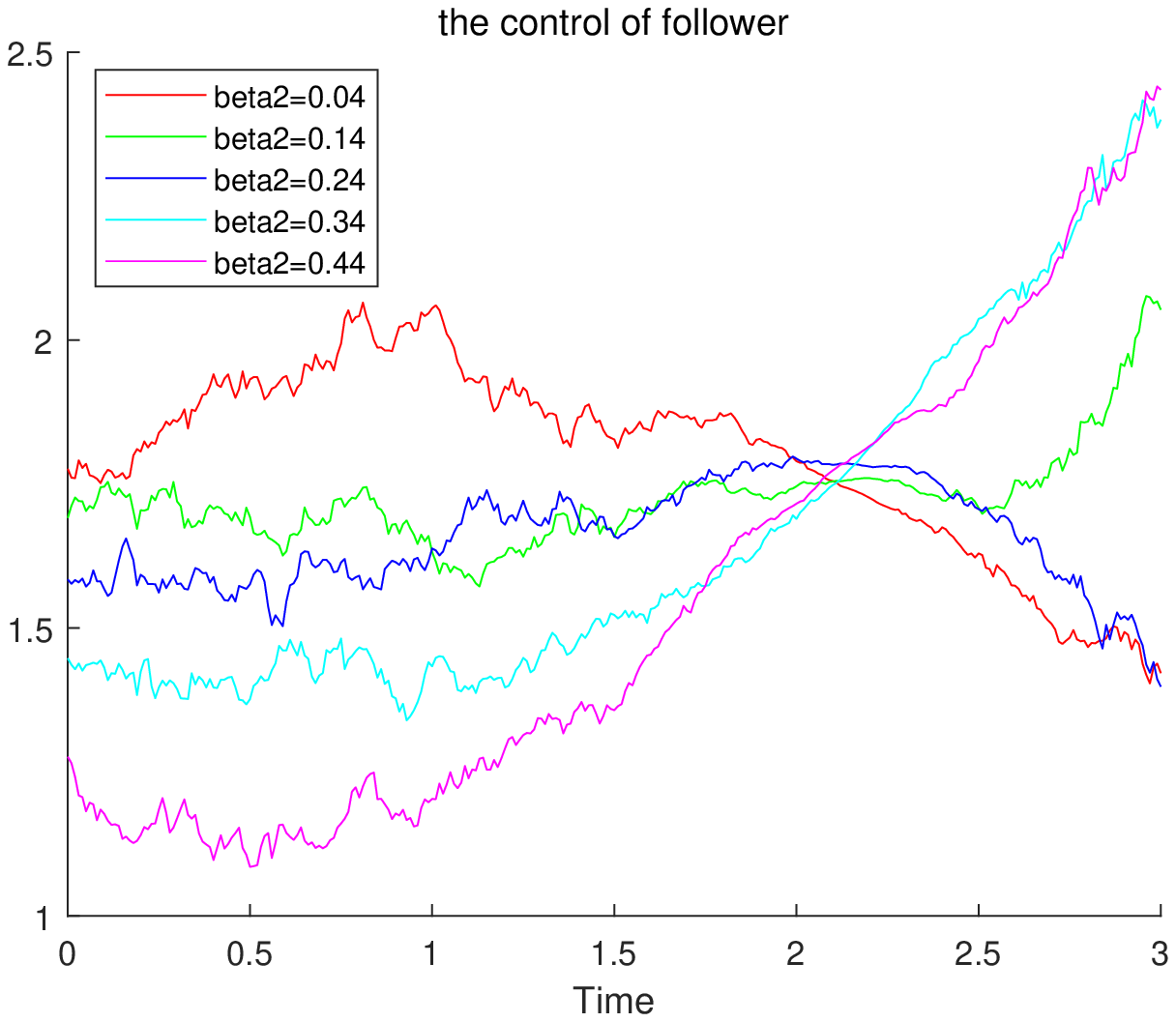}
}
\caption{The relationship between $v_1,v_2$ and $\beta_2$}
\end{figure}

In Figure 4, we notice that, in the left subfigure, the parameter $\beta_2$ has a positive impact on the leader's advertising rate, that is, large value of $\beta_2$ means that the advertising rate the leader controls occupies a greater proportion in the effect on the brand image compared to the fixed parameter $\beta_1$ of the retailer. Therefore, it is natural for the manufacturer to enhance propaganda by increase the advertising rate in national media, and we can see the purple line is at the top during the overall trend. It is worth noting that all the lines show the decrease trend, which, as we analyzed in the dynamic game, is because it is impossible for the manufacture to spend all the energy on just one product. However, for the retailer, we encounter the opposite situation. In the right subfigure, we find the red line becomes a highest one where the value of $\beta_2$ is relatively small compared to $\beta_1$. It causes that manufacturer's advertising rate has a weak influence on the brand image, that is, he keeps a weak advertising effort. But in this case, it is important for them to advertise the product as much as they can, which is the final goal. So the retailer have to shoulder the task and strengthen his local advertising effort by means of his information. Obviously, the final reduction is due to his damage on brand image.

In \eqref{rcf}, $\mu_1$ is the cost coefficient of $v_1$, which also represents the cost weight on $v_1$. We then study the changes of $v_2,v_1$ when $\mu_1$ takes these values in $[0.1,0.3,0.5,0.7,0.9,1.1]$, that is, we impose different levels of stress on the cost of retailer, and observe the tendency.

\begin{figure}[H]
\centering
\subfigure[The impact of $\mu_1$ on leader's advertising rate]
{
\includegraphics[width=3in]{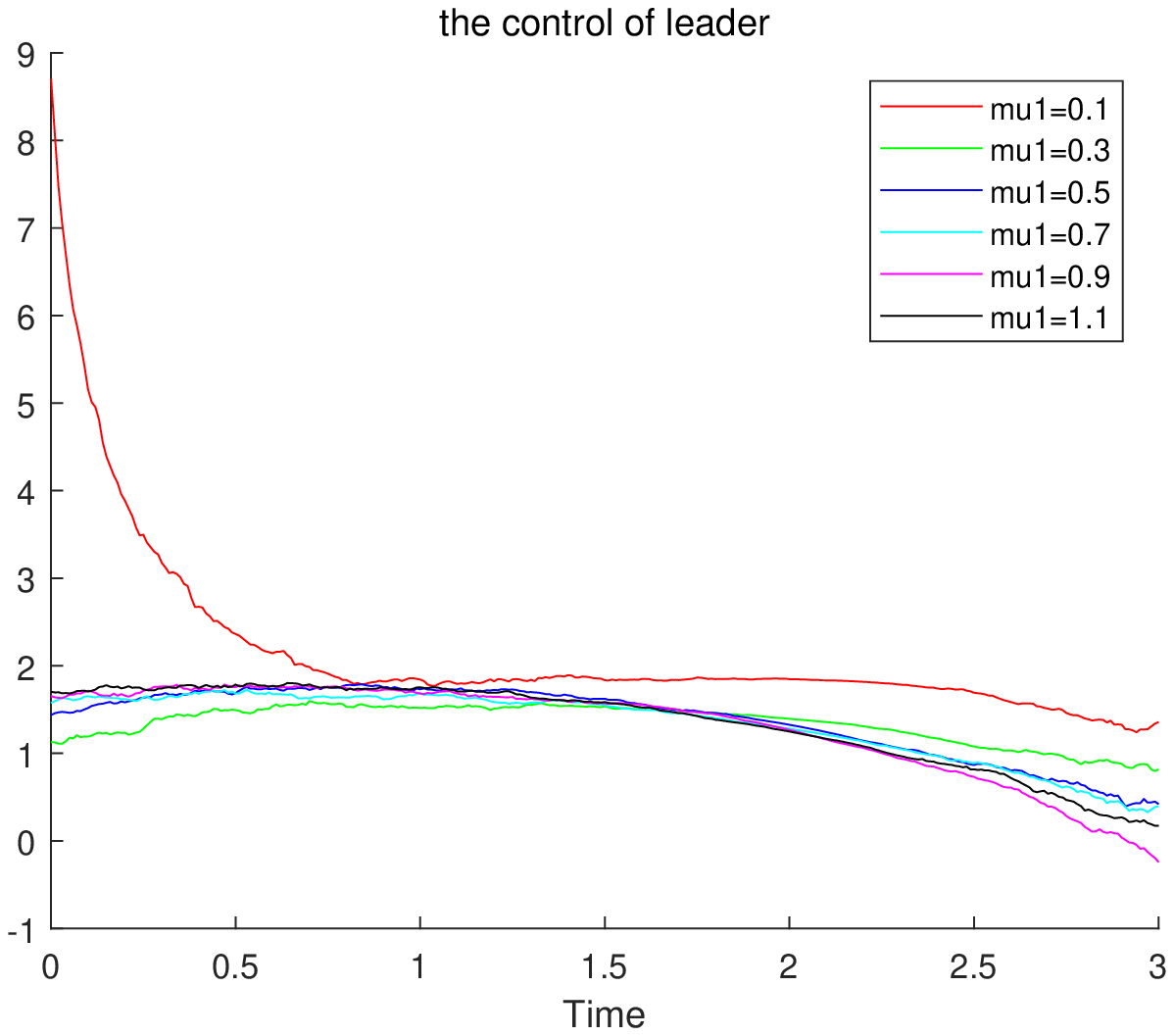}
}
\subfigure[The impact of $\mu_1$ on follower's local promotion effort]
{
\includegraphics[width=3in]{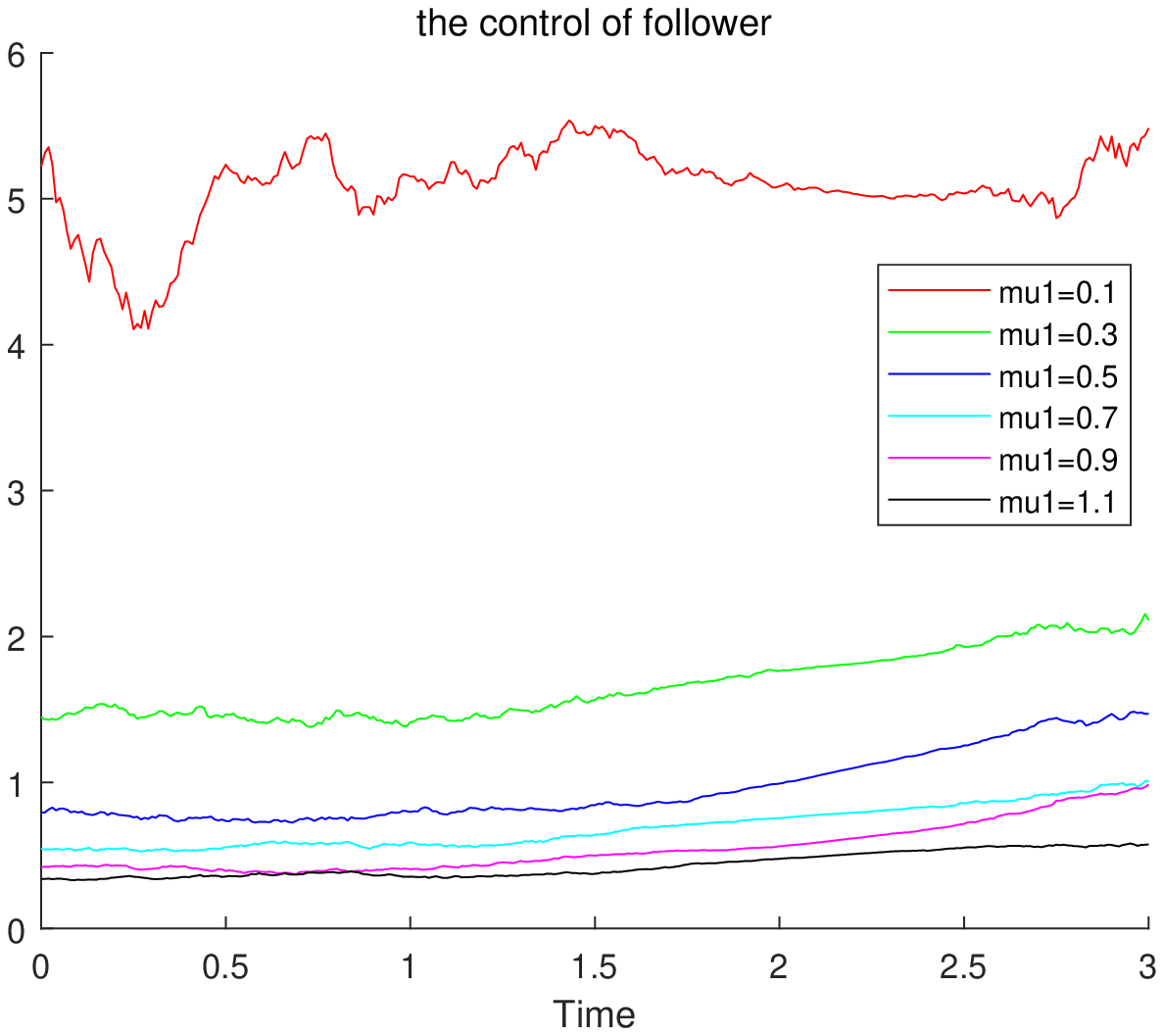}
}
\caption{The relationship between $v_1,v_2$ and $\mu_1$}
\end{figure}

In Figure 5, we study the quadratic cost coefficient $\mu_1$'s effect on the promotion $v_1$ and advertising effort $v_2$. First, in the left subfigure, we can see most lines hold a relatively steady and similar trend when the value of $\mu_1$ varies from 0.3 to 1.1, except for red line corresponding to $\mu_1=0.1$. Although the red line show a strange tendency, we can give a reasonable explanation. For the common goal which is to increase sales revenue and decrease cost, in the case of $\mu_1=0.1$, the cost coefficient of retailer's local promotion effort is the smallest. The manufacturer knows it is benefit for the retailer to strengthen the advertising effort, which can guarantee the ideal sales revenue, so he choose to reduce the strength of advertising effort in a period of time, but hold a stable advertising rate in the later period of time because he have to control that the brand image is not too negatively affected by the retailer. In the right subfigure, it is evident that the lower the cost coefficient, the stronger the retailer's promotion. Furthermore, the smallest cost coefficient stimulate the retailer's promotion which is suppressed as the cost coefficient goes up. Finally, it drops to a lowest level that the retailer have to receive and keep unchanged.

We finally study the changes of $x,\hat{x},\check{x}$ when $\mu_1$ takes these values in $[0.1,0.3,0.5,0.7,0.9,1.1]$, that is, compared to the quadratic cost coefficient of the manufacturer, we concern the fluctuations of brand image itself, brand image observed by the retailer and that observed by the manufacturer, respectively, if the retailer's cost weight increases or decreases.
\begin{figure}[H]
\centering
\subfigure[The impact of $\mu_1$ on $\check{x}$]
{
\includegraphics[width=3in]{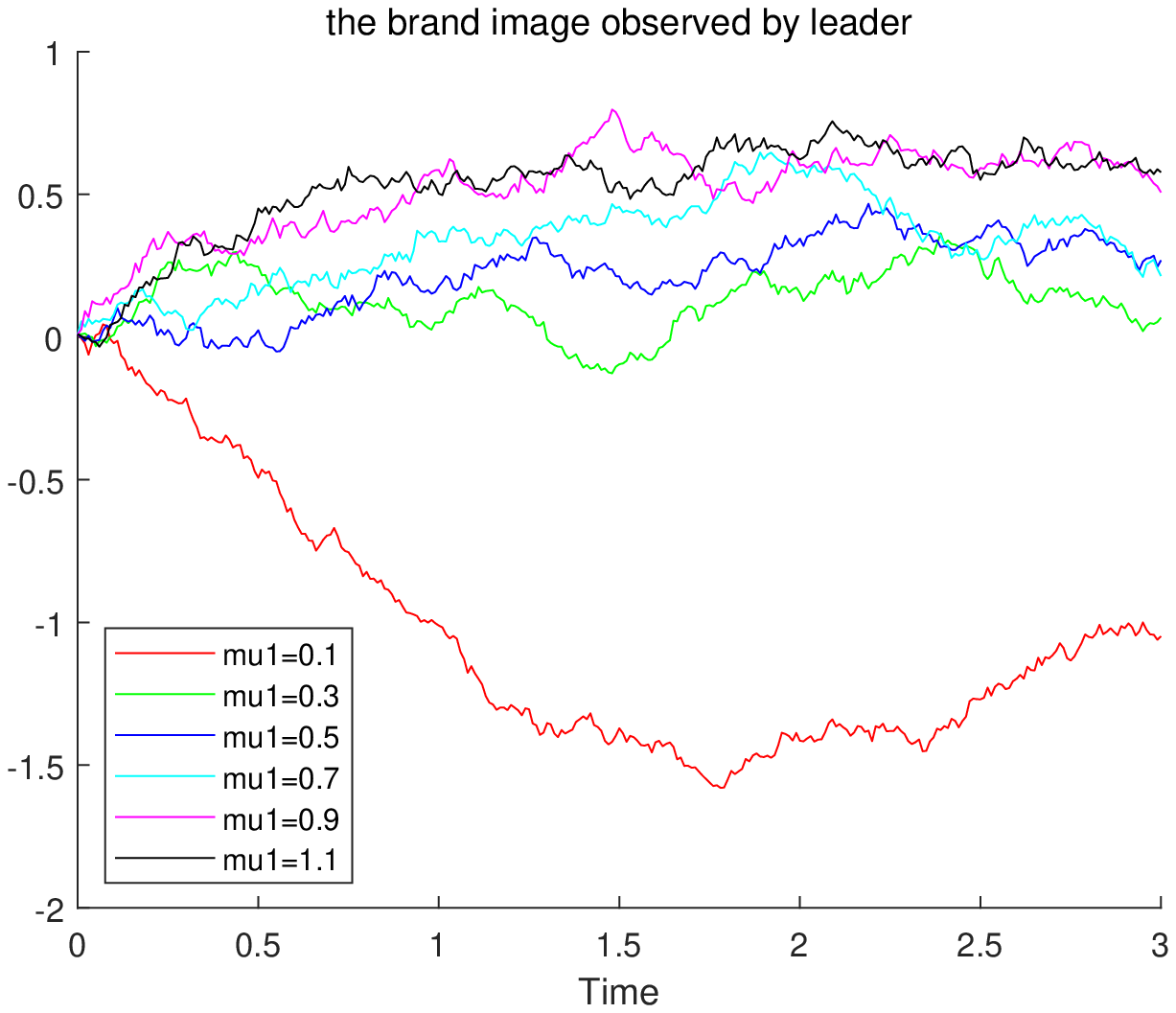}
}
\subfigure[The impact of $\mu_1$ on $\hat{x}$]
{
\includegraphics[width=3in]{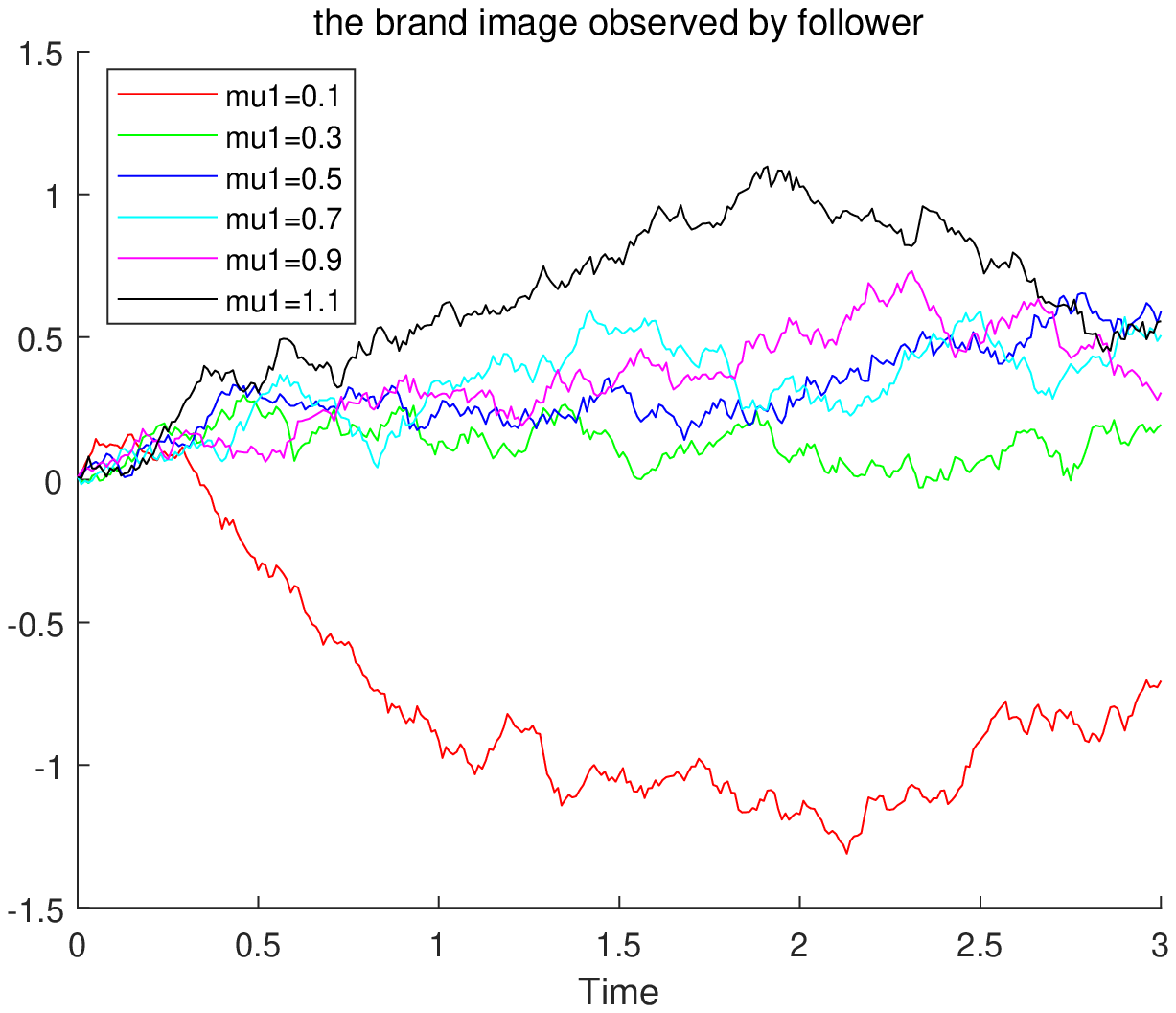}
}
\subfigure[The impact of $\mu_1$ on $x$]
{
\includegraphics[width=3in]{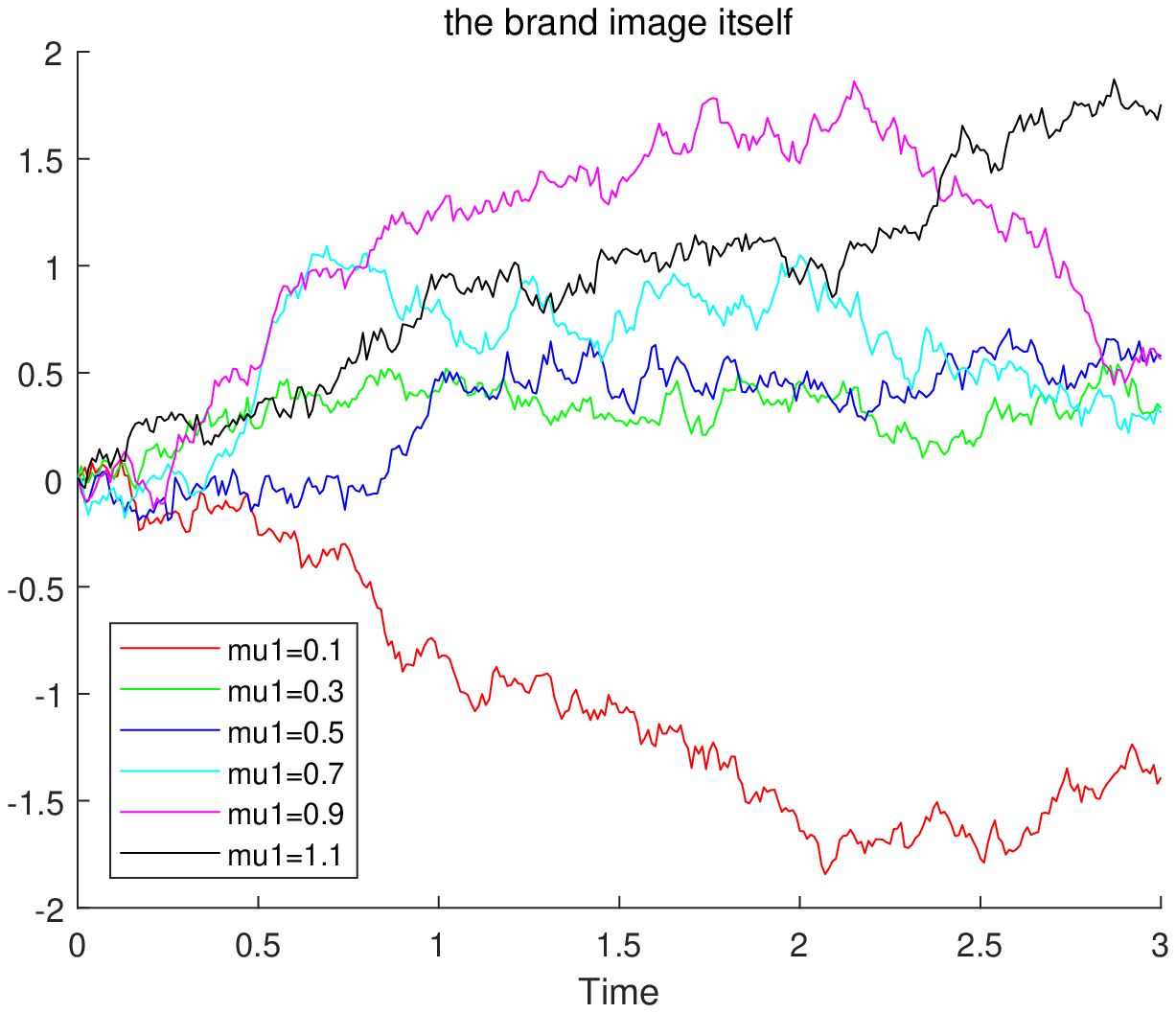}
}
\caption{The relationship between $\check{x},\hat{x},x$ and $\mu_1$}
\end{figure}

In Figure 6, except for the case of $\mu_1=0.1$, the other cases all show that the brand image have a normal fluctuation influenced by the manufacturer's advertising rate and the retailer's local promotion, and present a positive effect on the brand image in general. But when $\mu_1$ becomes sufficiently small, which is a situation beyond the control, the retailer's negative effect on the brand image will be magnified, which illustrate the tendency of red line in all three subfigures.

\section{Concluding remarks}

In this paper, we have studied an LQ partially observed Stackelberg differential game, different from the existing literatures \cite{SWX17}, \cite{WWZ20}, where both the leader and the follower can only observe the state process through their own observation process. Different from \cite{LFCM19}, \cite{LDFCM21} and \cite{ZS20}, for the technical need, the information available to the leader is less than that available to the follower. In order to obtain the necessary and sufficient conditions of the Stackelberg equilibrium point, we mainly utilize the state decomposition and backward separation techniques to solve the circular dependence between the observation process $Y_1(Y_2)$ and control variable $v_1(v_2$, respectively). First, the LQ partially observed optimal control problem of the follower is solved, and we get the state estimation feedback of optimal control via the FBSDFEs and one Riccati equation. Then, we introduce a new LQ partially observed optimal control problem of the leader where the state equation is not only driven by the $\mathcal{F}_t$-standard Brownian motions $W, \bar{W}$, but also by another $\mathcal{F}_t^{Y_1}$-Brownian motion $\widetilde{W}$ (innovation process) which is not independent of $W$ or $\bar{W}$, whose unique solvability can be obtained by a new combined idea. Meanwhile, a kind of fully-coupled FBSDE with filtering is studied as a by-product. Then we get the feedback representation form via an SDFE, a semi-martingale and three high-dimensional Riccati equations. We also study a special partially observed case and give the explicit expression of Stackelberg equilibrium points. Finally, we solve a dynamic cooperative advertising problem with asymmetric information and verify the effectiveness of proposed theoretical result by numerical simulation, moreover, we also discuss the relationship between the Stackelberg equilibrium points, state estimate and some practical parameters.

Possible extensions to the Stackelberg stochastic differential game with mixed leadership \cite{BBS10}, \cite{BCCSSY19}, multiple followers \cite{MX15}, \cite{WWZ20} and delayed systems \cite{XSZ18}, \cite{XZ16} promise to be interesting research topics. We will consider these challenging topics in our future work.


\begin{thebibliography}{0}

\bibitem{BO98}T. Ba\c{s}ar, G. J. Olsder, \emph{Dynamic Noncooperative Game Theory, 2nd Edition}, SIAM, Philadelphia, 1998.

\bibitem{BBS10}T. Ba\c{s}ar, A. Bensoussan, and S. P. Sethi, Differential games with mixed leadership: The open-loop solution. \emph{Appl. Math. Comput.}, {\bf 217}(3), 972-979, 2010.

\bibitem{Ben92}A. Bensoussan, \emph{Stochastic Control of Partially Observable Systems}, Cambridge University Press, 1992.

\bibitem{BV75}A. Bensoussan, M. Viot, Optimal control of stochastic linear distributed parameter systems. \emph{SIAM J. Control Optim.}, {\bf 13}(4), 904-926, 1975.

\bibitem{BCCSSY19}A. Bensoussan, S. K. Chen, A. Chutani, S. P. Sethi, C. C. Siu, and S. C. P. Yam, Feedback Stackelberg-Nash equilibrium in mixed leadeship games with an application to cooperative advertising. \emph{SIAM J. Control Optim.}, {\bf 57}(5), 3413-3444, 2019.

\bibitem{BCS15}A. Bensoussan, S. K. Chen, and S. P. Sethi, The maximum principle for global solutions of stochastic Stackelberg differential games. \emph{SIAM J. Control Optim.}, {\bf 53}(4), 1956-1981, 2015.

\bibitem{CA76}D. Castanon, M. Athans, On stochastic dynamic Stackelberg strategies. \emph{Automatica}, {\bf 12}(2), 177-183, 1976.

\bibitem{CS18}L. Chen, Y. Shen, On a new paradigm of optimal reinsurance: A stochastic Stackelberg differential game between an insurer and a reinsurer. \emph{ASTIN Bulletin}, {\bf 48}(2), 905-960, 2018.

\bibitem{CZ13}J. Cvitani\'{c}, J. F. Zhang, \emph{Contract Theory in Continuous-Time Models}, Springer-Verlag, Berlin, 2013.

\bibitem{DIM92}S. Davis, J. J. Inman, and L. McAlister, Promotion has a negative effect on brand evaluations-Or does it? Additional disconfirming evidence. \emph{J. Mark. Res.}, {\bf 29}(1), 143-148, 1992.

\bibitem{DW19}K. Du, Z. Wu, Linear-quadratic Stackelberg game for mean-field backward stochastic differential system and application. \emph{Math. Prob. Engin.}, {\bf 2019}, Article ID 1798585, 17 pages.

\bibitem{FHH20}X. W. Feng, Y. Hu, and J. H. Huang, Backward Stackelberg differential game with constraints: a mixed terminal-perturbation and linear-quadratic approach, \emph{arXiv:2005.11872v1}, May 2020.

\bibitem{HSW20}J. H. Huang, K. H. Si, and Z. Wu, Linear-quadratic mixed Stackelberg-Nash stochastic differential game with major-minor agents. \emph{Appl. Math. Opt.}, https://doi.org/10.1007/s00245-020-09713-z.

\bibitem{HWX09}J. H. Huang, G. C. Wang, and J. Xiong, A maximum principle for partial information backward stochastic control problems with applications. \emph{SIAM J. Control Optim.}, {\bf 48}(4), 2106-2117, 2009.

\bibitem{JSZ00}S. Jorgensen, S. P. Sigue, and G. Zaccour, Dynamic cooperative advertising in a channel. \emph{J. Retail.}, {\bf 76}(1), 71-92, 2000.

\bibitem{JTZ01}S. Jorgensen, S. Taboubi, and G. Zaccour, Cooperative advertising in a marketing channel. \emph{J. Optim. Theory Appl.}, {\bf 110}(1), 145-158, 2001.

\bibitem{JTZ03}S. Jorgensen, S. Taboubi, and G. Zaccour, Retail promotions with negative brand image effects: Is cooperation possible? \emph{Eur. J. Oper. Res.}, {\bf 150}, 395-405, 2003.

\bibitem{LY18}N. Li, Z. Y. Yu, Forward-backward stochastic differential equations and linear-quadratic generalized Stackelberg games. \emph{SIAM J. Control Optim.}, {\bf 56}(6), 4148-4180, 2018.

\bibitem{LS17}T. Li, S. P. Sethi, A review of dynamic Stackelberg game models. \emph{Discrete Contin. Dyn. Syst., Ser. B}, {\bf 22}(1), 125-159, 2017.

\bibitem{LT95}X. J. Li, S. J. Tang, General necessary conditions for partially observed optimal stochastic controls. \emph{J. Appl. Probab.}, {\bf 32}, 1118-1137, 1995.

\bibitem{LFCM19}Z. P. Li, M. Y. Fu, Q. Q. Cai, and W. Meng, Leader-follower stochastic differential games under partial observation. In: \emph{Proc. 38th Chinese Control Conference}, 1888-1892, Guangzhou, China, July 27-30, 2019.

\bibitem{LDFCM21}Z. P. Li, D. Marelli, M. Y. Fu, Q. Q. Cai, and W. Meng, Linear quadratic Gaussian Stackelberg game under asymmetric information patterns. \emph{Automatica}, {\bf 125}, 109406, 2021.

\bibitem{LZ01}A. E. B. Lim, X. Y. Zhou, Linear-quadratic control of backward stochastic differential equations. \emph{SIAM J. Control Optim.}, {\bf 40}(2), 450-474, 2001.

\bibitem{LJZ19}Y. N. Lin, X. S. Jiang, and W. H. Zhang, An open-loop Stackelberg strategy for the linear quadratic mean-field stochastic differential game. \emph{IEEE Trans. Autom. Control}, {\bf 64}(1), 97-110, 2019.

\bibitem{LS77}R. S. Liptser, A. N. Shiryayev, \emph{Statistics of Random Processes}, Springer-Verlag, New York, 1977.

\bibitem{MB18}J. Moon, T. Ba\c{s}ar, Linear quadratic mean field Stackelberg differential games. \emph{Automatica}, {\bf 97}, 200-213, 2018.

\bibitem{MX15}H. Mukaidani, H. Xu, Stackelberg strategies for stochastic systems with multiple followers. \emph{Automatica}, {\bf 53}, 53-79, 2015.

\bibitem{OSU13}B. \O ksendal, L. Sandal, and J. Ub\o e, Stochastic Stackelberg equilibria with applications to time dependent newsvendor models. \emph{J. Econ. Dyna. $\&$ Control}, {\bf 37}(7), 1284-1299, 2013.

\bibitem{PK96}P. Papatla, L. Krishnamurthi, Measuring the dynamic effects of promotions on brand choice. \emph{J. Mark. Res.}, {\bf 33}(1), 20-35, 1996.

\bibitem{PW99}S. G. Peng, Z. Wu, Fully coupled forward-backward stochastic differential equations and applications to optimal control. \emph{SIAM J. Control Optim.}, {\bf 37}(3), 825-843, 1999.

\bibitem{SWX16}J. T. Shi, G. C. Wang, and J. Xiong, Leader-follower stochastic differential game with asymmetric information and applications. \emph{Automatica}, {\bf 63}, 60-73, 2016.

\bibitem{SWX17}J. T. Shi, G. C. Wang, and J. Xiong, Linear-quadratic stochastic Stackelberg differential game with asymmetric information. \emph{Sci. China Infor. Sci.}, {\bf 60}, 092202:1-15, 2017.

\bibitem{SWX20}J. T. Shi, G. C. Wang, and J. Xiong, Stochastic linear quadratic Stackelberg differential game with overlapping information. \emph{ESAIM: Control, Optim. Calcu. Varia.}, {\bf 26}, Article Number 83, 2020.

\bibitem{SW10}J. T. Shi, Z. Wu, The maximum principle for partially observed optimal control of fully coupled forward-backward stochastic system. \emph{J. Optim. Theory Appl.}, {\bf 145}(3), 543-578, 2010.

\bibitem{S52}H. von Stackelberg, \emph{The Theory of the Market Economy}, Oxford University Press, London, 1952.

\bibitem{Tang98}S. J. Tang, The maximum principle for partially observed optimal control of stochastic differential equations. \emph{SIAM J. Control Optim.} {\bf 36}(5), 1596-1617, 1998.

\bibitem{WWZ20}G. C. Wang, Y. Wang, and S. S. Zhang, An asymmetric information mean-field type linear-quadratic stochastic Stackelberg differential game with one leader and two followers. \emph{Optim. Control Appl. Meth.}, {\bf 41}(4), 1034-1051, 2020.

\bibitem{WW08}G. C. Wang, Z. Wu, Kalman-Bucy filtering equations of forward and backward stochastic systems and applications to recursive optimal control problems. \emph{J. Math. Anal. Appl.}, {\bf 342}(2), 1280-1296, 2008.

\bibitem{WW09}G. C. Wang, Z. Wu, The maximum principle for stochastic recursive optimal control problems under partial information. \emph{IEEE Trans. Autom. Control}, {\bf 54}(6), 1230-1242, 2009.

\bibitem{WWX13}G. C. Wang, Z. Wu, and J. Xiong, Maximum principles for forward-backward stochastic control systems with correlated state and observation noises. \emph{SIAM J. Control Optim.}, {\bf 51}(1), 491-524, 2013.

\bibitem{WWX15}G. C. Wang, Z. Wu, and J. Xiong, A linear-quadratic optimal control problem of forward-backward stochastic differential equations with partial information. \emph{IEEE Trans. Autom. Control}, {\bf 60}(11), 2904-2916, 2015.

\bibitem{WWX18}G. C. Wang, Z. Wu, and J. Xiong, \emph{An Introduction to Optimal Control of FBSDE with Incomplete Information}, Springer Briefs in Mathematics, Switzerland, 2018.

\bibitem{Wu10}Z. Wu, A maximum principle for partially observed optimal control of forward-backward stochastic control systems. \emph{Sci. China Infor. Sci.}, {\bf 53}(11), 2205-2214, 2010.

\bibitem{WZ18}Z. Wu, Y. Zhuang, Linear-quadratic partially observed forward-backward stochastic differential games and its application in finance. \emph{Appl. Math. Comput.}, {\bf 321}, 577-592, 2018.

\bibitem{Xiong08}J. Xiong, \emph{An Introduction to Stochastic Filtering Theory}, Oxford University Press, London, 2008.

\bibitem{XZ07}J. Xiong, X. Y. Zhou, Mean-variance portfolio selection under partial information. \emph{SIAM J. Control Optim.}, {\bf 46}(1), 156-175, 2007.

\bibitem{XZZ19}J. Xiong, S. Q. Zhang, and Y. Zhuang, A partially observed non-zero sum differential game of forward-backward stochastic differential equations and its application in finance. \emph{Math. Control Relat. Fields}, {\bf 9}(2), 257-276, 2019.

\bibitem{XSZ18}J. J. Xu, J. T. Shi, and H. S. Zhang, A leader-follower stochastic linear quadratic differential game with time delay. \emph{Sci. China Infor. Sci.}, {\bf 61}, 112202:1-13, 2018.

\bibitem{XZ16}J. J. Xu, H. S. Zhang, Sufficient and necessary open-loop Stackelberg strategy for two-player game with time delay. \emph{IEEE Trans. Cyber.}, {\bf 46}(2), 438-449, 2016.

\bibitem{Xu95}W. S. Xu, Stochastic maximum principle for optimal control problem of forward and backward system. \emph{J. Aust. Math. Soc., Ser. B}, {\bf 37}, 172-185, 1995.

\bibitem{Yong02}J. M. Yong, A leader-follower stochastic linear quadratic differential games. \emph{SIAM J. Control Optim.}, {\bf 41}(4), 1015-1041, 2002.

\bibitem{Yong99}J. M. Yong, Linear forward-backward stochastic differential equations. \emph{Appl. Math. Optim.}, {\bf 39}(1), 93-119, 1999.

\bibitem{YZ99}J. M. Yong, X. Y. Zhou, \emph{Stochastic Controls: Hamiltonian Systems and HJB Equations}, Springer-Verlag, New York, 1999.

\bibitem{YDL00}B. Yoo, N. Donthu, N. Lee, An examination of selected marketing mix elements and brand equity. \emph{J. Acad. Mark. Sci.}, {\bf 28}(2), 195-211, 2000.

\bibitem{ZS19}Y. Y. Zheng, J. T. Shi, A Stackelberg game of backward stochastic differential equations with applications. \emph{Dyn. Games Appl.}, {\bf 10}, 968-992, 2020.

\bibitem{ZS20CCC}Y. Y. Zheng, J. T. Shi, A linear quadratic Stackelberg game of backward stochastic differential equations with partial information. In: \emph{Proc. 39th Chinese Control Conference}, 966-971, Shenyang, China, July 27-29, 2020.

\bibitem{ZS20}Y. Y. Zheng, J. T. Shi, Stackelberg stochastic differential game with asymmetric noisy observations. To appear in \emph{Inter. J. Control}. DOI: 10.1080/00207179.2021.1916078

\end{thebibliography}
\end{document}